\documentclass[a4paper,reqno,11pt]{amsart}

\usepackage{amsmath, amsfonts, amssymb, amsthm, amscd}
\usepackage{graphicx}
\usepackage{psfrag}
\usepackage{perpage}
\usepackage{url}
\usepackage{color}
\usepackage{mathrsfs}
\usepackage{tikz}
\usepackage{mathabx}
\usepackage{scalerel}
\usetikzlibrary{arrows,decorations.pathmorphing,backgrounds,positioning,fit,petri} 

\usepackage{dsfont} 

\usepackage[utf8]{inputenc}
\usepackage[T1]{fontenc}
\usepackage{microtype}

\usepackage[a4paper,scale={0.72,0.74},marginratio={1:1},footskip=7mm,headsep=10mm]{geometry}

\usepackage{hyperref}

\makeatletter
\def\@secnumfont{\bfseries\scshape}

\def\section{\@startsection{section}{1}%
  \z@{.7\linespacing\@plus\linespacing}{.5\linespacing}%
  {\normalfont\large\bfseries\scshape\centering}}

\def\subsection{\@startsection{subsection}{2}%
  \z@{.5\linespacing\@plus.7\linespacing}{-.5em}%
  {\normalfont\bfseries\scshape}}

\def\subsubsection{\@startsection{subsubsection}{3}%
  \z@{.5\linespacing\@plus.7\linespacing}{-.5em}%
  {\normalfont\scshape}}

\def\specialsection{\@startsection{section}{1}%
  \z@{\linespacing\@plus\linespacing}{.5\linespacing}%
  {\normalfont\centering\large\bfseries\scshape}}
\makeatother

\makeatletter

\renewenvironment{proof}[1][\proofname]{\par
\pushQED{\qed}%
\normalfont \topsep4\p@\@plus4\p@\relax
\trivlist
\item[\hskip\labelsep
\bfseries
#1\@addpunct{.}]\ignorespaces
}{%
\popQED\endtrivlist\@endpefalse
}
\makeatother

\makeatletter
\newcommand \Dotfill {\leavevmode \leaders \hb@xt@ 6pt{\hss .\hss }\hfill \kern \z@}
\makeatother

\makeatletter
\def\@tocline#1#2#3#4#5#6#7{\relax
  \ifnum #1>\c@tocdepth 
  \else
    \par \addpenalty\@secpenalty\addvspace{#2}%
    \begingroup \hyphenpenalty\@M
    \@ifempty{#4}{%
      \@tempdima\csname r@tocindent\number#1\endcsname\relax
    }{%
      \@tempdima#4\relax
    }%
    \parindent\z@ \leftskip#3\relax \advance\leftskip\@tempdima\relax
    \rightskip\@pnumwidth plus4em \parfillskip-\@pnumwidth
    #5\leavevmode\hskip-\@tempdima
      \ifcase #1
       \or\or \hskip 1.65em \or \hskip 3.3em \else \hskip 4.95em \fi%
      #6\nobreak\relax
    \Dotfill
    \hbox to\@pnumwidth{\@tocpagenum{#7}}\par
    \nobreak
    \endgroup
  \fi}
\makeatother

\makeatletter
\def\l@section{\@tocline{1}{0pt}{1pc}{}{\scshape}}
\renewcommand{\tocsection}[3]{%
\indentlabel{\@ifnotempty{#2}{\ignorespaces#1 #2.\hskip 0.7em}}#3}
\def\l@subsection{\@tocline{2}{0pt}{1pc}{5pc}{}}

\def\l@subsubsection{\@tocline{3}{0pt}{1pc}{7pc}{}}

\makeatother

\numberwithin{equation}{section}

\newtheoremstyle{mytheorem}{.7\linespacing\@plus.3\linespacing}{.7\linespacing\@plus.3\linespacing}%
     {\itshape}
     {}
     {\bfseries}
     {. }
     {0.3ex}
     {\thmname{{\bfseries #1}}\thmnumber{ {\bfseries #2}}\thmnote{ (#3)}}  

\theoremstyle{mytheorem}

\newtheorem{theorem}{Theorem}[section]
\newtheorem{lemma}[theorem]{Lemma}
\newtheorem{proposition}[theorem]{Proposition}
\newtheorem{corollary}[theorem]{Corollary}
\newtheorem{remark}[theorem]{Remark}

\newcommand{\bbE}{{\ensuremath{\mathbb E}} }

\newcommand{\bbP}{{\ensuremath{\mathbb P}} }

\newcommand{\bbT}{{\ensuremath{\mathbb T}} }

\newcommand\bsf{\boldsymbol{f}}

\newcommand\bsG{\boldsymbol{G}}

\newcommand\bsU{\boldsymbol{U}}

\newcommand\bsY{\boldsymbol{Y}}
\newcommand\bsZ{\boldsymbol{Z}}

\newcommand\bslambda{\boldsymbol{\lambda}}
\newcommand\bsnu{\boldsymbol{\nu}}

\newcommand{\cN}{{\ensuremath{\mathcal N}} }

\newcommand{\gl}{\lambda}

\DeclareMathSymbol{\leqslant}{\mathalpha}{AMSa}{"36} 
\DeclareMathSymbol{\geqslant}{\mathalpha}{AMSa}{"3E} 
\DeclareMathSymbol{\eset}{\mathalpha}{AMSb}{"3F}     

\newcommand{\sumtwo}[2]{\sum_{\substack{#1 \\ #2}}} 

\newcommand{\be}{\begin{equation}}
\newcommand{\ee}{\end{equation}}

\newcommand{\R}{\mathbb{R}}

\newcommand{\Z}{\mathbb{Z}}
\newcommand{\N}{\mathbb{N}}

\def\bs{\boldsymbol}

\newcommand{\PEfont}{\mathrm}

\newcommand{\p}{\ensuremath{\PEfont P}}

\DeclareMathOperator{\e}{\ensuremath{\PEfont E}}

\newcommand{\E}{\e}
\renewcommand{\P}{\p}
\newcommand{\Cov}{\mathrm{Cov}}

\DeclareMathOperator{\bbvar}{\ensuremath{\mathbb{V}ar}}

\newcommand{\ind}{\mathds{1}}

\newcommand{\eps}{\varepsilon}
\renewcommand{\epsilon}{\varepsilon}
\renewcommand{\theta}{\vartheta}
\renewcommand{\rho}{\varrho}

\newenvironment{myenumerate}{%
\renewcommand{\theenumi}{\arabic{enumi}}%
\renewcommand{\labelenumi}{{\rm(\theenumi)}}%
\begin{list}{\labelenumi}
	{%
	\setlength{\itemsep}{0.4em}%
	\setlength{\topsep}{0.5em}%
	\setlength\leftmargin{2.45em}%
	\setlength\labelwidth{2.05em}%
	\setlength{\labelsep}{0.4em}%
	\usecounter{enumi}%
	}%
	}%
{\end{list}
}

{\end{list}
}

{\end{list}
}

{\end{myenumerate}}

\newenvironment{myitemize}{%
\begin{list}{$\bullet$}%
 	{%
	\setlength{\itemsep}{0.4em}%
	\setlength{\topsep}{0.5em}%
	\setlength\leftmargin{2.65em}%
	\setlength\labelwidth{2.65em}%
	\setlength{\labelsep}{0.4em}%
	}%
	}%
{\end{list}}

\renewenvironment{itemize}{
\begin{myitemize}}%
{\end{myitemize}}

\MakePerPage[2]{footnote} 

\date{\today}

\newcommand\dd{\mathrm{d}}

\newcommand\sfa{\mathsf a}
\newcommand\sfc{\mathsf c}
\newcommand\even{\mathrm{even}}

\begin{document}

\title[The Dickman subordinator, renewal theorems, and disordered systems]{
The Dickman subordinator, renewal theorems,\\ and disordered systems}

\begin{abstract}
We consider the so-called Dickman subordinator,
whose L\'evy measure has density
$\frac{1}{x}$ restricted to the interval $(0,1)$.
The marginal density of this process,
known as the Dickman function, appears
in many areas of mathematics, from number theory to combinatorics.
In this paper, we study renewal processes in the domain of attraction
of the Dickman subordinator, for which we prove local renewal theorems. 
We then present applications to marginally
relevant disordered systems, such as pinning and directed
polymer models, and prove sharp second moment estimates
on their partition functions.
\end{abstract}

\author[F. Caravenna]{Francesco Caravenna}
\address{Dipartimento di Matematica e Applicazioni\\
 Universit\`a degli Studi di Milano-Bicocca\\
 via Cozzi 55, 20125 Milano, Italy}
\email{francesco.caravenna@unimib.it}

\author[R. Sun]{Rongfeng Sun}
\address{Department of Mathematics\\
National University of Singapore\\
10 Lower Kent Ridge Road, 119076 Singapore
}
\email{matsr@nus.edu.sg}

\author[N. Zygouras]{Nikos Zygouras}
\address{Department of Statistics\\
University of Warwick\\
Coventry CV4 7AL, UK}
\email{N.Zygouras@warwick.ac.uk}

\date{\today}


\maketitle

\section{Introduction and main results}

\subsection{Motivation}
We consider the subordinator (increasing L\'evy process)
denoted by $Y=(Y_s)_{s\geq 0}$, which is pure jump with
 L\'evy measure
\begin{equation} \label{eq:bsLevymeasure}
	\nu(\dd t) := \frac{1}{t} \, \ind_{(0,1)}(t)  \, \dd t  \,.
\end{equation}
Equivalently, its Laplace transform is given by
\begin{equation}\label{eq:LapY}
	\E[e^{\lambda  Y_s}] =
	\exp \bigg\{s \, \int_{0}^1
	(e^{\lambda t } - 1) \, \frac{\dd t}{t}
	 \bigg\} \,.
\end{equation}
We call $Y$ \emph{the Dickman subordinator} (see Remark~\ref{rem:Dickman} below).
It is suggestive to view it
as a ``truncated $0$-stable subordinator'',
by analogy with the well known $\alpha$-stable subordinator
whose L\'evy 
measure is $\frac{1}{t^{1+\alpha}} \ind_{(0,\infty)}(t) \, \dd t$,
for $\alpha \in (0,1)$.
In our case $\alpha = 0$ and the restriction $\ind_{(0,1)}(t)$ in \eqref{eq:bsLevymeasure}
ensures that $\nu$ is a legitimate L\'evy measure,
i.e.\ $\int_\R (t^2 \wedge 1) \, \nu(\dd t) < \infty$.

\smallskip

Interestingly, the Dickman subordinator admits an
explicit marginal density
\begin{equation} \label{eq:fst}
	f_s(t)\,  := \frac{\P(Y_s \in \dd t)}{\dd t} \,, \qquad
	\text{for } s, t \in (0,\infty) \,,
\end{equation}
which we recall in the following result. 

\begin{theorem}[Density of the Dickman subordinator] \label{th:scalingY}
For all $s \in (0,\infty)$ one has
\begin{equation}\label{eq:scalingf}
	 f_s(t) =  \left\{
	 \begin{aligned}
	 & \frac{s\, t^{s-1} \, e^{-\gamma \, s}}{\Gamma(s+1)} \quad \quad & \mbox{for } t\in (0, 1], \\
	 & \frac{s\, t^{s-1}e^{-\gamma s}}{\Gamma(s+1)} - st^{s-1} \int_0^{t-1} \frac{f_s(a)}{(1+a)^s} \dd a \quad \quad & \mbox{for } t\in (1, \infty),
	 \end{aligned}
	 \right.
\end{equation}
where $\Gamma(\cdot)$ denotes Euler's gamma function
and $\gamma = -\int_0^\infty \log u \, e^{-u} \, \dd u \simeq 0.577$
is the Euler-Mascheroni constant.
\end{theorem}

Theorem~\ref{th:scalingY} follows
from general results about \emph{self-decomposable L\'evy processes}
\cite{cf:Sato}.\footnote{We thank Thomas Simon for pointing out this connection.}
We give the details in Appendix~\ref{sec:Dickman},
where we also present an alternative, self-contained 
derivation of the density $f_s(t)$, based
on direct probabilistic arguments.
We refer to \cite{BKKK14}
for further examples of subordinators with explicit densities.

\begin{remark}[Dickman function and Dickman distribution]\label{rem:Dickman}
The function
\begin{equation*}
	\rho(t) := e^\gamma\, f_1(t)
\end{equation*}
is known
as the \emph{Dickman function}
and plays an important role in number theory and combinatorics \cite{cf:Ten,ABT}.
By \eqref{eq:scalingf} we see that $\rho$ satisfies
\begin{equation}\label{eq:Dickman}
	\rho(t) \equiv 1 \quad \text{for } \ t \in (0,1] \,, \qquad
	\qquad t \, \rho'(t) + \rho(t-1) = 0 \quad \text{for } \ t \in (1,\infty) \,,
\end{equation}
which is the classical
definition of the Dickman function. Examples where $\rho$ emerges are:
\begin{itemize}
\item If $X_n$ denotes the largest prime factor of a uniformly chosen integer in $\{1,\ldots, n\}$, 
then $\lim_{n\to\infty} \P(X_n \le n^t) = \rho(1/t)$ \cite{Dickman}. 

\item If $Y_n$ denotes the size of the longest cycle in a uniformly chosen
permutation of $n$ elements, then
$\lim_{n\to\infty} \P(Y_n \le n t) = \rho(1/t)$ \cite{Kingman}.
\end{itemize}
Thus both $(\log X_n / \log n)$ and $(Y_n / n)$ converge in law as $n\to\infty$
to a random variable $L_1$ with $\P(L_1 \le t) = \rho(1/t)$.
The density of $L_1$ equals
$t^{-1}  \rho(t^{-1} - 1)$, by \eqref{eq:Dickman}.

\smallskip

The marginal law $Y_1$ of  our subordinator,
called the \emph{Dickman distribution} in the literature,
also arises in many contexts, 
from logarithmic combinatorial structures \cite[Theorem~4.6]{ABT}
to theoretical computer science \cite{HT}.
We stress that $Y_1$ and $L_1$ are different
-- their laws are supported in $(0,\infty)$ and $(0,1)$, respectively --
though both are  related to the Dickman function: their
densities are $e^{-\gamma} \rho(t)$ and $t^{-1} \rho(t^{-1}-1)$,
respectively
\end{remark}

In this paper, we present a novel application of the Dickman subordinator 
in the context of \emph{disordered systems},
such as pinning and directed polymer models.
We will discuss the details in Section~\ref{sec:main3},
but let us give here the crux of the problem in an elementary way, 
which can naturally arise in various other settings. 

Given $q,r \in (0,\infty)$, let us consider the 
weighted series of convolutions
\begin{equation} \label{eq:uN}
	v_N :=\sum_{k = 1}^\infty \, q^k \, \sum_{0 < n_1 < n_2 < \ldots
	< n_k \le N}
	\, \frac{1}{n_1^r (n_2 - n_1)^r \cdots (n_k - n_{k-1})^r} \,.
\end{equation}
We are interested in the following question: for a fixed
exponent $r\in (0,\infty)$,
\emph{can one choose $q = q_N$ so that $v_N$ converges
to a non-zero and finite limit limit as $N\to\infty$}, i.e.\ $v_N \to v \in (0,\infty)$?
The answer naturally depends on the exponent $r$.

\smallskip

If $r < 1$, we can, straightforwardly, use a Riemann sum approximation 
and by choosing $q = \lambda N^{-1+r} $, for fixed $ \lambda \in (0,\infty)$, we have that $v_N$ will converge to
\begin{equation}
\begin{split}
	v & \,:=\,  \sum_{k=1}^\infty \, 
	\lambda^k \ \Bigg\{ \ \,
	\idotsint\limits_{0 < t_1 < \ldots < t_{k} < 1} \
	\frac{\dd t_1 \cdots \dd t_k}
	{t_1^r (t_2 - t_1)^r \cdots (t_k - t_{k-1})^r} \Bigg\}  
	\,=\, \sum_{k=1}^\infty \lambda^k
	\, \frac{\Gamma(r)^{k+1}}{\Gamma((k+1)r)}
\end{split}
\end{equation}
where the last equality is deduced from the normalization of the Dirichlet distribution.

\smallskip

 If $r \geq 1$, then, as it is readily seen, the Riemann sum 
approach fails, as it leads to iterated integrals which are
 infinite. The idea now is to express the series
\eqref{eq:uN} as a renewal function. The case $r>1$ is easy:
 we can take a  small, but \emph{fixed} $q > 0$, more precisely
\begin{equation*}
	q \in \bigg(0, \frac{1}{R}\bigg) \,, \qquad \text{where} \qquad R :=
	\sum_{n\in\N} \frac{1}{n^r} \ \in (0, \infty) \,,
\end{equation*}
and consider the renewal process $\tau = (\tau_k)_{k\ge 0}$ with inter-arrival law $\P(\tau_1 = n)= \frac{1}{R} \, \frac{1}{n^r}$ for $n\in\N$. 
We can then write
\begin{equation*}
	v_N = \sum_{k=1}^\infty
	\big(q R \big)^k \, \P(\tau_k \le N) 
	\ \xrightarrow[\, N\to\infty \, ]{} \
	v := \frac{q R}{1- q R} \ \in (0,\infty) \,.
\end{equation*}

\smallskip

The case $r = 1$ is more interesting\footnote{It can be called \emph{marginal}
or \emph{critical}, due to its relations to disordered
systems, see \cite{CSZ17b} for the relevant terminology and 
statistical mechanics background.}.
 This case is subtle because the normalization $R=\sum_{n\in\N} \frac{1}{n} =\infty$.
The way around this problem  is to first normalize $\frac{1}{n}$ to a probability
on $\{1,2,\ldots, N\}$. More precisely, we take
\begin{equation*}
	R_N := \sum_{n=1}^N \frac{1}{n} = \log N \, \big(1+o(1)\big) \,,
\end{equation*}
and consider the renewal process
$\tau^{(N)} = (\tau^{(N)}_k)_{k\ge 0}$ with inter-arrival law
\begin{equation} \label{eq:tauN0}
	\P\big( \tau^{(N)}_1=n \big)
	= \frac{1}{R_N}\, \frac{1}{n} \qquad 
	\text{for } n \in \{1,2,\ldots, N\} \,.
\end{equation}
Note that this renewal process is a discrete analogue 
of the Dickman subordinator.
Choosing $q = \lambda / R_N$, with $\lambda<1$, we can see, via dominated convergence, that
\begin{equation} \label{eq:vNas}
	v_N = \sum_{k=1}^\infty \lambda^k \, \P(\tau^{(N)}_k \le N)
	\ \xrightarrow[\, N\to\infty \, ]{} \
	v := \frac{\lambda}{1- \lambda} \ \in (0,\infty) \,
\end{equation}
because $\P(\tau^{(N)}_k \le N) \to 1$ as $N\to\infty$, for any fixed $k\in\N$. But when
$\lambda = 1$, then $v_N \to \infty$ and then finer questions emerge, e.g., 
at which rate does $v_N \to \infty$? 
Or what happens if
instead of $\P(\tau^{(N)}_k \le N)$ we consider $\P(\tau^{(N)}_k = N)$ in \eqref{eq:vNas}, i.e.\ if we fix
$n_k = N$ in \eqref{eq:uN}?

\medskip

To answer these questions,
it is necessary to explore the domain of attraction of the 
Dickman subordinator --- to which $\tau^{(N)}$ belongs,
as we show below --- and to 
prove renewal theorems. Indeed, the left hand side of \eqref{eq:vNas}
for $\lambda = 1$ defines the \emph{renewal measure of $\tau^{(N)}$}.
Establishing results of this type is the core of our paper.

\smallskip

\subsection{Main results}

We study a class of renewal processes $\tau^{(N)}$ which generalize \eqref{eq:tauN0}.
Let us fix a sequence $(r(n))_{n\in\N}$ such that
\begin{align} \label{eq:rn}
	r(n) &:= \frac{\sfa}{n} (1+o(1)) \qquad \text{as } n \to \infty \,, 
\end{align}
for some constant $\sfa \in (0,\infty)$, so that
\begin{align} 
	\label{eq:RN}
	R_N &:= \sum_{n=1}^N r(n) = \sfa \, \log N (1+o(1)) \qquad \text{as } N \to \infty \,.
\end{align}
For each $N\in\N$, we consider i.i.d.\ random variables $(T^{(N)}_i)_{i\in\N}$ with distribution
\begin{equation} \label{eq:XN}
	\P(T^{(N)}_i = n) := \frac{r(n)}{R_N} \, \ind_{\{1,\ldots, N\}}(n) \,.
\end{equation}
(The precise value of the constant $\sfa$ is immaterial, since it gets simplified in \eqref{eq:XN}.)

Let $\tau^{(N)} = (\tau^{(N)}_k)_{k \in \N_0}$ denote the associated random walk (renewal process):
\begin{equation} \label{eq:deftau}
	\tau^{(N)}_0 := 0 \,, \qquad \tau^{(N)}_k := \sum_{i=1}^k T^{(N)}_i \,.
\end{equation}
We first show that $\tau^{(N)}$ is in the domain of attraction
of the Dickman subordinator $Y$.

\begin{proposition}[Convergence of rescaled renewal process]\label{prop:YN0}
The rescaled process
\begin{equation*}
	\Bigg( \frac{\tau^{(N)}_{\lfloor s \log N \rfloor}}{N} \Bigg)_{s\geq 0}
\end{equation*}
converges in distribution to the Dickman subordinator $(Y_s)_{s\geq 0}$, as $N\to\infty$.
\end{proposition}

We then define an \emph{exponentially weighted renewal density} $U_{N,\lambda}(n)$
for $\tau^{(N)}$, which is a local version of the quantity which appears
in \eqref{eq:vNas}:
\begin{equation} \label{UN}
	U_{N,\lambda}(n) := \sum_{k\ge 0} \lambda^k \,
	\P(\tau^{(N)}_k=n) \qquad  \text{for} \quad N, n \in \N, 
	\ \lambda\in (0,\infty) \,.
\end{equation}
We similarly define the corresponding quantity for the
Dickman subordinator:
\begin{equation} \label{eq:G0}
	G_{\theta}(t) := \int_0^\infty 
	e^{\theta s} \, f_s(t) \, \dd s
	\qquad \text{for} \quad t \in (0,\infty) \,, \
	\theta \in \R \,,
\end{equation}
which becomes more explicit for $t \in (0,1]$, by \eqref{eq:scalingf}:
\begin{equation} \label{eq:G0+}
	G_{\theta}(t) = \int_0^\infty  \frac{ e^{(\theta - \gamma) s} \,  
	s \, t^{s-1}}{\Gamma(s+1)} \, \dd s 
	\qquad \text{for} \quad t \in (0,1] \,, \
	\theta \in \R \,.
\end{equation}

Our main result identifies
the asymptotic behavior of the renewal density $U_{N,\lambda}(n)$
for large $N$ and $n = O(N)$. 
This is shown to be of the order
 $\E[T_1^{(N)}]^{-1} \sim (\frac{N}{\log N})^{-1}$,
in analogy with the classical renewal theorem,
with a sharp prefactor given by $G_\theta(\frac{n}{N})$.

\begin{theorem}[Sharp renewal theorem]\label{th:WR}
Fix any $\theta \in \R$ and let $(\lambda_N)_{N\in\N}$ satisfy
\begin{equation}\label{eq:lambdaN}
	\lambda_N = 1 + \frac{\theta}{\log N} \big(1+o(1)\big) \qquad
	\text{as } N \to \infty \,.
\end{equation}
For any fixed
$0 < \delta < T < \infty$, the following relation holds as $N \to \infty$:
\begin{equation} \label{eq:UNas}
	U_{N,\lambda_N}(n) = \frac{\log N}{N} \, G_{\theta}(\tfrac{n}{N}) \, (1+o(1)) \,,
	\quad \ \ 
	\text{uniformly for } \ \delta N \le n \le T N \,.
\end{equation}
Moreover, for any fixed $T < \infty$,
the following uniform bound holds, for a suitable $C \in (0,\infty)$:
\begin{equation} \label{eq:unibo}
	U_{N,\lambda_N}(n) \le C \, \frac{\log N}{N} \, G_{\theta}(\tfrac{n}{N}) \, ,
	\qquad \forall 0 < n \le T N \,.
\end{equation}
\end{theorem}

As anticipated, we will present an application to disordered systems
in Section~\ref{sec:main3}: for
pinning and directed polymer models,
we derive the sharp asymptotic
behavior of the second moment of the partition function
in the weak disorder regime (see Theorems~\ref{th:pinning}
and~\ref{sec:dpre}).

\smallskip

We stress that Theorem~\ref{th:WR}
extends the literature on \emph{renewal theorems
in the case of infinite mean}. 
Typically, the cases studied in the literature correspond to 
renewal processes of the form $\tau_n= T_1+\cdots +T_n $, 
where the i.i.d.\ increments $(T_i)_{i\geq 1}$ have law
\begin{align}\label{regular}
\P(T_1=n)=\phi(n)\, n^{-(1+\alpha)},
\end{align}
with $\phi(\cdot)$ a slowly varying function. 
In case $\alpha \in (0,1]$, limit theorems for the
renewal density $U(n)=\sum_{k\geq 1} \P(\tau_k=n)$
have been the subject of many works,
e.g.\ \cite{GL}, \cite{E70}, \cite{D97}, 
just to mention a few of the most notable ones.
  The sharpest results in this direction have been recently established in  \cite{CD16+} when
$\alpha \in (0,1)$, and in \cite{B17+} when $\alpha = 1$.

\smallskip

In the case of  \eqref{regular} with $\alpha =0$, results 
of the sorts of Theorem \ref{th:WR} have been obtained in \cite{NW08, N12, AB16}.
One technical difference between these references and our result 
is that we deal with a non-summable sequence $1/n$,
hence it is necessary to consider a family of renewal processes $\tau^{(N)}$ whose
law varies with $N\in\N$
(\emph{triangular array}) via a suitable cutoff. 
This brings our renewal process
out of the scope of the cited references.

\smallskip

We point out that it is possible to generalize our assumption \eqref{eq:rn} to more general renewals with
inter-arrival decay exponent $\alpha=0$. More precisely, replace the constant $\sfa$ 
therein by a slowly varying function $\phi(n)$ \emph{such that $\sum_{n\in\N} \phi(n)/n = \infty$}, 
in which case $R_N=\sum_{n=1}^N \phi(n)/n$ is 
also a slowly varying function with $R_N/\phi(N)\to\infty$ (see \cite[Prop.~1.5.9a]{BGT89}).
We expect that our results extend to this case
with the same techniques,
but we prefer to stick to the simpler assumption \eqref{eq:rn},
which considerably simplifies notation.

\medskip

Let us give an overview of the proof of
Theorem~\ref{th:WR} (see Section~\ref{sec:WR} for more details).
In order to prove the upper bound \eqref{eq:unibo}, a key tool is the following
sharp estimate on the local probability $\P(\tau_k^{(N)}=n)$. 
It suggests that the main contribution to $\{\tau_k^{(N)}=n\}$
comes from the strategy that \emph{a single increment $T_i^{(N)}$ takes values close to $n$}.

\begin{proposition}[Sharp local estimate]\label{prop:sharp}
Let us set $\log^+(x) := (\log x)^+$.
There are constants $C \in (0,\infty)$ and $c \in (0,1)$ such that
for all $N, k \in \N$ and $n \le N$ we have
\begin{equation}\label{eq:pezzofin0intro}
	\P\big(\tau_k^{(N)}=n\big) 
	\ \le \ C \, k \, \P\big(T_1^{(N)} = n\big) \,
	\, \P\big(T_1^{(N)} \le n\big)^{k-1}
	\, e^{-\frac{c \, k}{\log n + 1} \, \log^+\frac{c\, k}{\log n + 1}}  \,.
\end{equation}
\end{proposition}

\noindent
We point out that \eqref{eq:pezzofin0intro} sharpens 
\cite[eq. (1.11) in Theorem~1.1]{AB16}, thanks to the
last term which decays super-exponentially in $k$.
This will be essential for us, in order to counterbalance the exponential weight
$\lambda^k$ in the renewal density $U_{N,\lambda}(n)$, see \eqref{UN}.

\smallskip

In order to prove the local limit theorem \eqref{eq:UNas}, we use a strategy of
independent interest: we are going to deduce it
from the weak convergence in Proposition~\ref{prop:YN0}
by exploiting \emph{recursive formulas}
for the renewal densities 
$U_{N,\lambda}$ and $G_\theta$, based on a decomposition according to the jump that
straddles a fixed site; see \eqref{eq:Uren} and \eqref{eq:Gren}.
These formulas provide integral representations of the renewal densities
$U_{N,\lambda}$ and $G_\theta$ which reduce a \emph{local}
limit behavior to an \emph{averaged} one, thus allowing to strengthen weak 
convergence results to local ones.

\medskip

Finally, we establish fine asymptotic properties of
the continuum renewal density $G_\theta$. 

\begin{proposition}\label{prop:asGthetac}
For any fixed $\theta \in \R$,
the function $G_{\theta}(t)$ is continuous
(actually $C^\infty$) and strictly positive for $t \in (0,1]$.
As $t \downarrow 0$ we have $G_\theta(t) \to \infty$, more precisely
\begin{equation} \label{eq:Gthetacas}
	G_{\theta}(t) =
	\frac{1}{t (\log\frac{1}{t})^2}
	\Bigg\{ 1
	+ \frac{2\theta}{\log \frac{1}{t}}
	\ + \ O\bigg(\frac{1}{(\log \frac{1}{t})^{2}} \bigg) \Bigg\} \,.
\end{equation}
\end{proposition}

\smallskip

\begin{remark}\label{rem:density}
Our results also apply to renewal processes with a density. 
Fix a bounded and continuous function $r: [0,\infty) \to (0,\infty)$ with
$r(t) = \frac{\sfa}{t}(1+o(1))$ as $t \to \infty$, so that
$R_N := \int_0^N r(t) \, \dd t = \sfa \, \log N(1+o(1))$.
If we consider the renewal process $\tau^{(N)}_k$ in \eqref{eq:deftau} with
\begin{equation*}
	\P(T^{(N)}_i \in \dd t) = \frac{r(t)}{R_N} \, \ind_{[0,N]}(t) \, \dd t \,,
\end{equation*}
then Proposition~\ref{prop:YN0}, Theorem~\ref{th:WR} and Proposition~\ref{prop:sharp}
still hold, provided $\P\big(\tau_k^{(N)}=n\big)$ denotes
the density of $\tau_k^{(N)}$. The proofs can be easily adapted, replacing
sums by integrals. 
\end{remark}

\smallskip

\subsection{Organization of the paper.}
In Section~\ref{sec:main2} we present multi-dimensional extensions
of our main results,
where we extend the subordinator and the renewal processes with a spatial 
component.
This is guided by applications to the directed polymer model.

In Section~\ref{sec:main3} we discuss the applications of our results 
to disordered systems and
more specifically to pinning and directed polymer models.
A result of independent interest is Proposition~\ref{prop:EM}, where
we prove sharp 
asymptotic results on the expected number of encounters at the origin 
of two independent simple random walks on $\Z$; this also gives 
the expected number of encounters (anywhere) of two independent
simple random walks on $\Z^2$.

The remaining sections \ref{sec:conve}-\ref{sec:Green} are devoted to the proofs. 
Appendix~\ref{sec:appds} contains results for disordered systems,
while Appendix~\ref{sec:Dickman} is devoted to the Dickman subordinator.

\medskip

\section{Multidimensional extensions}
\label{sec:main2}

We extend our subordinator $Y$ by adding a spatial component, 
that for simplicity we assume to 
be Gaussian. More precisely, we fix a dimension $d \in \N$ and we
let $W = (W_t)_{t \in [0,\infty)}$ denote a
standard Brownian motion on $\R^d$.
Its density is given by
\begin{equation} \label{eq:gt}
	g_t(x) := \frac{1}{(2\pi t )^{d/2}}
	\exp(- \tfrac{|x|^2}{2 t} ) \,,
\end{equation}
where $|x|$ is the Euclidean norm.
Note that $\sqrt{\sfc} \, W_t$
has density $g_{\sfc t}(x)$, for every $\sfc \in (0,\infty)$.

\smallskip

Recall the definition \eqref{eq:bsLevymeasure} of the measure $\nu$.
We denote by $\bsY^{\sfc}:=(\bsY^{\sfc}_s )_{s\geq 0} 
= (Y_s, V^{\sfc}_s)_{s\geq 0}$ 
the L\'evy process on $[0,\infty) \times \R^d$
with zero drift, no Brownian component, and L\'evy measure
\begin{equation} \label{eq:bsLevymeasure_space}
	\bsnu(\dd t, \dd x) := \nu(\dd t) \, g_{\sfc t}(x) \, \dd x =
	\frac{\ind_{(0,1)}(t)}{t} \, g_{\sfc t}(x) \, \dd t \,\dd x \,.
\end{equation}
Equivalently, for all $\bslambda\in\R^{1+d}$ and $s \in [0,\infty)$,
\begin{equation}\label{eq:bsLapY}
	\E[e^{\langle \bslambda , \bsY^\sfc_s \rangle}] =
	\exp \bigg\{s \, \int_{(0,1) \times \R^d}
	(e^{\langle \bslambda, (t,x) \rangle} - 1) \, \frac{g_{\sfc t}(x)}{t}
	\, \dd t \, \dd x \bigg\} \,.
\end{equation}

We can identify the probability density of 
$\bsY^\sfc_s$ for $s\in[0,\infty)$ as follows.

\begin{proposition}[Density of L\'evy process] \label{th:bsYY}
We have the following representation:
\begin{equation*}
	(\bsY^\sfc_s)_{s \in [0,\infty)} \quad \overset{d}{=} \quad
	\big( (Y_s, \, \sqrt{\sfc} \, W_{Y_s}) \big)_{s \in [0,\infty)} \,,
\end{equation*}
with $W$ independent of $Y$.
Consequently, $\bsY_s^\sfc$ has probability density
(recall \eqref{eq:fst} and \eqref{eq:gt})
\begin{equation}\label{eq:densities}
	\bsf_s(t,x) = f_s(t) \, g_{ \sfc t}(x) \,.
\end{equation}
\end{proposition}

We now define a family
of random walks in the domain of attraction of $\bs{Y}^{ \sfc }$.
Recall that $r(n)$ was defined in \eqref{eq:rn}. We consider
a family of probability kernels $p(n, \cdot )$ on $\Z^d$, indexed by $n\in\N$,
which converge in law to $\sqrt{\sfc} \, W_1$
when rescaled diffusively. More precisely,
we assume  the following conditions:
\begin{equation} \label{eq:proppnx}
\begin{aligned}
{\rm (i)} &\quad \sum_{x\in \Z^d} x_i \,p(n,x) = 0
\quad \text{for } i=1,\ldots, d\\
{\rm (ii)}& \quad 
\sum_{x\in \Z^d} |x|^2 \,p(n,x) = 
O(n)
\quad \text{as } n \to \infty
 \\
{\rm (iii)}& \quad
\sup_{x\in\Z^d}  \big| 
n^{d/2} \, p(n,x) - g_{\sfc}
\big( \tfrac{x}{\sqrt{n}} \big) \big| = o(1)
\quad \text{as } n \to \infty \,.
\end{aligned} 
\end{equation}
Note that $\sfc \in (0,\infty)$
is the asymptotic variance of each component.
Also note that, by (iii),
\begin{equation}\label{eq:lltnorm}
	\sup_{x\in\Z} \, p(n,x) = O\bigg( \frac{1}{n^{d/2}} \bigg)
	\quad \text{as } n \to \infty \,.
\end{equation}

Then we define, for every $N\in\N$, the i.i.d.\ random variables $(T^{(N)}_i, X^{(N)}_i)\in\N \times \Z^d$ by
\begin{equation} \label{eq:bsXN}
	\P\big((T^{(N)}_i, X^{(N)}_i) = (n,x)\big) 
	:= \frac{r(n) \, p(n, x)}{R_N} \, \ind_{\{1,\ldots, N\}}(n) \,,
\end{equation}
with $r(n)$, $R_N$ as in \eqref{eq:rn}, \eqref{eq:RN}.
Let $(\tau^{(N)}, S^{(N)})$ be the associated random walk, i.e.\
\begin{equation} \label{eq:tauS}
	\tau^{(N)}_k := T^{(N)}_1 + \ldots + T^{(N)}_k \,, \qquad
	S^{(N)}_k := X^{(N)}_1 + \ldots + X^{(N)}_k \,.
\end{equation}
We have the following analogue of Proposition~\ref{prop:YN0}.

\begin{proposition}[Convergence of rescaled L\'evy process]\label{prop:YN01}
Assume that the conditions in \eqref{eq:proppnx} hold.
The rescaled process
\begin{equation*}
	\Bigg( \frac{\tau^{(N)}_{\lfloor s \log N \rfloor}}{N},
	\frac{S^{(N)}_{\lfloor s \log N \rfloor}}{\sqrt{N}} \Bigg)_{s\geq 0}
\end{equation*}
converges in distribution to $(\bsY^{\sfc}_s
:= (Y_s,  V^\sfc_s))_{s\geq 0}$, as $N\to\infty$.
\end{proposition}

We finally introduce the exponentially weighted renewal density
\begin{equation} \label{UN2}
	\bsU_{N,\lambda}(n,x) :=
	\sum_{k\ge 0} \lambda^k \,
	\P(\tau^{(N)}_k=n, \, S^{(N)}_k=x) \,,
\end{equation}
as well as its continuum version:
\begin{equation} \label{eq:G}
	\bsG_{\theta}(t,x) := \int_0^\infty 
	e^{\theta s} \, \bsf_s(t,x) \, \dd s
	= G_{\theta}(t) \, 	g_{ \sfc t}(x) 
	\qquad \text{for } t \in (0,\infty) \,, \ x \in \R^d \,,
\end{equation}
where the second equality follows by \eqref{eq:G0} and Proposition~\ref{th:bsYY}.
Recall \eqref{UN} and observe that
\begin{equation}\label{eq:Usum}
	\sum_{x\in\Z^d} \bsU_{N,\lambda}(n,x) = U_{N,\lambda}(n)
\end{equation}

The following result is an extension of Theorem~\ref{th:WR}.

\begin{theorem}[Space-time renewal theorem]\label{th:WR2}
Fix any $\theta \in \R$ and let $(\lambda_N)_{N\in\N}$ satisfy
\begin{equation*}
	\lambda_N = 1 + \frac{\theta}{\log N} \big(1+o(1)\big) \qquad
	\text{as } N \to \infty \,.
\end{equation*}
For any fixed $0 < \delta < T < \infty$, 
the following relation holds as $N \to \infty$:
\begin{equation} \label{eq:bsUNas}
\begin{split}
	& \bsU_{N,\lambda_N}(n,x) 
	\,=\, \frac{\log N}{N^{1+d/2}} \, 
	G_{\theta}\big(\tfrac{n}{N}\big) \, g_{\sfc \frac{n}{N}}
	\big(\tfrac{x}{\sqrt{N}}\big) \big(1+o(1)\big) \,, \\
	& \rule{0pt}{1.3em} \quad \text{uniformly for } \
	\delta N \le n \le T N, \
	|x| \le \tfrac{1}{\delta}\sqrt{N} \,.
\end{split}
\end{equation}
Moreover, for any fixed $T<\infty$, the following uniform bound holds, for a suitable $C \in (0,\infty)$:
\begin{gather} \label{eq:unibo2}
	\bsU_{N,\lambda_N}(n,x) \le C \, \frac{\log N}{N} \, \frac{1}{n^{d/2}} \, 
	G_{\theta}(\tfrac{n}{N}) \, ,
	\qquad \forall 0 < n \le T N \,, \ \forall x \in \Z^d \,.
\end{gather}
\end{theorem}

The bound \eqref{eq:unibo2} is to be expected, in view of
\eqref{eq:bsUNas}, because $\sup_{z\in\R^d}
g_{t}(z) \le \frac{C}{t^{d/2}}$. Finally,
we show that the probability
$\frac{\bsU_{N,\lambda}(n,\cdot)}{U_{N,\lambda}(n)}$
is concentrated on the diffusive scale $O(\sqrt{n})$.

\begin{theorem}\label{th:Udiffusive}
There exists a constant $C \in (0,\infty)$ such that
for all $N\in\N$ and $\lambda \in (0,\infty)$
\begin{equation}	\label{eq:Udiffusive}
	\sum_{x \in \Z^d: \; |x| > M \sqrt{n}} 
	\frac{\bsU_{N,\lambda}(n,x)}{U_{N,\lambda}(n)}
	\le \frac{C}{M^2} \,, \qquad \forall n \in \N \,, \
	\forall M > 0  \,.
\end{equation}
\end{theorem}

\medskip

\section{Applications to disordered systems}
\label{sec:main3}

In this section we discuss applications of our previous results 
to two marginally relevant disordered
systems: the pinning model with tail exponent $1/2$ and the $(2+1)$-dimensional
directed polymer model.
For simplicity, we focus on the case when these models are built
from the simple random walk on $\Z$ and on $\Z^2$, respectively.

Both models contain \emph{disorder}, given by a family
$\omega = (\omega_i)_{i\in\bbT}$ of i.i.d.\ random variables; $\bbT = \N$
for the pinning model, $\bbT = \N \times \Z^2$ for the directed polymer model.
We assume that
\begin{equation} \label{eq:assomega}
	\bbE[\omega_i] = 0 \,, \qquad \bbE[\omega_i^2] = 1 \,, \qquad
	\lambda(\beta) := \log \bbE[\exp(\beta\omega_i)] < \infty \quad \forall \beta > 0 \,.
\end{equation}
An important role is played by
\begin{equation}\label{eq:sigmabeta2}
	\sigma_\beta^2 := e^{\lambda(2\beta)-2\lambda(\beta)}-1 \,.
\end{equation}

Before presenting our results, in order to put them into context and to provide motivation,
we discuss the key notion of \emph{relevance of disorder}.

\subsection{Relevance of disorder}

Both the pinning model and the directed polymer model are Gibbs measures on random
walk paths, which depend on the realization of the disorder.
A key question for these models, and more generally for disordered systems,
is whether an arbitrarily small, but fixed amount of disorder is able to change
the large scale properties of the model without disorder. When the answer is positive
(resp.\ negative), the model is called \emph{disorder relevant} (resp.\ \emph{irrelevant}).
In borderline cases, 
where the answer depends on finer properties,
the model is called \emph{marginally relevant} or \emph{irrelevant}.

Important progress has been obtained in recent years in the mathematical understanding
of the relevance of disorder, in particular for the pinning model, where the problem
can be cast in terms of \emph{critical point shift} (and \emph{critical exponents}).
We refer to~\cite{G10}
for a detailed presentation of the key results and for the relevant literature.

The pinning model based on the simple random walk on $\Z$ is 
\emph{marginally relevant}, as shown in \cite{GLT10}. Sharp
estimates on the critical point shift were more recently obtained in \cite{BL18}.
For the directed polymer model based on the simple random walk on $\Z^2$,
analogous sharp results are given in \cite{BL17}, in terms of \emph{free energy} estimates.

\smallskip

In \cite{CSZ17a} we proposed a different approach to study disorder
relevance: when a model
is disorder relevant, it should be possible to suitably 
\emph{rescale the disorder strength to zero,
as the system size diverges,} and still obtain a non-trivial limiting model where disorder 
is present. Such an \emph{intermediate disorder regime}
had been  investigated in \cite{AKQ14a,AKQ14b} for the directed polymer model
based on the simple random walk on $\Z$, which is disorder relevant.
The starting point to build a non-trivial limiting model is to determine the scaling limits of the
family of \emph{partition functions}, which encode a great deal of information.

The scaling limits of partition functions
were obtained in \cite{CSZ17a} for several models that are 
disorder relevant (see also  \cite{CSZ15}).
However, the case of marginally relevant models
--- which include the pinning model on $\Z$ and
the directed polymer model on $\Z^2$ --- is much more delicate.
In \cite{CSZ17b} we showed that for such models a phase 
transition emerges on a suitable intermediate
disorder scale, and below the critical point, the family of partition functions 
converges to an explicit Gaussian random field (the solution of the additive stochastic 
heat equation, in the case of the directed polmyer on $\Z^2$). 

\smallskip

In this section we focus on a suitable window
around the critical point, which corresponds to a precise way of scaling
down the disorder strength to zero (see \eqref{eq:sigmaresc}
and \eqref{eq:sigmarescbis} below). In this critical window,
the partition functions are expected to converge to a non-trivial limiting random field,
which has fundamental
connections with singular stochastic PDEs (see the discussion in \cite{CSZ17b}).

Our new results, described
in Theorems~\ref{th:pinning} and~\ref{th:dpre} below, give sharp asymptotic estimates
for the second moment of partition functions. 
These estimates, besides providing an important piece of information by
themselves, are instrumental to investigate scaling limits. Indeed,
we proved in the recent paper \cite{CSZ18} that
the family of partition functions of 
the directed polymer on $\Z^2$ admits non-trivial random field limits, 
whose covariance exhibits logarithmic divergence along the diagonal. 
This is achieved by a third moment computation on the partition function, 
where the second moment estimates derived here play a crucial role. 

\subsection{Pinning model}
\label{sec:pinning}

Let $X = (X_n)_{n\in\N_0}$ be the simple symmetric random walk on $\Z$ with
probability and expectation denoted by $\P(\cdot)$ and $\E[\cdot]$, respectively.  We set
\begin{equation} \label{eq:un}
	u(n) := \P(X_{2n}=0) = \frac{1}{2^{2n}} \, \binom{2n}{n}
	= \frac{1}{\sqrt{\pi}} \, \frac{1}{\sqrt{n}} \, \big(1+o(1)\big) \qquad
	\text{as } n\to\infty \,.
\end{equation}
Fix a sequence of i.i.d.\ random variables $\omega = (\omega_n)_{n\in\N}$,
independent of $X$, satisfying \eqref{eq:assomega}.
The (constrained) partition function of the \emph{pinning model} is defined as follows:
\begin{equation} \label{eq:pinning}
	Z_{N}^{\beta} := \E\Big[ e^{\sum_{n=1}^{N-1} 
	(\beta \omega_n - \lambda(\beta)) \ind_{\{X_{2n}=0\}}}
	\, \ind_{\{X_{2N}=0\}}\Big] \,,
\end{equation}
where we work with $X_{2n}$ rather than $X_{n}$ to avoid periodicity issues.

Writing $Z_{N}^{\beta}$ as a polynomial chaos expansion \cite{CSZ17a}
(we review the computation in Appendix~\ref{sec:polychaos}), we obtain the following
expression for the second moment:
\begin{equation} \label{eq:ZN2}
	\bbE[(Z_{N}^{\beta})^2] = \sum_{k \ge 1} \, (\sigma_\beta^2)^{k-1} 
	\sum_{0 < n_1 < \ldots < n_{k-1} < n_k := N}
	u(n_1)^2 \, u(n_2-n_1)^2 \, \cdots \, u(n_k - n_{k-1})^2 \,,
\end{equation}
where $\sigma_\beta^2$ is defined in \eqref{eq:sigmabeta2}.
Let us define
\begin{align} \label{eq:rpinning}
	r(n) & := u(n)^2 = \frac{1}{\pi \, n} \, \big(1 + o(1)\big) \,, \\
	\label{eq:Rpinning}
	R_N &:= \sum_{n=1}^N r(n)
	= \sum_{n=1}^N \bigg\{\frac{1}{2^{2n}} \,
	\binom{2n}{n}\bigg\}^2 = \frac{1}{\pi} \, \log N \, \big(1+o(1)\big) \,,
\end{align}
and denote by $(\tau^{(N)}_k)_{k\in\N_0}$ the renewal process with
increments law given by \eqref{eq:XN}.
Then, recalling \eqref{eq:ZN2}
and \eqref{UN}, for every $N\in\N$ and $1 \le n \le N$ we can write
\begin{equation} \label{eq:repZN}
\begin{split}
	\bbE[(Z_{n}^{\beta})^2] & = \frac{1}{\sigma_\beta^2} 
	\sum_{k\ge1} \big( \sigma_\beta^2 \, R_N \big)^k \,
	\P(\tau_k^{(N)} = n) \\
	& = \frac{1}{\sigma_\beta^2} \, U_{N,\lambda}(n) \,, \qquad
	\text{where} \qquad \lambda := \sigma_\beta^2 \, R_N \,.
\end{split}
\end{equation}
As a direct corollary of Theorem~\ref{th:WR}, we have the following result.

\begin{theorem}[Second moment asymptotics for pinning model]\label{th:pinning}
Let $Z_{N}^{\beta}$ be the partition function of the pinning model based on the
simple symmetric random walk on $\Z$,
see \eqref{eq:pinning}. Define $\sigma_\beta^2$ by \eqref{eq:sigmabeta2}
and $R_N$ by \eqref{eq:Rpinning}.
Fix $\theta \in \R$ and rescale
$\beta = \beta_N$ so that
\begin{equation} \label{eq:sigmaresc}
	\sigma_{\beta_N}^2 
	= \frac{1}{R_N} \bigg(  1 + \frac{\theta}{\log N}
	\big(1+o(1)\big)  \bigg) \qquad \text{as } N \to \infty \,.
\end{equation}
Then, for any fixed $\delta > 0$, the following relation holds as $N \to \infty$:
\begin{equation} \label{eq:pinningUNas}
	\bbE[(Z_{n}^{\beta_N})^2] = \frac{(\log N)^2}{\pi \, N} 
	\, G_{\theta}(\tfrac{n}{N}) \, (1+o(1)) \,, 
	\quad \ \ 
	\text{uniformly for } \ \delta N \le n \le N \,.
\end{equation}
Moreover, the following uniform bound holds, for a
suitable constant $C \in (0,\infty)$:
\begin{equation} \label{eq:pinningunibo}
	\bbE[(Z_{n}^{\beta_N})^2] \le C \, \frac{(\log N)^2}{N} \, 
	G_{\theta} (\tfrac{n}{N}) \,,
	\qquad \forall 1 \le n \le N \,.
\end{equation}
\end{theorem}

In view of \eqref{eq:Rpinning}, it is tempting to replace $R_N$ by $\frac{1}{\pi}\log N$ in 
\eqref{eq:sigmaresc}. However, to do this properly,
a sharper asymptotic estimate on $R_N$ as $N\to\infty$ is needed.
The following result, of independent interest,
is proved in Appendix~\ref{sec:EM}.

\begin{proposition}\label{prop:EM}
As $N\to\infty$
\begin{equation} \label{eq:EM}
	R_N := \sum_{n=1}^N
	\bigg\{ \frac{1}{2^{2n}} \binom{2n}{n} \bigg\}^2 =
	\frac{\log N + \alpha}{\pi} + o(1) \,,
	\qquad \text{with} \quad
	\alpha := \gamma + \log 16 - \pi \,,
\end{equation}
where $\gamma = -\int_0^\infty \log u \, e^{-u} \, \dd u \simeq 0.577$
is the Euler-Mascheroni constant.
\end{proposition}

\begin{corollary}\label{cor:sigmaresc}
Relation \eqref{eq:sigmaresc} can be rewritten as follows,
with $\alpha := \gamma + \log 16 - \pi$:
\begin{equation} \label{eq:sigmaresc2}
	\sigma_{\beta_N}^2 = \frac{\pi}{\log N} \bigg( 1 + 
	\frac{\theta - \alpha}{\log N}
	\big(1+o(1)\big) \bigg) \qquad \text{as } N \to \infty \,.
\end{equation}
\end{corollary}

We stress that identifying the constant $\alpha$ in \eqref{eq:EM}
is subtle, because it is a non asymptotic
quantity (changing \emph{any single term}
of the sequence in brackets modifies the value of $\alpha$!).
To accomplish the task, in Appendix~\ref{sec:EM} we relate $\alpha$ to a truly
asymptotic property, i.e.\ the tail behavior
of the first return to zero of the simple symmetric
random walk on $\Z^2$.

\begin{remark}
From \eqref{eq:repZN}, we note that $\bbE[(Z_n^{\beta_N})^2]$ is in fact the partition 
function of a \emph{homogeneous pinning model}, see~\cite{G07}, 
with underlying renewal $\tau^{(N)}$, which has inter-arrival exponent $\alpha=0$. 
Theorem~\ref{th:pinning} effectively identifies the ``critical window'' for
such a pinning model and determines the asymptotics 
of the partition function in this critical window. 
Analogous results when $\alpha>0$ have been obtained in \cite{S09}. 
\end{remark}

\begin{remark}\label{rem:betaexpl}
Relation \eqref{eq:sigmaresc2}
can be made more explicit, by expressing 
$\sigma_{\beta_N}^2$ in terms of $\beta_N^2$.
The details are carried out in Appendix~\ref{sec:explbeta}.
\end{remark}

\begin{remark}
If one removes the constraint
$\{X_{2N}=0\}$ from \eqref{eq:pinning}, then one obtains the 
\emph{free partition function} $Z_{N}^{\beta, \mathrm{f}}$.
The asymptotic behavior of its second moment can be determined explicitly,
in analogy with Theorem~\ref{th:pinning}, see 
Appendix~\ref{sec:unibo}.
\end{remark}

\subsection{Directed polymer in random environment}
\label{sec:dpre}

Let $S = (S_n)_{n\in\N_0}$ be the simple symmetric random walk on $\Z^2$,
with probability and expectation denoted by $\P(\cdot)$ and $\E[\cdot]$, respectively.
We set
\begin{equation} \label{eq:qn}
	q_n(x) := \P(S_n = x) \,,
\end{equation}
and note that, recalling the definition \eqref{eq:un}
of $u(n)$, we can write
\begin{equation} \label{eq:P2n0}
	\sum_{x\in\Z^2} q_n(x)^2 =
	\P(S_{2n}=0) = \bigg\{ \frac{1}{2^{2n}} \, \binom{2n}{n} \bigg\}^2
	=: u(n)^2 \,,
\end{equation}
where the second equality holds because
the projections of $S$ along the two main
diagonals are independent simple random walks on $\Z/\sqrt{2}$.

Note that $\Cov[S_1^{(i)}, S_1^{(j)}] = \frac{1}{2} \, \ind_{\{i=j\}}$,
where $S_1^{(i)}$ is the $i$-th component of $S_1$, for $i=1,2$.
As a consequence, $S_n / \sqrt{n}$ converges in distribution
to the Gaussian law on $\R^2$ with density $g_{\frac{1}{2}}(\cdot)$ (recall \eqref{eq:gt}).
The random walk $S$ is periodic,
because $(n,S_n)$ takes values in
\begin{equation*}
	\Z^3_\even := \big\{ z = (z_1, z_2, z_3) \in \Z^3: \
	z_1 + z_2 + z_3 \in 2\Z\big\} \,.
\end{equation*}
Then the local central limit theorem gives that, as $n\to\infty$,
\begin{equation} \label{eq:llt0}
	n \, q_n(x) \,=\, 
	g_{\frac{1}{2}}\big(\tfrac{x}{\sqrt{n}}\big) 
	\, 2 \, \ind_{\{(n,x) \in \Z^3_\even\}}
	\,+\, o(1) \,, \qquad \text{uniformly for } x \in \Z^2 \,,
\end{equation}
where the factor $2$ is due to periodicity, because
the constraint $(n,x) \in \Z^3_\even$ restricts $x$ in a 
sublattice of $\Z^2$ whose cells have area equal to $2$.

\smallskip

Fix now a sequence of i.i.d.\ random variables $\omega = (\omega_{n,x})_{(n,x)
\in\N \times \Z^2}$ satisfying \eqref{eq:assomega},
independent of $S$.
The (constrained) partition function of the \emph{directed polymer
in random environment} is defined as follows:
\begin{equation} \label{eq:dpre}
\begin{split}
	\bsZ_{N}^{\beta}(x)
	&:= \E\Big[ e^{\sum_{n=1}^{N-1} (\beta \omega_{n,S_n} - \lambda(\beta)) }
	\, \ind_{\{S_{N}=x\}}\Big] \\
	&= \E\Big[ e^{\sum_{n=1}^{N-1} \sum_{z\in\Z^2} 
	(\beta \omega_{n,z} - \lambda(\beta)) \ind_{\{S_n = z\}} }
	\, \ind_{\{S_{N}=x\}}\Big] \,.
\end{split}
\end{equation}
In analogy with \eqref{eq:ZN2} (see Appendix~\ref{sec:polychaos}),
we have a representation for the second moment:
\begin{equation} \label{eq:ZNdp2}
\begin{split}
	\bbE\big[\big(\bsZ_{N}^{\beta}(x) \big)^2 \big] = \sum_{k \ge 1} \, (\sigma_\beta^2)^{k-1} 
	\sumtwo{0 < n_1 < \ldots < n_{k-1} < n_k = N}{x_1, \ldots, x_k \in \Z^2: \
	x_k = x}
	& q_{n_1}(x_1)^2 \, q_{n_2-n_1}(x_2 - x_1)^2 \, \cdot \\
	& \qquad \cdots \, q_{n_k - n_{k-1}}(x_k - x_{k-1})^2 \,.
\end{split}
\end{equation}

To apply the results in Section~\ref{sec:main2}, we define
for $(n,x) \in \N \times \Z^2$
\begin{equation*}
	p(n,x) := \frac{q_n(x)^2}{u(n)^2} \,, \qquad \text{where} \qquad
	u(n) := \frac{1}{2^{2n}} \, \binom{2n}{n} \,.
\end{equation*}
Note that $p(n,\cdot)$ is a probability kernel on $\Z^2$, by \eqref{eq:P2n0}.
Since $g_t(x)^2 = \frac{1}{4\pi t} g_{t/2}(x)$ (see \eqref{eq:gt}),
it follows by \eqref{eq:llt0} and \eqref{eq:un} that,
uniformly for $x\in\Z^2$,
\begin{equation}\label{eq:lltq}
	n \, p(n,x) \,=\, 
	g_{\frac{1}{4}}\big(\tfrac{x}{\sqrt{n}}\big) \, 
	2 \, \ind_{\{(n,x) \in \Z^3_\even\}}
	\,+\, o(1) \,.
\end{equation}
Thus $p(n,\cdot)$ fulfills
condition (iii) in \eqref{eq:proppnx} with $\sfc = \frac{1}{4}$
(the multiplicative factor 2 is a minor correction, due to periodicity).
Conditions (i) and (ii) in \eqref{eq:proppnx} are also fulfilled.

Let $(\tau^{(N)}, S^{(N)}) = (\tau^{(N)}_k, S^{(N)}_k)_{k\ge 0}$ be
the random walk with increment law given by \eqref{eq:bsXN},
where $r(n)$ and $R_N$ are the same as in \eqref{eq:rpinning}-\eqref{eq:Rpinning}.
More explicitly:
\begin{equation} \label{eq:bsXNb}
	\P\big((\tau^{(N)}_1, S^{(N)}_1) = (n,x)\big) 
	:= \frac{1}{R_N} \, q_n(x)^2 \, \ind_{\{1,\ldots, N\}}(n) \,.
\end{equation}
Recalling \eqref{eq:ZNdp2} and \eqref{UN2}, we can write
\begin{equation} \label{eq:ZNdp2b}
\begin{split}
	\bbE\big[\big(\bsZ_{n}^{\beta}(x)\big)^2 \big] 
	&= \frac{1}{\sigma_\beta^2} 
	\sum_{k\ge1} \big(\sigma_\beta^2 \,  R_N \big)^k \,
	\P(\tau_k^{(N)} = n, \, S_k^{(N)}=x) \\
	& = \frac{1}{\sigma_\beta^2} \, \bsU_{N,\lambda}(n,x) \,, \qquad
	\text{where} \qquad \lambda := \sigma_\beta^2 \,  R_N \,.
\end{split}
\end{equation}
As a corollary of Theorem~\ref{th:WR2},
taking into account periodicity, we have the following result.

\begin{theorem}[Second moment asymptotics for directed polymer]\label{th:dpre}
Let $\bsZ_{N}^{\beta}(x)$ be the partition function of the directed 
polymer in random environment based on the
simple symmetric random walk on $\Z^2$,
see \eqref{eq:dpre}. Define $\sigma_\beta^2$ by \eqref{eq:sigmabeta2}
and $R_N$ by \eqref{eq:Rpinning}.
Fix $\theta \in \R$ and rescale
$\beta = \beta_N$ so that
\begin{equation} \label{eq:sigmarescbis}
	\sigma_{\beta_N}^2 
	= \frac{1}{R_N} \bigg(  1 + \frac{\theta}{\log N}
	\big(1+o(1)\big)  \bigg) \qquad \text{as } N \to \infty \,.
\end{equation}
For any fixed $\delta > 0$, the following relation holds as $N \to \infty$:
\begin{equation} \label{eq:dpreUNas}
\begin{split}
	& \bbE\big[\big(\bsZ_{n}^{\beta_N}(x)\big)^2\big] = \frac{(\log N)^2}{\pi \, N^2} 
	\, G_{\theta}\big(\tfrac{n}{N}\big) \, g_{\frac{n}{4N}}
	\big(\tfrac{x}{\sqrt{N}}\big)
	\, 2 \, \ind_{\{(n,x) \in \Z^3_\even\}} \, (1+o(1)) \,, \\
	& \rule{0pt}{1.5em}\qquad
	\text{uniformly for} \ \, \delta N \le n \le N , \
	|x| \le \tfrac{1}{\delta}\sqrt{N}  \,.
\end{split}
\end{equation}
\end{theorem}

\begin{remark}
Relation \eqref{eq:sigmarescbis} 
can be equivalently rewritten as relation \eqref{eq:sigmaresc2},
as explained in Corollary~\ref{cor:sigmaresc}. These
conditions on $\sigma_{\beta_N}^2$ can be
explicitly reformulated in terms of $\beta_N^2$,
see Appendix~\ref{sec:explbeta} for details.
\end{remark}

\begin{remark}
Also for the directed polymer model we can define a 
\emph{free partition function} $\bsZ_{N}^{\beta,\mathrm{f}}$,
removing the constraint
$\{S_{2N}=x\}$ from \eqref{eq:dpre}. 
The asymptotic behavior of its second moment is determined in
Appendix~\ref{sec:unibo}.
\end{remark}

\medskip

\section{Preliminary results}
\label{sec:conve}

In this section we prove Propositions~\ref{prop:YN0}, \ref{prop:asGthetac},
\ref{th:bsYY}, and~\ref{prop:YN01}.

\smallskip

We start with Propositions~\ref{prop:YN0} and~\ref{prop:YN01},
for which
we prove convergence in the sense of finite-dimensional distributions.
It is not difficult to obtain convergence in the Skorokhod topology,
but we omit it for brevity, since we do not need such results.

\begin{proof}[Proof of Proposition~\ref{prop:YN0}]

We recall that the renewal process $\tau_k^{(N)}$
was defined in \eqref{eq:deftau}. We set
\begin{equation}\label{eq:YN}
	Y^{(N)}_s := \frac{\tau^{(N)}_{\lfloor s \, \log N \rfloor}}{N}.
\end{equation}

Note that the process $Y^{(N)}_s$ has independent and stationary increments
(for $s \in \frac{1}{\log N}\N_0$), hence the convergence of its finite-dimensional
distributions follows if we show that
\begin{equation}\label{eq:convindist}
	Y^{(N)}_s \, \xrightarrow[\,N\to\infty\,]{} \, Y_s
	\quad \text{in distribution} \,
\end{equation}
for every fixed $s \in [0,\infty)$.
This could be proved by checking the convergence of
Laplace transforms. We give a more direct proof,
which will be useful in the proof of Proposition~\ref{prop:YN01}.

\smallskip

Fix $\epsilon > 0$ and let $\Xi^{(\eps)}$ be
a Poisson Point Process on $[\eps, 1]$ with intensity measure $s \frac{\dd t}{t}$. 
More explicitly, we can write
\begin{equation*}
	\Xi^{(\eps)} = \{t^{(\eps)}_i\}_{i=1,\ldots, \cN^{(\eps)}} \,,
\end{equation*}
where the number of points $\cN^{(\eps)}$ 
has a Poisson distribution:
\begin{equation} \label{eq:lambdaeps}
	\cN^{(\eps)} \sim \mathrm{Pois}(\lambda^{(\eps)}) \,, \qquad \text{where} \qquad
	\lambda^{(\eps)} = \int_\epsilon^1 s \, \frac{\dd t}{t} = s \, \log 1/\epsilon \,,
\end{equation}
while $(t^{(\eps)}_i)_{i\in\N}$ are i.i.d.\ random variables with law
\begin{equation} \label{eq:tieps}
	\P(t^{(\eps)}_i > x)
	= \frac{\int_x^1 s \, \frac{\dd t}{t}}{\int_\eps^1 s \, \frac{\dd t}{t}}
	= \frac{\log x}{\log \epsilon}
	\qquad \text{for} \quad x \in [\epsilon, 1] \,.
\end{equation}

We define
\begin{equation} \label{eq:Yeps}
	Y^{(\eps)}_s := \sum_{t \in\Xi^{(\eps)}} t
	= \sum_{i=1}^{\cN^{(\eps)}} t^{(\eps)}_i \,,
\end{equation}
which is a compound Poisson random variable. Its Laplace transform equals
\begin{equation*}
	\E[e^{-\lambda Y^{(\eps)}_s}]
	= \exp \bigg( -s \int_\epsilon^1 \frac{1-e^{-\lambda t}}{t} \, \dd t \bigg) \,,
\end{equation*}
from which it follows that $\lim_{\eps \to 0} Y^{(\eps)}_s = Y^s$ in distribution
(recall \eqref{eq:LapY}).

Next we define
\begin{equation} \label{eq:YNeps}
	Y^{(N, \eps)}_s  := \frac{1}{N}\sum_{i\in I^{(N, \eps)}_s} T^{(N)}_i \,, \qquad
	\text{where} \qquad
	I^{(N, \eps)}_s:= 
	\big\{ 1\leq i\leq \lfloor s \log N\rfloor:
	\ T^{(N)}_i> \eps N \big\} \,.
\end{equation}
Note that, by \eqref{eq:rn}-\eqref{eq:RN}, for some constant $C \in (0,\infty)$ we can write
\begin{equation}\label{eq:YNapprox1}
\begin{aligned}
\E\big[ \big| Y^{(N)}_s-Y^{(N, \eps)}_s
\big| \big] & = \frac{1}{N} \E\Bigg[\sum_{i\notin I^{(N, \eps)}_s} T^{(N)}_i \Bigg] =
 \frac{\lfloor s\log N\rfloor}{N}  \E\Big[T^{(N)}_1\ind_{\{T^{(N)}_1\leq \eps N\}} \Big] \\
& =  \frac{\lfloor s\log N\rfloor}{N}  \sum_{n=1}^{\lfloor \eps N \rfloor}n \,\frac{r(n)}{R_N}
\le C \, \frac{\lfloor s\log N\rfloor}{N} \frac{\lfloor \eps N \rfloor}{\log N} 
\le C \, \eps s \,.
\end{aligned}
\end{equation}
Thus $Y^{(N)}_s$ and $Y^{(N, \eps)}_s$ are close in distribution
for $\eps > 0$ small, uniformly in $N\in\N$.

The proof of \eqref{eq:convindist} will be completed
if we show that
$\lim_{N\to\infty} Y^{(N, \eps)}_s = Y^{(\eps)}_s$ in distribution,
for any fixed $\epsilon > 0$.
Let us define the point process
\begin{equation*}
	\Xi^{(N, \eps)}:= \bigg\{
	t^{(N,\eps)}_i := \frac{1}{N}T^{(N)}_i: \ \ i \in I^{(N, \eps)}_s \bigg\} \,,
\end{equation*}
so that we can write
\begin{equation*}
	Y^{(N,\eps)}_s := \sum_{t \in\Xi^{(N,\eps)}} t 
	= \sum_{i \in I_s^{(N, \eps)}} t^{(N,\eps)}_i \,.
\end{equation*}
It remains to show that $\Xi^{(N, \eps)}$
converges in distribution to $\Xi^{(\eps)}$ as $N\to\infty$ (recall \eqref{eq:Yeps}).
\begin{itemize}
\item The number of points $|I^{(N, \eps)}_s|$
in $\Xi^{(\eps)}$ has a Binomial distribution $\mathrm{Bin}(n,p)$, with
\begin{equation*}
	n = \lfloor s \log N \rfloor \,, \qquad
	p = \P(T_1^{(N)} > \epsilon N) \sim \frac{\log 1/\epsilon}{\log N} \,,
\end{equation*}
hence as $N\to\infty$ it converges in distribution to $\cN^{(\eps)}
\sim \mathrm{Pois}(\lambda^{(\eps)})$, see \eqref{eq:lambdaeps}.

\item Each point $t^{(N,\eps)}_i \in \Xi^{(N, \eps)}$ has the law of $\frac{1}{N}T_1^{(N)}$
conditioned on $T_1^{(N)} > \eps N$, and it follows by \eqref{eq:rn}-\eqref{eq:RN}
that as $N\to\infty$ this converges in distribution to $t^{(\eps)}_1$, see \eqref{eq:tieps}.
\end{itemize}
This completes the proof of Proposition~\ref{prop:YN0}.
\end{proof}

\begin{proof}[Proof of Proposition~\ref{prop:YN01}]

We recall that the random walk $(\tau_k^{(N)}, S_k^{(N)})$ 
was introduced in \eqref{eq:tauS}.
We introduce the shortcut
\begin{equation} \label{eq:bsYN}
	\bsY^{(N)}_s := (Y^{(N)}_s, V^{(N)}_s):=  \Bigg( \frac{\tau^{(N)}_{\lfloor s \, \log N \rfloor}}{N},
	\frac{S^{(N)}_{\lfloor s \, \log N \rfloor}}{\sqrt{N}} \Bigg), \quad  s\geq 0.
\end{equation}
In analogy with \eqref{eq:convindist}, it suffices to show that for every fixed
$s \in [0,\infty)$
\begin{equation}\label{eq:convindist2}
	\bsY^{(N)}_s \, \xrightarrow[\,N\to\infty\,]{} \, 
	\bsY_s := (Y_s, V^{\sfc}_s)
	\quad \text{in distribution} \,.
\end{equation}

Fix $\eps > 0$ and recall that $Y^{(\eps)}_s$ was defined in \eqref{eq:Yeps}. 
With Proposition~\ref{th:bsYY} in mind, we define
\begin{equation} \label{eq:Vepss}
	V^{(\eps)}_s := \sqrt{\sfc} \, W_{Y^{(\eps)}_s} \,,
\end{equation}
where $W$ is an independent Brownian motion on $\R^d$.
Since $\lim_{\eps \to 0} Y^{(\eps)}_s = Y_s$ in distribution,
recalling Proposition~\ref{th:bsYY} we see that for every fixed $s \in [0,\infty)$
\begin{equation*}
	\bsY^{(\eps)}_s := (Y^{(\eps)}_s, V^{(\eps)}_s )
	\, \xrightarrow[\,\eps\to 0\,]{d}\, \bsY_s  = (Y_s, V^{\sfc}_s) \,.
\end{equation*}

Recall the definition \eqref{eq:YNeps} of $Y^{(N,\eps)}_s$ and
$I^{(N,\eps)}_s$. We define similarly
\begin{equation} \label{eq:VNeps}
	V^{(N, \eps)}_s := 
	\frac{1}{\sqrt N}\sum_{i\in I^{(N,\eps)}_s} X^{(N)}_i \,.
\end{equation}
We showed in \eqref{eq:YNapprox1} that 
$Y^{(N, \eps)}_s$ approximates $Y^{(N)}_s$ in $L^1$, for $\eps > 0$ small.
We are now going to show that 
$V^{(N, \eps)}_s$ approximates $V^{(N)}_s$ in $L^2$.
Recalling \eqref{eq:bsXN}, \eqref{eq:proppnx}, we can write
\begin{equation}\label{eq:EX12}
	\E\big[\big| X^{(N)}_1\big|^2 \,\big|\, T^{(N)}_1=n\big] 
	= \sum_{x\in\Z^2} | x |^2 \, p(n,x)
	\leq c\,  n \,.
\end{equation}
Since conditionally on $(T^{(N)}_i)_{i\notin I^{(N, \eps)}_s}$, $(X^{(N)}_i)_{i\notin I^{(N, \eps)}_s}$ 
are independent with mean $0$, we have
\begin{equation}\label{eq:YNapprox2}
\begin{aligned}
\E\big[ \big| V^{(N)}_s-V^{(N, \eps)}_s\big|^2\big] 
& = \frac{1}{N}\E\big[\big|\sum_{i\notin I^{(N, \eps)}_s}  X^{(N)}_i\big|^2\big] \\
& \leq \frac{c}{N} \E\big[ \sum_{i\notin I^{(N, \eps)}_s} T^{(N)}_i\big] = c\E[Y^{(N)}_s-Y^{(N, \eps)}_s]
\le c \, C \, \eps \, s \,,
\end{aligned}
\end{equation}
where we have applied \eqref{eq:YNapprox1}.
This, together with \eqref{eq:YNapprox1}, 
proves that we can approximate $\bsY^{(N)}_s$ by $\bsY^{(N, \eps)}_s$ 
in distribution,
uniformly in $N$, by choosing $\eps$ small.

To complete the proof of \eqref{eq:convindist2}, it remains to show that, for every fixed $\eps > 0$,
\begin{equation}\label{eq:convindisteps}
	\bsY^{(N,\eps)}_s
	:= \big(Y^{(N,\eps)}_s, V^{(N,\eps)}_s \big) \, \xrightarrow[\,N\to\infty\,]{} \, 
	\bsY^{(\eps)}_s = (Y^{(\eps)}_s, V^{(\eps)}_s )
	\quad \text{in distribution} \,,
\end{equation}
where $V^{(\eps)}_s$ was defined in \eqref{eq:Vepss}.
In the proof of Proposition~\ref{prop:YN0} we showed that
$\Xi^{(N, \eps)}$ converges in distribution to $\Xi^{(\eps)}$
as $N\to\infty$.
By Skorohod's representation theorem, we can construct a coupling such that
$\Xi^{(N, \eps)}$ converges almost surely to $\Xi^{(\eps)}$,
that is the number and sizes of jumps of $Y^{(N,\eps)}_s$ 
converge almost surely to those of $Y^{(\eps)}_s$. 
Given a sequence of jumps of $(Y^{(N, \eps)}_s)_{N\in\N}$, 
say $t^{(N,\eps)}_{i_N} \to t^{(\eps)}_i$ for some jump
$t^{(\eps)}_i$ of $Y^{(\eps)}_s$, we have that 
$X^{(N)}_{i_N}/\sqrt{N}$ converges in distribution to a centered Gaussian random variable 
with covariance matrix $(\sfc \, t_i^{(\eps)} \, I)$, 
by the definition of $X^{(N)}_{i_N}$ in \eqref{eq:bsXN} 
and the local limit theorem in \eqref{eq:proppnx}. 
Therefore, conditionally on all the jumps, the random variables $V^{(N, \eps)}_s$ in
\eqref{eq:VNeps} converges in distribution to the Gaussian 
law with covariance matrix
\begin{equation*}
	\sum_{i=1}^{\cN^{(\eps)}} (\sfc \, t_i^{(\eps)} \, I)
	= \sfc \, Y^{(\eps)}_s \, I\,,
\end{equation*}
which is precisely the law of $V_s^{(\eps)} := \sqrt{\sfc} \, W_{Y^{(\eps)}_s}$.
This proves \eqref{eq:convindisteps}.
\end{proof}

\begin{proof}[Proof of Proposition~\ref{prop:asGthetac}]
Note that $\P(Y_s \le 1)= e^{-\gamma s}/\Gamma(s+1)$, by 
the first line of \eqref{eq:scalingf}.
With the change of variable $u = (\log\frac{1}{t})s $ in \eqref{eq:G0+}, we can write
\begin{equation*}
\begin{split}
	G_\theta(t)
	& = \frac{1}{t} \int_0^\infty
	s \, e^{(\log t )s} \, e^{\theta s}
	\, \P(Y_s \le 1) \, \dd s \\
	& = \frac{1}{t (\log\frac{1}{t})^2} \int_0^\infty u \, e^{-u} \,
	e^{\frac{\theta}{\log (1/t)} u}
	\, \P(Y_{u / \log(1/t)} \le 1) \, \dd u \,.
\end{split}
\end{equation*}
Note that $\P(Y_{u / \log(1/t)} \le 1) = 1 - O(\frac{1}{(\log (1/t))^{2}})$ as $t\downarrow 0$,
for any fixed $u > 0$,
by \eqref{eq:asrip}.
Expanding the exponential,
as $t \downarrow 0$, we obtain  by dominated convergence
\begin{equation*}
\begin{split}
	G_\theta(t)
	& = \frac{1}{t (\log\frac{1}{t})^2}
	\Bigg\{ \int_0^\infty u \, e^{-u} \, \dd u
	\ + \ \frac{\theta}{\log (1/t)}
	\int_0^\infty u^2 \, e^{-u} \, \dd u
	\ + \ O\bigg(\frac{1}{(\log (1/t))^{2}} \bigg) \Bigg\} \,,
\end{split}
\end{equation*}
which coincides with \eqref{eq:Gthetacas}.
\end{proof}

\begin{proof}[Proof of Proposition~\ref{th:bsYY}]
It suffices to compute the joint Laplace transform of
$(Y_s, \, \sqrt{\sfc}\, W_{Y_s})$
and show that it agrees with \eqref{eq:bsLapY}.
For $\rho \in \R^2$, $s \ge 0$, $t > 0$, by independence of $Y$ an $W$,
\begin{equation*}
	\E[e^{\langle \rho, \sqrt{\sfc}\,W_{Y_s}\rangle} \,|\, Y_s = t] =
	\E[e^{\langle \rho, \sqrt{\sfc}\,W_{t}\rangle}]
	= \E[e^{\sqrt{\sfc \, t} \, \langle \rho , W_{1}\rangle}]
	= e^{\frac{1}{2} \sfc  |\rho|^2 t } \,.
\end{equation*}
Then for $\lambda \in \R$,
\begin{equation*}
	\E[e^{\lambda Y_s + \langle \rho, \sqrt{\sfc}\, W_{Y_s}\rangle} ] =
	\E[e^{(\lambda + \frac{1}{2} \sfc |\rho|^2) Y_s} ] =
	\exp \bigg\{s \, \int_0^1 (e^{(\lambda + \frac{1}{2} \sfc |\rho|^2) t} - 1) \,
	\frac{1}{t} \, \dd t \bigg\} \,,
\end{equation*}
where we have applied \eqref{eq:LapY}. It remains to observe that, by explicit computation,
\begin{equation} \label{eq:makingint}
\begin{split}
	e^{(\lambda + \frac{1}{2} \sfc |\rho|^2) t} - 1 &= \int_{\R^2}
	(e^{\lambda t + \langle \rho, x \rangle} - 1)
	\, g_{\sfc t}(x) \, \dd x \,,
\end{split}
\end{equation}
which gives \eqref{eq:bsLapY}.
\end{proof}

\medskip

\section{Proof of Proposition~\ref{prop:sharp}}
\label{sec:renewal}

This section is devoted to the proof of Proposition~\ref{prop:sharp}.
Let us rewrite relation \eqref{eq:pezzofin0intro}:
\begin{equation}\label{eq:pezzofin0}
	\P\big(\tau_k^{(N)}=n\big) 
	\ \le \ C \, k \, \P\big(T_1^{(N)} = n\big) \,
	\, \P\big(T_1^{(N)} \le n\big)^{k-1}
	\, e^{-\frac{c \, k}{\log n + 1} \, \log^+\frac{c\, k}{\log n + 1}}  \,.
\end{equation}
The strategy, as in
\cite{AB16}, is to isolate the contribution of the largest increment
$T^{(N)}_i$.
Our analysis is complicated by the fact that our renewal processes $\tau^{(N)}$ 
varies with $N\in\N$.

\medskip

Before proving Proposition~\ref{prop:sharp}, we derive some useful consequences.
We recall that the renewal process $(\tau^{(N)}_k)_{k\ge 0}$ 
was defined in \eqref{eq:deftau}.

\begin{proposition}\label{LLTupper}
There are constants $C \in (0,\infty)$,
$c \in (0,1)$ and, for every $\epsilon > 0$, $N_\epsilon \in \N$
such that for all $N \ge N_\epsilon$, $s \in (0,\infty)
\cap \frac{1}{\log N}\N$, $t \in (0,1] \cap \frac{1}{N}\N$ we have
\begin{equation}\label{eq:pezzofinlevy}
\begin{split}
	\P\big(\tau^{(N)}_{s \log N} = t N \big) 
	& \ \le \ C \, \frac{1}{N} \, \frac{s}{t} \, t^{(1-\epsilon)s} \,
	e^{- c s \, \log^+ (cs)} \,.
\end{split}
\end{equation}
Recalling that $f_s(t)$ is the density of $Y_s$, see \eqref{eq:scalingf},
it follows that for $N \in \N$ large enough
\begin{equation}\label{eq:pezzofinlevydens}
\begin{split}
	\P\big(\tau^{(N)}_{s \log N} = t N \big) 
	& \ \le \ C' \, \frac{1}{N} \, f_{cs}(t) \,.
\end{split}
\end{equation}
\end{proposition}

\begin{proof}
Let us prove \eqref{eq:pezzofinlevydens}.
Since $\Gamma(s+1) = e^{s (\log s - 1) + \log(\sqrt{2\pi}s)}(1+o(1))$
as $s\to\infty$, by Stirling's formula,
and since $\gamma \simeq 0.577 < 1$, it follows by \eqref{eq:scalingf} that
there is $c_1 > 0$ such that
\begin{equation} \label{eq:bdb}
	f_s(t) \ge c_1 \, \frac{s}{t} \, t^{s} \, e^{-s \, \log^+(s)} \,,
	\qquad \forall t\in (0,1] \,, \ \forall s \in (0,\infty) \,.
\end{equation}
Then, if we choose $\epsilon = 1-c$ in
\eqref{eq:pezzofinlevy}, we see that \eqref{eq:pezzofinlevydens} follows (with
$C' = C / (c c_1)$).

\smallskip

In order to prove \eqref{eq:pezzofinlevy},
let us derive some estimates.
We denote by $c_1, c_2, \ldots$ generic absolute constants in $(0,\infty)$.
By \eqref{eq:XN}-\eqref{eq:RN}, 
\begin{equation} \label{eq:finot0}
	\P\big(T_1^{(N)}\leq r\,\big) = \frac{R_r}{R_N}
	\le c_1 \, \frac{\log r}{\log N} \,, \qquad \forall r,N\in\N \,.
\end{equation}
At the same time
\begin{equation} \label{eq:finot}
	\P\big(T_1^{(N)}\leq r\,\big) = \frac{R_r}{R_N}
	= 1 - \frac{R_N - R_r}{R_N} \le e^{-\frac{R_N - R_r}{R_N}} \,.
\end{equation}
By \eqref{eq:rn}, we can fix $\eta > 0$ small enough so that
$\frac{R_N - R_r}{R_N} \ge \eta \, \frac{\log (N/r)}{\log N}$
\emph{for all $r , N \in \N$ with $r\le N$}. Plugging this into \eqref{eq:finot}, we obtain
a bound that will be useful later:
\begin{equation} \label{eq:ulat}
	\P\big(T_1^{(N)}\leq r\,\big) \le 
	\bigg( \frac{r}{N} \bigg)^{\frac{\eta}{\log N}} \,,
	\qquad \forall N \in \N, \ \forall r=1,\ldots, N \,.
\end{equation}

We can sharpen this bound. For every $\epsilon > 0$,
let us show that there is $N_\epsilon < \infty$ such that
\begin{equation}\label{eq:itrem}
	\P\big(T_1^{(N)} \le r\big) \le
	\bigg(\frac{r}{N}\bigg)^{\frac{1-\epsilon}{\log N}} \,,
	\qquad \forall N \ge N_\epsilon \,, \ \forall r = 1,2,\ldots, N \,.
\end{equation}
We first consider the range $r \le N^\theta$, where $\theta := e^{-1} / c_1$. Then, by \eqref{eq:finot0},
\begin{equation*}
	\P\big(T_1^{(N)}\leq r\,\big) \le 
	\P\big(T_1^{(N)}\leq N^\theta \,\big) \le c_1 \, \theta = e^{-1}
	= \big(\tfrac{1}{N}\big)^{\frac{1}{\log N}}
	\le \big(\tfrac{r}{N}\big)^{\frac{1}{\log N}}
	\le \big(\tfrac{r}{N}\big)^{\frac{1-\epsilon}{\log N}} \,.
\end{equation*}
Next we take $r \ge N^\theta$. Then $\frac{R_N - R_r}{R_N} \ge (1-\epsilon) \, \frac{\log (N/r)}{\log N}$ for $N$ large enough, by \eqref{eq:rn},
which plugged into \eqref{eq:finot} completes the proof of \eqref{eq:itrem}.
We point out that the bounds \eqref{eq:ulat}, \eqref{eq:itrem}
are poor for small $r$, but they provide a simple and
unified expression, valid for all $r=1,\ldots, N$.

\smallskip

We can finally show that \eqref{eq:pezzofinlevy} follows by
\eqref{eq:pezzofin0} (from Proposition~\ref{prop:sharp})
where we plug $k = s \log N$ and $n = tN$,
for $s \in (0,\infty) \cap \frac{1}{\log N}\N_0$ and $t \in (0,1] \cap \frac{1}{N}\N$.
Indeed, note that:
\begin{itemize}
\item by \eqref{eq:XN}-\eqref{eq:RN} we have
$k \, \P\big(T_1^{(N)} = n\big) \le c_2 \frac{k}{(\log N) n}
= c_2 \, \frac{1}{N} \, \frac{s}{t}$;

\item since $\frac{k}{\log n + 1} \ge \frac{k}{\log N + 1} \ge c_3 \, s$ for $n \le N$,
the last term in \eqref{eq:pezzofin0} matches with
the corresponding term in \eqref{eq:pezzofinlevy};

\item by \eqref{eq:itrem}
we have $\P\big(T_1^{(N)} \le n\big)^{k-1} 
\le t^{(1-\epsilon) s} \,
t^{-\frac{1}{\log N}} \le t^{(1-\epsilon) s} \,
(\frac{1}{N})^{-\frac{1}{\log N}} = e \, t^{(1-\epsilon) s}$,
because $t \ge \frac{1}{N}$,
hence \eqref{eq:pezzofinlevy} is deduced.\qedhere
\end{itemize}
\end{proof}

\smallskip

Before starting with the proof of Proposition~\ref{prop:sharp},
we derive some large deviation estimates.
We start by giving an upper bound on
the upper tail $\P(\tau^{(m)}_k \ge n)$ 
for arbitrary $m,k,n \in \N$. This is a Fuk-Nagaev type inequality,
see \cite[Theorem 1.1]{cf:Nag}.

\begin{lemma}\label{markov}
There exists a constant $C \in (1,\infty)$ such that for all
$m\in\N$ and $s,t \in [0,\infty)$
\begin{equation}\label{eq:alistable}
	\P\big( \tau^{(m)}_{\lfloor s (\log m + 1) \rfloor} \ge t m \big)
	\,\le\, e^{-t \, \log^+ (\frac{t}{Cs})} \,.
\end{equation}

\end{lemma}
\begin{proof}
We are going to prove that for all $m,n, k\in \N$
\begin{align} \label{eq:ali}
	\P\big(\tau^{(m)}_k \geq n\big) 
	\,\leq \, \bigg(\frac{C \,k\,m}{n \, (\log m + 1)}
	\wedge 1 \bigg)^{\frac{n}{m}} \,,
\end{align}
which is just a rewriting of \eqref{eq:alistable}.
For some $c_1 <\infty$ we have
$\E[\tau^{(m)}_1] \le c_1 \, \frac{m}{\log m + 1}$,
see \eqref{eq:rn}-\eqref{eq:XN}.
Since $\tau^{(m)}_1 \le m$, we can estimate
\begin{equation*}
\begin{split}
	\E\big[e^{\gl \tau_1^{(m)}}\big] 
	& = 1 + \sum_{j\ge 1} \frac{\lambda^j}{j!} \, \E[(\tau_1^{(m)})^j]	
	\le 1 + \sum_{j\ge 1} \frac{\lambda^j}{j!} \, m^{j-1} \E[\tau_1^{(m)}]	
	\le 1 + \frac{c_1}{\log m + 1} \sum_{j\ge 1} \frac{(\lambda m)^j}{j!} \\
	& \le 1 + \frac{c_1}{\log m + 1} \, e^{\lambda m} \,.
\end{split}
\end{equation*}
This yields, by Markov inequality, for all $\lambda \ge 0$,
\begin{align}
	\P\big(\tau^{(m)}_k \geq n\big)
	&\leq e^{-\gl n} \, \E\big[e^{\gl \tau_1^{(m)}}\big]^k
	= e^{- \gl n} \,\big( 1+\tfrac{c_1}{\log m + 1} e^{\gl m}  \big)^k  \notag \\
	\label{Markov}
	& \leq  e^{- \gl n}  \exp\big( \tfrac{c_1\,k}{\log m + 1} e^{\gl m} \big).
\end{align}
We now choose $\lambda$ such that
\begin{align*}
	\tfrac{k}{\log m + 1} \, e^{\gl m} = \tfrac{n}{m} \,, \qquad \text{that is} \qquad
	e^{-\lambda} = \big( \tfrac{m \, k}{n \, (\log m + 1)} \big)^{\frac{1}{m}} \,.
\end{align*}
If $\frac{m \, k}{n \, (\log m + 1)} > 1$ relation \eqref{eq:ali} holds trivially,
so we assume $\frac{m \, k}{n \, (\log m + 1)} \le 1$, so that $\lambda \ge 0$.
This choice of $\lambda$, when plugged into \eqref{Markov}, 
gives \eqref{eq:ali} with $C = e^{c_1+1}$.
\end{proof}

\begin{remark}
Heuristically, the upper bound \eqref{eq:ali} corresponds to requiring that
among the $k$ increments $T^{(m)}_1, T^{(m)}_2, \ldots, T^{(m)}_k$
there are $\ell := \frac{n}{m}$ ``big jumps'' of size comparable to $m$.
To be more precise, let us first recall the standard Cramer large deviations bound
\begin{equation*}
	\P(\mathrm{Pois}(\lambda) > t) \le 
	e^{- t (\log \frac{t}{\lambda} - 1)} = \big(\tfrac{e \lambda}{t}\big)^t \,, \qquad
	\forall \lambda, t > 0 \,.
\end{equation*}
Now fix $a \in (0,1)$ and note that
$\P(T^{(m)}_1 > a m) \sim p_m := \frac{c}{\log m}$ (where $c = \log \frac{1}{a}$).
If we denote by $N_{k,a m}$ the number of increments
$T_i^{(m)}$ of size at least $a m$, we can write
\begin{equation*}
	\P(N_{k,m} \ge \ell) =
	\P(\mathrm{Bin}(k,p_m) \ge \ell) \approx \P( \mathrm{Pois}(k \, p_m) \ge \ell) 
	\le \Big( \frac{e \, k \, p_m}{\ell}\Big)^\ell \,.
\end{equation*}
If we choose $\ell = \frac{n}{m}$, we obtain the same bound as in \eqref{eq:ali}.
This indicates that the strategy just outlined captures the essential
contribution to the event $\{ \tau^{(m)}_k \geq n \}$.
\end{remark}

We complement Lemma~\ref{markov} with a bound on the 
lower tail $\P(\tau^{(m)}_k \le n)$.

\begin{lemma}\label{markov2}
There exists a constant $c \in (0, 1)$ such that
for all $m\in\N$ and $s,t \in [0,\infty)$
\begin{equation} \label{eq:largeMstable}
	\P\big(\tau^{(m)}_{\lfloor s (\log m + 1) \rfloor} \le t m\big)
	\,\le\, e^{-c\, s \, \log^+(\frac{cs}{t})} \,.
\end{equation}
\end{lemma}

\begin{proof}
We are going to prove that there exists $c \in (0,1)$ such that
for all $m,n,k\in\N$
\begin{equation}\label{eq:largeM}
	\P\big(\tau^{(m)}_k \le n\big) \le 
	\bigg(\frac{n \, (\log m + 1)}{c \, k\,m}
	\wedge 1 \bigg)^{\frac{c \, k}{\log m + 1}} \,,
\end{equation}
which is just a rewriting of \eqref{eq:largeMstable}.
For $\lambda \ge 0$ we have
\begin{equation} \label{eq:fina}
	\P(\tau^{(m)}_k \le n) = \P(e^{-\lambda \tau^{(m)}_k} \ge e^{-\lambda n})
	\le e^{\lambda n} \, \E[e^{-\lambda T_1^{(m)}}]^k \,.
\end{equation}

Next we evaluate, by \eqref{eq:rn}-\eqref{eq:RN},
\begin{equation*}
\begin{split}
	\E[e^{-\lambda T_1^{(m)}}] = \sum_{n=1}^m e^{-\lambda n} \, \frac{r(n)}{R_m}
	=1- \sum_{n=1}^m (1-e^{-\lambda n}) \, \frac{r(n)}{R_m}
	\le 1 - \frac{c_1}{\log m + 1} \sum_{n=1}^m \frac{1-e^{-\lambda n}}{n} \,,
\end{split}
\end{equation*}
for some $c_1 \in (0,1)$.
Since the function $x \mapsto \frac{1-e^{-x}}{x}$ is decreasing for $x \ge 0$, we can bound
\begin{equation*}
\begin{split}
	\E[e^{-\lambda T_1^{(m)}}] \le 
	1 - \frac{c_1}{\log m + 1} \int_{1}^{m+1} 
	\frac{1-e^{-\lambda t}}{t} \, \dd t
	= 1 - \frac{c_1}{\log m + 1} \int_{\lambda}^{\lambda(m+1)} 
	\frac{1-e^{-x}}{x} \, \dd x \,.
\end{split}
\end{equation*}
We are going to fix $\frac{1}{m} \le \lambda \le 1$. Restricting the integration
to the interval $1 \le x \le \lambda m$ and bounding $1-e^{-x} \ge (1-e^{-1})$
we obtain, for $c_2 := (1-e^{-1}) c_1$,
\begin{equation*}
	\E[e^{-\lambda T_1^{(m)}}] \le 
	1 - \tfrac{c_2}{\log m + 1} \, \log(\lambda m)
	\le e^{- \frac{c_2}{\log m + 1} \, \log(\lambda m) }
	= \Big( \tfrac{1}{\lambda m} \Big)^{\frac{c_2}{\log m + 1}} \,.
\end{equation*}
Looking back at \eqref{eq:fina}, we obtain
\begin{equation} \label{eq:fina2}
	\P(\tau_k^{(m)} \le n) \le e^{\lambda n} \, 
	\Big( \tfrac{1}{\lambda m} \Big)^{c_2 \, \frac{k}{\log m + 1}} \,.
\end{equation}

We are ready to prove \eqref{eq:largeM}. Assume first that $k \le n$ and let
$\lambda := \frac{k}{n \, (\log m + 1) } \le 1$.
We may assume that $\lambda \ge \frac{1}{m}$,
because for $\lambda m < 1$
the right hand side of \eqref{eq:largeM} equals $1$ and there is nothing to prove.
We then have $\frac{1}{m} \le \lambda \le 1$.
Plugging $\lambda$ into \eqref{eq:fina2}
gives
\begin{equation*}
	\P(\tau_k^{(m)} \le n) \le \bigg( 
	\frac{e^{\frac{1}{c_2}} \,n\, (\log m + 1)}{k m}
	\wedge 1 \bigg)^{c_2 \, \frac{k}{\log m + 1}} \, ,
\end{equation*}
where we inserted ``$\wedge 1$'' because the left hand side is a probability.
Since $x \ge e^{-1/x}$ for $x \ge 0$, 
in the exponent we can replace $c_2$ by $c := e^{-1/c_2}$, which yields
\eqref{eq:largeM}.

Finally, for $k > n$ the left hand side of \eqref{eq:largeM} vanishes,
because $\tau_k^{(m)} \ge k$.
\end{proof}

\begin{remark}
For renewal processes with a density, see Remark~\ref{rem:density},
the proof of Lemma~\ref{markov2} can be easily adapted, replacing sums
by integrals. The only difference is that
we no longer have $\tau_k^{(m)} \ge k$, so the case $k > n$
needs a separate treatment. To this purpose, we note that
\begin{equation*}
\begin{split}
	\E[e^{-\lambda T_1^{(m)}}] = \int_0^m e^{-\lambda t} \, \frac{r(t)}{R_m} \,
	\dd t
	& \le \frac{c_0}{\log m + 1}
	\int_0^\infty e^{-\lambda t} \, \dd t
	= \frac{c_0}{\log m + 1} \, \frac{1}{\lambda}\,,
\end{split}
\end{equation*}
for some $c_0 \in (1,\infty)$.
If we set $\lambda = \frac{k}{n}$, by \eqref{eq:fina}  we get
\begin{equation} \label{eq:debou}
	\P(\tau^{(m)}_k \le n)
	\le \bigg(\frac{n}{k} \bigg)^k \,
	\bigg( \frac{e \, c_0}{\log m + 1}\bigg)^{k}  \,.
\end{equation}
We now give a lower bound on the right hand side of \eqref{eq:largeM}.
We assume that the fraction therein is $\le 1$, otherwise there is nothing
to prove. Since $c \in (0,1)$, for $k > n$ we can bound 
\begin{equation*}
	\bigg(\frac{n}{k}
	\bigg)^{\frac{c \, k}{\log m + 1}} \,
	\bigg(\frac{\log m + 1}{c \, m}
	\bigg)^{\frac{c \, k}{\log m + 1}}
	\ge \bigg(\frac{n}{k}
	\bigg)^{k} \,
	\bigg(\frac{1}{m + 1}
	\bigg)^{\frac{c \, k}{\log m + 1}}
	=
	\bigg(\frac{n}{k}
	\bigg)^{k} \,
	e^{- c \, k}
	\ge \bigg(\frac{n}{k}
	\bigg)^{k} \,
	e^{- k} \, .
\end{equation*}
This is larger than the right hand side of \eqref{eq:debou},
if we take $m \ge m_0 := \lfloor \exp(e^2 \, c_0) \rfloor$
(so that $\frac{e \, c_0}{\log m + 1} \le e^{-1}$). This shows
that \eqref{eq:debou} holds for $k > n$ and $m \ge m_0$.

It remains to consider the case $k > n$ and $m < m_0$.
Note that lowering $c$ increases the right hand side
of \eqref{eq:largeM}, so we can assume that
$c \le \frac{\log m_0 + 1}{e \, c_0 \, m_0}$.
Since $m \mapsto \frac{\log m + 1}{m}$ is decreasing for $m \ge 1$, we can
bound the right hand side of \eqref{eq:largeM} from below
(assuming that the fraction therein is $\le 1$) as follows,
for $k > n$ and $m < m_0$t:
\begin{equation*}
	\bigg(\frac{n}{k} \,
	\frac{\log m_0 + 1}{c \, m_0}
	\bigg)^{\frac{c \, k}{\log m + 1}}
	\ge	\bigg(\frac{n}{k}
	\, e \, c_0 \bigg)^{\frac{c \, k}{\log m + 1}}
	\ge	\bigg(\frac{n}{k}
	\, \frac{e \, c_0}{\log m + 1} \bigg)^{\frac{c \, k}{\log m + 1}} \,,
\end{equation*}
which is larger than the right hand side of \eqref{eq:debou}.
This completes the proof of \eqref{eq:largeM} for renewal
processes with a density, as in Remark~\ref{rem:density}.
\end{remark}

\smallskip

\begin{proof}[Proof of Proposition~\ref{prop:sharp}]
We have to prove relation \eqref{eq:pezzofin0} for all
$N,k,n\in\N$ with $n \le N$.

Let us set
\begin{equation*}
	M^{(N)}_k := \max_{1 \le i \le k} T^{(N)}_i \,,
\end{equation*}
and note that $\{\tau_k^{(N)} = n\} \subseteq \{M^{(N)}_k \le n\}$. This yields
\begin{equation} \label{eq:inview}
	\frac{\P\big( \tau_k^{(N)} = n \big)}{\P\big( T_1^{(N)} \le n \big)^k} 
	= \P\big( \tau_k^{(N)} = n \,\big|\, M^{(N)}_k \le n \big)
	= \P\big( \tau_k^{(n)} = n \big) \,,
\end{equation}
where the last equality holds because the random variables $T^{(N)}_i$, conditioned
on $\{T^{(N)}_i \le n\}$, have the same law as $T^{(n)}_i$,
see \eqref{eq:XN}. Let us now divide both sides of \eqref{eq:pezzofin0}
by $\P\big( T_1^{(N)} \le n \big)^k$. The equality
\eqref{eq:inview} and the observation
that $\P(T_1^{(N)} = n)/\P(T_1^{(N)} \le n) = \P(T_1^{(n)}=n)$
show that \eqref{eq:pezzofin0} is implied by
\begin{equation}\label{eq:pezzofin}
\begin{split}
	\P\big(\tau_k^{(n)}=n\big) 
	&\ \le \ C \, k  \, \frac{1}{n \, (\log n + 1)} \,
	\, e^{-\frac{c \, k}{\log n + 1} \, \log^+\frac{c\, k}{\log n + 1}}  \,.
\end{split}
\end{equation}
Note that there is no longer dependence on $N$.

\smallskip

It remains to prove \eqref{eq:pezzofin}.
By Lemma~\ref{markov2}, more precisely by \eqref{eq:largeM}, we can bound
\begin{equation*}
	\P\big(\tau_k^{(n)}=n\big)  \le \P\big(\tau_k^{(n)} \le n\big) 
	\le \bigg(\frac{\log n + 1}{c \, k}
	\wedge 1 \bigg)^{\frac{c \, k}{\log n + 1}}
	= e^{-\frac{c \, k}{\log n + 1} \, \log^+ \frac{c \, k}{\log n + 1}} \,.
\end{equation*}
This shows that \eqref{eq:pezzofin} holds for every $k\in\N$ 
if we take $C = C(n) := n \, (\log n + 1)$.
Then, for any fixed $\bar n \in \N$, we can set
$C := \max_{n \le \bar n} C(n)$ and relation
\eqref{eq:pezzofin} holds for all $n \le \bar n$ and $k\in\N$.
As a consequence, it remains to prove that
there is another constant $C < \infty$ such that
relation \eqref{eq:pezzofin} holds for all $n \ge \bar n$ and $k\in\N$.
Note that $\bar n \in \N$ is arbitrary.

We start by estimating, for any $m \in (1, n]$
(possibly not an integer, for later convenience)
\begin{equation}\label{eq:step1}
\begin{split}
	&\P\big(\tau_k^{(n)}=n\,,\,M_k^{(n)}\in (e^{-1}m, m] \,\big)\\
	&\le k \sum_{r\in (e^{-1}m,m]} \P(T_1^{(n)}=r) \,
	\P\big(\tau^{(n)}_{k-1} = n-r\,,\, M^{(n)}_{k-1} \le r\,\big) \\
	&\leq k 
	\, \max_{r\in (e^{-1}m,m]} \P(T_1^{(n)}=r) \
	\P\big(T_1^{(n)}\leq m\,\big)^{k-1} \!\!\!
	\sum_{r\in (e^{-1}m,m]} \!\!\!\!\!
	\P\big(\tau^{(n)}_{k-1} = n-r\,\big|\, M^{(n)}_{k-1} \le m \big) \,.
\end{split}
\end{equation}
Since $T^{(n)}_i$ conditioned on $T^{(n)}_i\leq m$ is distributed
as $T_i^{(m)} := T_i^{(\lfloor m \rfloor)}$, we get, by \eqref{eq:XN}-\eqref{eq:RN},
\begin{equation} \label{eq:appli}
\begin{split}
	& \P\big(\tau_k^{(n)}=n\,,\,M_k^{(n)}\in (e^{-1}m, m] \,\big) \\
	& \qquad\le \, c_4 \, k \, \frac{1}{m \, (\log n + 1)} \,
	\, \P\big(T_1^{(n)}\leq m\,\big)^{k-1} \,
	\, \P\big(n-m \le \tau^{(m)}_{k-1} < n-e^{-1} m \big) \,.
\end{split}
\end{equation}
We bound $\P(T_1^{(n)} \le m)^{k-1} 
\le (\frac{m}{n})^{\frac{\eta(k-1)}{\log n}} \le e \, (\frac{m}{n})^{\frac{\eta k}{\log n}}$, by \eqref{eq:ulat}.
Choosing $m = e^{-\ell} n$ in \eqref{eq:appli} and summing over
$0 \le \ell \le \log n$, we obtain the key bound
\begin{equation} \label{eq:dechi}
\begin{split}
	& \P\big(\tau_k^{(n)} =n \,\big) 
	\, = \, \sum_{\ell = 0}^{\lfloor \log n \rfloor}
	\P\big(\tau_k^{(n)}=n\,,\,M_k^{(n)}\in (e^{-\ell-1} n, e^{-\ell} n] \,\big)  \\
	& \ \le \, c_4 \, k \, \frac{1}{n \, (\log n + 1)} \,
	\sum_{\ell = 0}^{\lfloor \log n \rfloor}
	\, e^\ell \, \P\big(T_1^{(n)}\leq e^{-\ell} n\,\big)^{k-1} \,
	\P\Big( (1-e^{-\ell})n \le
	\tau^{(e^{-\ell} n)}_{k-1} < (1-e^{-(\ell+1)})n \Big) \,.
\end{split}
\end{equation}
To complete the proof of \eqref{eq:pezzofin}, we show that, for
suitable $C \in (0,\infty)$ and $c\in (0,1)$,
\begin{equation}\label{eq:ourgoal}
\begin{split}
	& \sum_{\ell = 0}^{\lfloor \log n \rfloor} 
	e^\ell \, \P\big(T_1^{(n)}\leq e^{-\ell} n\,\big)^{k-1} \,
	\P\Big( (1-e^{-\ell})n \le
	\tau^{(e^{-\ell} n)}_{k-1} < n \Big)
	\le C \, 
	e^{-\frac{c \, k}{\log n + 1} \log^+ \frac{c\, k}{\log n + 1}} \,.
\end{split}
\end{equation}

\smallskip
\emph{Let $c \in (0,1)$ be the constant in Lemma~\ref{markov2}.}
We recall that we may fix $\bar n$ arbitrarily and 
focus on $n \ge \bar n$.
We fix $c' \in (0,1)$ with $c' > c$,
and we choose $\bar n$ so that, by \eqref{eq:itrem} with $N = n$ 
and $r = e^{-\ell} n$,
\begin{equation*}
	\P\big(T_1^{(n)}\leq e^{-\ell} n\,\big) \le
	(e^{-\ell})^{\frac{c'}{\log n}} \qquad \forall n \ge \bar n\,, \
	\forall \ell = 0,1,\ldots, \lfloor \log n\rfloor \,.
\end{equation*}
Then \eqref{eq:ourgoal} is reduced to showing that
for all $n \ge \bar n$ and $k =1,\ldots, n$
\begin{equation}\label{eq:ourgoall}
\begin{split}
	& \sum_{\ell = 0}^{\lfloor \log n \rfloor} 
	e^\ell \, (e^{-\ell})^{\frac{c' (k-1)}{\log n}} \,
	\P\Big( (1-e^{-\ell})n \le
	\tau^{(e^{-\ell} n)}_{k-1} < n \Big)
	\le C \, 
	e^{-\frac{c \, k}{\log n + 1} \log^+ \frac{c\, k}{\log n + 1}} \,.
\end{split}
\end{equation}

We first consider the regime of $k\in\N$ such that
\begin{equation}\label{eq:regime1}
	k > 1 + \tfrac{2}{c'-c} \, (\log n + 1) \,.
\end{equation}
We use Lemma~\ref{markov2}
to bound the probability in \eqref{eq:ourgoall}.
More precisely, we apply relation \eqref{eq:largeMstable} with $m = e^{-\ell} n$,
$s = \frac{k-1}{\log(e^{-\ell}n) + 1}$, $t = e^\ell$ and
with $\log^+$ replaced by $\log$, to get an upper bound.
Since $e^{-\ell} n \le n$, we get by monotonicity
\begin{equation} \label{eq:treline}
\begin{split}
	\P\big(\tau^{(e^{-\ell} n)}_{k-1} < n \big)
	& \,\le\, 
	e^{-\frac{c \, (k-1)}{\log (e^{-\ell}n) + 1} 
	\log \left( e^{-\ell}
	\, \frac{c \, (k-1)}{\log (e^{-\ell}n) + 1} \right)}
	 \,\le\,  
	e^{-\frac{c \, (k-1)}{\log n + 1} 
	\log \left( e^{-\ell}
	\, \frac{c \, (k-1)}{\log n + 1} \right)} \\
	& \,=\, 
	\Big\{ e^{-\frac{c \, (k-1)}{\log n + 1} 
	\log \frac{c \, (k-1)}{\log n + 1} } \Big\}
	\ \big( e^{\frac{c \, (k-1)}{\log n}} \big)^\ell \,.
\end{split}
\end{equation}
Since $k-1 \ge \frac{k}{2}$ for $k \ge 2$, if we redefine $c/2$ as $c$,
we see that the term in brackets in \eqref{eq:treline} matches with
the right hand side of \eqref{eq:ourgoall} (where we can replace
$\log^+$ by $\log$, by \eqref{eq:regime1} and $\frac{2}{c'-c} > c$). The other
term in \eqref{eq:treline}, when inserted in the left hand side of \eqref{eq:ourgoall}, 
gives a contribution to the sum which is uniformly bounded, by \eqref{eq:regime1}:
\begin{equation*}
	\sum_{\ell=0}^{\lfloor \log n \rfloor} e^\ell \, (e^{-\ell})^{\frac{c' (k-1)}{\log n}} \ 
	\big( e^{\frac{c \, (k-1)}{\log n}} \big)^\ell \,\le\,
	\sum_{\ell=0}^{\infty} \big(e^{1 - (c'-c)\frac{k}{\log n}} \big)^{\ell} 
	\,\le\, \sum_{\ell=0}^{\infty} e^{-\ell} \,<\, \infty \,.
\end{equation*}
This completes the proof of \eqref{eq:ourgoall}
under the assumption \eqref{eq:regime1}.

\smallskip

Next we consider the complementary regime of \eqref{eq:regime1}, that 
is
\begin{equation} \label{eq:regime2}
	k \le A \log n + B \,,
\end{equation}
for suitably fixed constants $A,B$.
In this case the right hand side
of \eqref{eq:ourgoall} is uniformly bounded from below by a 
positive constant. Therefore
it suffices to show that
\begin{equation}\label{eq:ourgoal2}
\begin{split}
	& \sum_{\ell = 1}^{\lfloor \log n \rfloor} 
	e^\ell \, \P\Big( \tfrac{n}{2} \le
	\tau^{(e^{-\ell} n)}_{k-1} < n \Big)
	\le C \, \,,
\end{split}
\end{equation}
where, in order to lighten
notation, we removed from \eqref{eq:ourgoal} the term $\ell=0$
(which contributes at most one)
and then bounded $(1-e^{-\ell})n \ge \tfrac{n}{2}$ for $\ell \ge 1$.

\smallskip

We apply Lemma~\ref{markov} (with the constant $C$ renamed
$D$, to avoid confusion with \eqref{eq:ourgoal2}).
Relation \eqref{eq:alistable} with $m = e^{-\ell} n$,
$s = \frac{k}{\log (e^{-\ell}n) + 1}$, $t = \frac{1}{2}e^\ell$ gives
\begin{equation} \label{eq:from2}
	\P\big(\tau^{(e^{-\ell}n)}_k \geq \tfrac{n}{2}\big) 
	\,\leq \, e^{- \frac{1}{2} e^\ell
	\, \log^+ \left( \frac{e^\ell}{2D} 
	\frac{\log n - \ell + 1}{k} \right)}
	\,=\, e^{-e^\ell \left\{ \frac{1}{2} \, \log^+ \left(
	\frac{1}{2D} \, \frac{1}{x_\ell} \right) \right\}} \,,
\end{equation}
where we have introduced the shorthand
\begin{equation} \label{eq:xell}
	x_\ell := \tfrac{k \, e^{-\ell}}{\log n - \ell + 1} \,.
\end{equation}
For $\ell$ such that $x_{\ell} < \frac{1}{2D e^2}$
the right hand side of \eqref{eq:from2} is at most $e^{-e^\ell}$.
We claim that
\begin{equation}\label{eq:ccll}
	x_{\ell} < \tfrac{1}{2D e^2} \qquad \text{for all 
	$\ell \ge \bar \ell$, where} \qquad \bar \ell := 
	\lfloor \log \big( 4(A+B) D e^2) \rfloor + 1 \,.
\end{equation}
This completes the proof of \eqref{eq:ourgoal2}, because the sum is at most
$\sum_{\ell = 1}^{\bar \ell} e^\ell
+ \sum_{\ell = \bar \ell + 1}^\infty e^\ell \, e^{-e^\ell} < \infty$.

It remains to prove that relation \eqref{eq:ccll} holds in regime
\eqref{eq:regime2}. We recall that we may assume that $n$ is large enough.
Consider first the range $\frac{1}{2} \log n \le \ell \le \lfloor \log n\rfloor$:
then
\begin{equation*}
	x_\ell \le k \, e^{-\ell} \le \tfrac{k}{\sqrt{n}} \le 
	\tfrac{A \log n + B}{\sqrt{n}}
	\,\xrightarrow[n\to\infty]{}\, 0 \,,
\end{equation*}
hence we have $x_{\ell} < \frac{1}{2D e^2}$ for $n$ large enough.
Consider finally the range $\ell < \frac{1}{2} \log n$: then
\begin{equation*}
	x_\ell \le \tfrac{k}{\frac{1}{2} \log n} \, e^{-\ell}
	\le \tfrac{A \log n + B}{\frac{1}{2} \log n} \, e^{-\ell}
	\le 2(A+B) \, e^{-\bar \ell} \le \tfrac{1}{2D e^2} \,,
\end{equation*}
by the definition \eqref{eq:ccll} of $\bar\ell$. This completes the proof.
\end{proof}

\medskip

We conclude this section by extending Proposition~\ref{LLTupper} to the multidimensional
setting. We recall that $(\tau_k^{(N)}, \, S_k^{(N)})$
is defined in \eqref{eq:tauS}.

\begin{proposition}\label{LLTupper2}
There are constants $C \in (0,\infty)$,
$c \in (0,1)$ and, for every $\epsilon > 0$, $N_\epsilon \in \N$
such that for all $N \ge N_\epsilon$, $s \in (0,\infty)
\cap \frac{1}{\log N}\N$, $t \in (0,1] \cap \frac{1}{N}\N$
and $x \in \frac{1}{\sqrt{N}} \Z^d$ we have
\begin{equation}\label{eq:pezzofinlevy2}
	\P\big(\tau^{(N)}_{s \log N} = t N \,, \
	S^{(N)}_{s \log N} = x \sqrt{N} \, \big) 
	\ \le \ C \, \frac{1}{N^{1+\frac{d}{2}}} \, 
	\frac{s}{t^{1+\frac{d}{2}}} \, t^{(1-\epsilon)s} \,
	e^{- c s \, \log^+ (cs)} \,.
\end{equation}
It follows that for $N \in \N$ large enough
\begin{equation}\label{eq:pezzofinlevydens2}
\begin{split}
	\P\big(\tau^{(N)}_{s \log N} = t N \,, \
	S^{(N)}_{s \log N} = x \sqrt{N} \, \big) 
	& \ \le \ C' \, \frac{1}{N} \, \frac{1}{(Nt)^{\frac{d}{2}}} \, f_{cs}(t) \,.
\end{split}
\end{equation}
\end{proposition}

\begin{proof}
We follow closely the proof of Proposition~\ref{LLTupper}.
Relation~\eqref{eq:pezzofinlevydens2} follows from \eqref{eq:pezzofinlevy2}
with $\epsilon = 1-c$,
thanks to the bound \eqref{eq:bdb}, so we focus on \eqref{eq:pezzofinlevy2}.

We will prove an analog
of relation \eqref{eq:pezzofin0}: for all $N,k,n \in \N$ with $n \le N$
and for all $z\in\Z^d$
\begin{equation}\label{eq:pezzofin02}
	\P\big(\tau_k^{(N)}=n \,, \ S_k^{(N)} = z\big) 
	\le \, C \, \frac{k}{n^{\frac{d}{2}}} \, \P\big(T_1^{(N)} = n\big) \,
	\, \P\big(T_1^{(N)} \le n\big)^{k-1}
	\, e^{-\frac{c \, k}{\log n + 1} \, \log^+\frac{c\, k}{\log n + 1}}  .
\end{equation}
Note that the only difference with respect to \eqref{eq:pezzofin0}
is the term $n^{\frac{d}{2}}$ in the denominator.

In the proof of Proposition~\ref{LLTupper} we showed that
\eqref{eq:pezzofinlevy} follows from \eqref{eq:pezzofin0}.
In exactly the same way, relation \eqref{eq:pezzofinlevy2} follows from \eqref{eq:pezzofin02},
by choosing $k = s \log N$, $n = Nt$, $z = x \sqrt{N}$.

\smallskip
It remains to prove \eqref{eq:pezzofin02}. Arguing as in
\eqref{eq:inview}, we remove the dependence on $N$
and it suffices to prove the following analog of
\eqref{eq:pezzofin}: for all $n,k \in \N$ 
and for all $z\in\Z^d$
\begin{equation}\label{eq:pezzofin2}
\begin{split}
	\P\big(\tau_k^{(n)}=n \,, \ S_k^{(n)} = z\big) 
	&\ \le \ C \, \frac{k}{n^{\frac{d}{2}}}  \, \frac{1}{n \, (\log n + 1) } \,
	\, e^{-\frac{c \, k}{\log n + 1} \, \log^+\frac{c\, k}{\log n + 1}}  \,.
\end{split}
\end{equation}
To this purpose, we claim that we can modify \eqref{eq:appli} as follows:
\begin{equation} \label{eq:appli2}
\begin{split}
	& \P\big(\tau_k^{(n)}=n\,,\ S_k^{(n)} = z\,,
	\ M_k^{(n)}\in (e^{-1}m, m] \,\big) \\
	& \qquad\le \, c_4 \, \frac{k}{m^{\frac{d}{2}}} \, \frac{1}{m \, (\log n + 1)} \,
	\, \P\big(T_1^{(n)}\leq m\,\big)^{k-1} \,
	\, \P\big(n-m \le \tau^{(m)}_{k-1} < n-e^{-1} m \big) \,.
\end{split}
\end{equation}
This is because, arguing as in  \eqref{eq:step1}, we can write
\begin{equation*}
\begin{split}
	&\P\big(\tau_k^{(n)}=n\,,\, S_k^{(n)}=x \,,\, M_k^{(n)}\in (e^{-1}m, m] \,\big)\\
	&\ \ \le k \sum_{r\in (e^{-1}m,m]\,,\, y \in \Z^d} \P(T_1^{(n)}=r \,, \,
	X_1^{(n)} = y) \,
	\P\big(\tau^{(n)}_{k-1} = n-r\,,\,
	S^{(n)}_{k-1} = x-y \,, \,  M^{(n)}_{k-1} \le r\,\big) \\
	&\ \ \leq k \,\,
	\Big\{ \max_{r\in (e^{-1}m,m]\,,\, y \in \Z^d} \P(T_1^{(n)}=r \,, \,
	X_1^{(n)} = y) \Big\} \,\, \P\big(T_1^{(n)}\leq m\,\big)^{k-1}  \, \\
	& \qquad\qquad\qquad\qquad\qquad\qquad
	\qquad\qquad \sum_{r\in (e^{-1}m,m]} 
	\P\big(\tau^{(n)}_{k-1} = n-r\,\big|\, M^{(n)}_{k-1} \le m \big) \,,
\end{split}
\end{equation*}
and it follows by \eqref{eq:bsXN}, \eqref{eq:lltnorm} and \eqref{eq:rn}-\eqref{eq:RN} that 
\begin{equation*}
	\max_{r\in (e^{-1}m,m]\,,\, y \in \Z^d} \P(T_1^{(n)}=r \,, \,
	X_1^{(n)} = y)
	\le \frac{C}{\log n + 1} \, \frac{1}{m^{1+\frac{d}{2}}} \,.
\end{equation*}

We can now plug $m = e^{-\ell} n$ into \eqref{eq:appli2} and sum over
$\ell=0,1,\ldots, \lfloor \log n\rfloor$, as in \eqref{eq:dechi}.
This leads to our goal \eqref{eq:pezzofin2}, provided we prove the
following analog of \eqref{eq:ourgoal}:
\begin{equation*}
\begin{split}
	& \sum_{\ell = 0}^{\lfloor \log n \rfloor} 
	e^{(1+\frac{d}{2})\ell} \, \P\big(T_1^{(n)}\leq e^{-\ell} n\,\big)^{k-1} \,
	\P\Big( (1-e^{-\ell})n \le
	\tau^{(e^{-\ell} n)}_{k-1} < n \Big)
	\le C \, 
	e^{-\frac{c \, k}{\log n + 1} \log^+ \frac{c\, k}{\log n + 1}} \,.
\end{split}
\end{equation*}
The only difference with respect to \eqref{eq:ourgoal} is the term
$e^{(1+\frac{d}{2})\ell}$ instead of $e^\ell$ in the sum.
It is straightforward to adapt the lines following \eqref{eq:ourgoal}
and complete the proof.
\end{proof}

\section{Proof of Theorem~\ref{th:WR}: case $T = 1$}
\label{sec:WR}

In this section we prove Theorem~\ref{th:WR} for
$T = 1$. The case $T > 1$ will be deduced in the next Section~\ref{sec:WR>}.
We prove separately the uniform upper bound \eqref{eq:unibo}
and the local limit theorem \eqref{eq:UNas},
assuming throughout the section that $n \le N$ (because $T = 1$).

\smallskip

For later use, we state an immediate corollary of Lemma~\ref{markov2}.

\begin{lemma}\label{large s}
There is a constant $c \in (0, 1)$ such that
for all $N \in \N$ and $s,t \in [0,\infty)$
\begin{equation}\label{eq:larges}
	\P\big(\tau^{(N)}_{\lfloor s \log N \rfloor} \le t N\big) \le 
	e^{s - c \, s \log \frac{s}{t}} \,.
\end{equation}
\end{lemma}

\subsection{Proof of \eqref{eq:unibo}}
\label{sec:unibo0}
Recall the definition \eqref{eq:YN} of $Y^{(N)}_s$.
From the definition \eqref{UN} of $U_{N,\lambda}(n)$
and the upper bound \eqref{eq:pezzofinlevydens} (which we can apply
for $\frac{n}{N} \le 1$), we get
for large $N$
\begin{equation} \label{eq:desaim}
\begin{split}
	U_{N,\lambda}(n) 
	= \sum_{k\ge 0} \lambda^k \,
	\P\Big(Y^{(N)}_{\frac{k}{\log N}} = \tfrac{n}{N} \Big)
	\le C \, \frac{\log N}{N} \,
	\Bigg\{ \frac{1}{\log N} \sum_{k \ge 0} \lambda^k \,
	f_{c\frac{k}{\log N}}\big( \tfrac{n}{N}\big) \Bigg\} \,.
\end{split}
\end{equation}
We now choose $\lambda = \lambda_N$ as in \eqref{eq:lambdaN}.
Then for some $A \in (0,\infty)$ we have
\begin{equation*}
	\lambda_N \le 1 + A \, \tfrac{\theta}{\log N} \le 
	e^{A \, \frac{\theta}{\log N} } \,,
	\qquad \forall N \in \N \,,
\end{equation*}
hence
\begin{equation} \label{eq:riesu}
	U_{N,\lambda_N}(n)  \le C \, 
	\frac{\log N}{N} \,
	\Bigg\{ \frac{1}{\log N} \sum_{k \ge 0} e^{\frac{k}{\log N} \, A \, \theta} \,
	f_{c\frac{k}{\log N}}\big( \tfrac{n}{N}\big) \Bigg\} \,.
\end{equation}
The bracket is a Riemann sum, which converges as $N\to\infty$ to the corresponding
integral. It follows that for every $N\in\N$ we can write,
recalling \eqref{eq:G0},
\begin{equation} \label{eq:riein}
	U_{N,\lambda_N}(n)  \le C' \, 
	\frac{\log N}{N} \,
	\bigg\{ \int_0^\infty e^{s \, A \, \theta} \,
	f_{cs}\big( \tfrac{n}{N}\big) \, \dd s \bigg\} 
	= \frac{C'}{c} \, 
	\frac{\log N}{N} \, G_{\frac{A}{c} \theta}\big(\tfrac{n}{N}\big) \,,
\end{equation}
for some constant $C'$. (The fact that $C'$ is uniform
over $1 \le n \le N$ is proved below.)

To complete the proof of \eqref{eq:unibo}, we can
replace $G_{\frac{A}{c} \theta}\big(\tfrac{n}{N}\big)$ 
by $G_{\theta}\big(\tfrac{n}{N}\big)$, possibly enlarging the constant $C'$,
because the function $t \mapsto G_{\theta}(t)$ is strictly positive, continuous
and its asymptotic behavior as $t \to 0$ for different values of $\theta$
is comparable, by Proposition~\ref{prop:asGthetac}.
(Note that in Theorem~\ref{th:WR} the parameter $\theta$ is fixed.)

\smallskip

We finally prove the following claim:
\emph{we can bound the Riemann sum in \eqref{eq:riesu}
by a multiple of
the coresponding integral in \eqref{eq:riein}, uniformly over $1 \le n \le N$}.
By \eqref{eq:scalingf} we can write
\begin{equation} \label{eq:repfs}
	e^{s A\theta} \, f_{cs}(t) =
	\frac{1}{t} \, \exp\Big( \big(\log t + \frac{A}{c} \, \theta - \gamma\big)cs - \log \Gamma(cs) \Big) \,.
\end{equation}
Since $\log \Gamma(\cdot)$ is smooth and strictly convex, 
given any $t \in (0,\infty)$,
the function $s \mapsto e^{s  A \theta} \, f_{cs}(t)$ 
is increasing for
$s \le \bar{s}$ and decreasing for $s \ge \bar{s}$, where
$\bar{s} = \bar{s}(t, A\, \theta, c)$ is characterized by
\begin{equation} \label{eq:Gamma'}
	(\log \Gamma)'(c \bar{s}) = \log t + \frac{A}{c} \, \theta - \gamma \,.
\end{equation}
Henceforth we fix $t = \frac{n}{N}$, with $1 \le n \le N$.

Let us now define $s_k := \frac{k}{\log N}$ and write
\begin{equation} \label{eq:rierew}
\begin{split}
	\frac{1}{\log N} \sum_{k \ge 0} e^{\frac{k}{\log N} A \, \theta} \,
	f_{c\frac{k}{\log N}}\big( \tfrac{n}{N}\big)
	= \sum_{k \ge 0} \tfrac{1}{\log N} \, e^{s_k \, A \, \theta} \,
	f_{c s_k}\big( \tfrac{n}{N}\big) \,.
\end{split}
\end{equation}
If we set $\bar k:= \max\{k\ge 0: \ s_k \le \bar{s}\}$, so that
$s_{\bar k} \le \bar{s} < s_{\bar k+1}$, we note that
each term in the sum \eqref{eq:rierew}
with $k \le \bar k - 1$ (resp.\ with $k \ge \bar k + 2$) 
can be bounded from above by the corresponding
integral on the interval $[s_k, s_{k+1})$ (resp.\ on the interval
$[s_{k-1}, s_k)$), by monotonicity of the function
$s \mapsto e^{s A \theta} \, f_{cs}(t)$.
For the two remaining terms, corresponding to $k=\bar k$ and $k = \bar k + 1$,
we replace $s_k$ by $\bar s$ where the maximum is achieved. This yields
\begin{equation} \label{eq:sumintsum}
	\frac{1}{\log N} \sum_{k \ge 0} e^{\frac{k}{\log N} A \, \theta} \,
	f_{c\frac{k}{\log N}}\big( \tfrac{n}{N}\big)
	\le \int_0^\infty e^{s A \theta} \,
	f_{c s}\big( \tfrac{n}{N}\big) \, \dd s
	\ + \ \tfrac{2}{\log N} \, e^{\bar s  A \theta} \,
	f_{c \bar s}\big( \tfrac{n}{N}\big) \,.
\end{equation}

It remains to deal with the last term.
Recall that $s \mapsto e^{s A \theta} f_{cs}(\frac{n}{N})$ is maximized
for $s = \bar s$.
We will show that shifting $\bar s$ by $\frac{1}{\log N}$ decreases
the maximum by a multiplicative constant:
\begin{equation} \label{eq:ficl}
	c := \sup_{N\in\N, \ 1 \le n \le N}  \
	\frac{e^{\bar s \, A \, \theta} \,
	f_{c \bar s}( \tfrac{n}{N})}{e^{(\bar s + \frac{1}{ \log N}) \, A \, \theta} \,
	f_{c(\bar s + \frac{1}{\log N})}( \tfrac{n}{N})}
	\ < \ \infty \,.
\end{equation}
Since $s \mapsto e^{s A \theta} f_{cs}(\frac{n}{N})$ is decreasing for $s \ge \bar s$,
we can bound the last term in \eqref{eq:sumintsum} as follows:
\begin{equation*}
	\tfrac{2}{\log N} \, e^{\bar s \, A \, \theta} \,
	f_{c \bar s}\big( \tfrac{n}{N}\big)
	\le 2 c \, \int_{\bar s}^{\bar s + \frac{1}{\log N}}
	e^{s  A \theta} \,
	f_{c s}\big( \tfrac{n}{N}\big) \, \dd s
	\le 2 c \, \int_{0}^\infty
	e^{s A \theta} \,
	f_{c s}\big( \tfrac{n}{N}\big) \, \dd s \,,
\end{equation*}
which completes the proof of the claim.

It remains to prove \eqref{eq:ficl}. By the representation \eqref{eq:repfs},
the ratio in \eqref{eq:ficl} equals
\begin{equation*}
\begin{split}
	& \exp\big\{ -\big(\log \tfrac{n}{N} + \tfrac{A}{c} \, \theta - \gamma\big) \tfrac{c}{\log N}
	+ \big( \log \Gamma(c \bar s + \tfrac{c}{\log N})
	- \log \Gamma(c\bar s) \big) \big\} \\
	& \ \le \exp\big\{ O(1)
	+ \tfrac{c}{\log N} \, (\log \Gamma)'(c\bar s + \tfrac{c}{\log N}) \big\}\, ,
\end{split}
\end{equation*}
by $1 \le n \le N$ and by convexity of $\log \Gamma(\cdot)$.
It follows by \eqref{eq:Gamma'} 
that $\bar s$ is uniformly bounded from above
(indeed $\bar s \le A \theta/c - \gamma$, because $t = \frac{n}{N} \le 1$
and $(\log \Gamma)'(\cdot)$ is increasing).
Then $(\log \Gamma)'(c\bar s + \tfrac{c}{\log N}) \le
(\log \Gamma)'(A \theta - c\gamma + \tfrac{c}{\log N})$ is also
uniformly bounded from above.
\qed

\subsection{Proof of \eqref{eq:UNas}}
\label{sec:UNas}

We organize the proof in three steps.

\medskip

\noindent
\textbf{Step 1.}
We first prove an ``integrated version'' of \eqref{eq:UNas}.
Let us define a measure $G_\lambda^{(N)}$ on $[0,\infty)$ as follows:
\begin{equation}\label{eq:G^N}
	G_\lambda^{(N)}(\,\cdot\,) :=
	\frac{1}{\log N} \, \sum_{n =0}^{\infty} U_{N,\lambda}(n) 
	\, \delta_{\frac{n}{N}}(\,\cdot\,) \,,
\end{equation}
where $\delta_t(\,\cdot\,)$ is the Dirac mass at $t$,
and $U_{N,\lambda}(\cdot)$ is defined in \eqref{UN}.
Recall also \eqref{eq:G0}.

\begin{lemma}\label{th:weakGU}
Fix $\theta \in \R$ and choose $\lambda = \lambda_N$ as in \eqref{eq:lambdaN}.
As $N\to\infty$, the measure $G_{\lambda_N}^{(N)}$ converges vaguely to
$G_\theta(t) \, \dd t$, i.e.\ for every 
compactly supported continuous $\phi: [0,\infty) \to \R$
\begin{equation} \label{eq:limNint}
	\int_0^{\infty} \phi(t) \, G_{\lambda_N}^{(N)}(\dd t) 
	\, \xrightarrow[\,N\to\infty\,]{} \, \int_0^{\infty}
	\phi(t) \, G_\theta(t) \, \dd t \,.
\end{equation}
\end{lemma}

\begin{proof}
Recalling the definition \eqref{UN} of $U_{N,\lambda}(n)$, we can write
\begin{equation} \label{eq:eqsp}
\begin{split}
	\int_0^{\infty} \phi(t) \, G_{\lambda_N}^{(N)}(\dd t) 
	&\,=\, \frac{1}{\log N} \, 
	\sum_{n =0}^{\infty} U_{N,\lambda}(n) \, \phi\big(\tfrac{n}{N}\big) \\
	&\,=\, \frac{1}{\log N} \, \sum_{k\ge 0} (\lambda_N)^k \, 
	\E \Big[ \phi\big( \tfrac{\tau^{(N)}_k}{N}\big) 
	\Big] \\
	&\,=\, \int_0^\infty
	(\lambda_N)^{\lfloor s \log N \rfloor} \,  \E \Big[ \phi\big( 
	\tfrac{\tau^{(N)}_{\lfloor s \log N \rfloor}}{N}\big) 
	\Big] \, \dd s  \,.
\end{split}
\end{equation}
Note that $\lim_{N\to\infty} (\lambda_N)^{\lfloor s \log N \rfloor} = e^{\theta s}$,
by \eqref{eq:lambdaN}. Similarly, by Proposition~\ref{prop:YN0}
\begin{equation*}
	\lim_{N\to\infty} \E \Big[ \phi\big( 
	\tfrac{\tau^{(N)}_{\lfloor s \log N \rfloor}}{N}\big) 
	\Big] \,=\, \E\big[ \phi\big( Y_s \big) 
	\big] \,.
\end{equation*}
Interchanging limit and integral, which we justify in a moment, we obtain
from \eqref{eq:eqsp}
\begin{equation*}
\begin{split}
	\lim_{N\to\infty}
	\int_0^{\infty} \phi(t) \, G_{\lambda_N}^{(N)}(\dd t) 
	& = \int_0^\infty
	e^{\theta s} \,  \E \big[ \phi\big( Y_s \big) 
	\big] \, \dd s \,.
\end{split}
\end{equation*}
If we write $\E \big[ \phi\big( Y_s \big) 
\big] = \int_0^\infty \phi(t) \, f_s(t) \, \dd t$, we have proved \eqref{eq:limNint}
(recall \eqref{eq:G0}).

Let us finally justify that we can bring the limit inside the integral in \eqref{eq:eqsp}.
Since $(\lambda_N)^{\lfloor s \log N \rfloor} \le e^{C s}$ for some constant $C$,
by \eqref{eq:lambdaN}, and since $\phi$ is bounded,
we can apply dominated convergence on any bounded interval $s \in [0,M]$.
It remains to show that the integral restricted
to $s \in [M,\infty)$ is small for large $M$,
uniformly in $N\in\N$. To this purpose, we use Lemma~\ref{large s}:
since $\phi$ is compactly supported, say in $[0,A]$,
the bound \eqref{eq:larges} yields
\begin{equation*}
	\|\phi\|_\infty \, \int_M^\infty
	e^{Cs} \, \P( \tau^{(N)}_{\lfloor s \log N \rfloor} \le A N ) \, \dd s
	\le \|\phi\|_\infty \, \int_M^\infty
	e^{s (C+1 - c\, \log \frac{s}{A})} \, \dd s \,.
\end{equation*}
If we take $M$ large, so that $c \log \frac{M}{A} \ge C + 2$,
the integral is at most $\int_M^\infty e^{-s} \, \dd s = e^{-M}$.
\end{proof}

\medskip
\noindent
\textbf{Step 2.}
We now derive representation formulas
for $U_{N,\lambda}(n)$ and $G_\theta(t)$:
for any $\bar n, \bar t \in (0,\infty)$
\begin{align} \label{eq:Uren}
	U_{N,\lambda}(n) & =
	\lambda \, \sum_{0 \le l < 
	\bar{n} \le m \le n} U_{N,\lambda}(l) \, \P(T_1^{(N)}=m-l)
	\, U_{N,\lambda}(n-m)   \ \ \quad
	\forall n \in \N \cap ( \bar{n},\infty) ,\\
	\label{eq:Gren}
	G_\theta(t) &= \int_{0 < u < \bar{t} \le v < t}
	G_\theta(u) \, \frac{1}{v-u}\, \ind_{(0,1)}(v-u)
	\, G_\theta(t-v) \, \dd u \, \dd v \,, \qquad
	\forall t\in ( \bar{t} , \infty) \,.
\end{align}
(Note that for $t \in (0,1]$ the indicator function $\ind_{(0,1)}(v-u) \equiv 1$ disappears.)

Relation \eqref{eq:Uren} is obtained through a renewal decomposition: if we sum over
the unique index $i \in \{1,\ldots, k\}$ such that 
$\tau^{(N)}_{i-1} < \bar{n}$ while $\tau^{(N)}_i \ge \bar{n}$,
we can write
\begin{equation*}
\begin{split}
	\P(\tau^{(N)}_k=n) 
	&= \sum_{i=1}^k \P\big( \tau^{(N)}_{i-1} < \bar{n},
	\, \tau^{(N)}_i \ge \bar{n}, \, \tau^{(N)}_k = n \big) \\
	&= \sum_{0 \le l < \bar{n} \le m \le n} \
	\sum_{i=1}^k \P\big( \tau^{(N)}_{i-1} = l) \,
	\P\big( T^{(N)}_1 = m-l\big)
	\, \P\big( \tau^{(N)}_{k-i} = n-m \big) \,.
\end{split}
\end{equation*}
Plugging this into the definition \eqref{UN} of $U_{N,\lambda}(n)$,
we obtain \eqref{eq:Uren}.

Next we prove \eqref{eq:Gren}. Define the stopping time $\tau := \inf\{r \in [0,\infty):
Y_r > \bar{t}\}$ and note that $Y_{\tau-} \le \bar{t}$,
$Y_\tau > \bar{t}$. The joint law of $(\tau, Y_{\tau-}, Y_\tau)$ is explicit:
for $r \in (0,\infty)$ and $u \le \bar{t} < v$
\begin{equation*}
\begin{split}
	\P(\tau \in \dd r, \, Y_{\tau-} \in \dd u, \, Y_\tau \in \dd v) & =
	\dd r \, \P(Y_r \in \dd u) \, \nu(\dd v-u) \\
	& =
	\dd r \, \P(Y_r \in \dd u) \, \frac{1}{v-u} \, \ind_{(0,1)}(v-u) \, \dd v \,,
\end{split}
\end{equation*}
by a slight generalization of \cite[Prop.~2 in Ch.~III]{Ber96}. 
By the strong Markov property
\begin{equation*}
\begin{split}
	\P(Y_s \in \dd t) 
	&= \int_{(0,s) \times (0,\bar{t}) \times (\bar{t},t)} \P(\tau \in \dd r, 
	\, Y_{\tau-} \in \dd u , \, Y_\tau \in \dd v) \, \P(Y_{s-r}\in \dd t - v) \\
	&= \int_0^s \dd r
	\int_{0 < u < \bar{t} < v < t} 
	\P(Y_r \in \dd u) \, \frac{1}{v-u} \, \ind_{(0,1)}(v-u) \, \dd v
	\, \P(Y_{s-r}\in \dd t - v) \,,
\end{split}
\end{equation*}
which yields a corresponding relation between densities:
\begin{equation*}
	f_s(t) = \int_0^s \dd r
	\int_{0 < u < \bar{t} < v < t} 
	f_r(u) \, \frac{1}{v-u} \, \ind_{(0,1)}(v-u)\, f_{s-r} (t - v) 
	\, \dd u \, \dd v\,.
\end{equation*}
Multiplying by $e^{\theta s} = e^{\theta r} e^{\theta(s-r)}$ and
integrating over $s \in (0,\infty)$,
we get \eqref{eq:Gren} (recall \eqref{eq:G0}).


\medskip
\noindent
\textbf{Step 3.}
The final step in the proof of \eqref{eq:UNas}
consists in combining formulas \eqref{eq:Uren}-\eqref{eq:Gren}
with Lemma~\ref{th:weakGU}. First of all we note that
in order to prove \eqref{eq:UNas} uniformly for $\delta N \le n \le N$,
it suffices to consider an arbitrary but fixed sequence $n = n_N$ such that
\begin{equation}\label{eq:tN}
	t_N := \frac{n_N}{N} \xrightarrow[\,N\to\infty\,]{} t \in (0,1] \,,
\end{equation}
and prove that
\begin{equation}\label{eq:UNasbis}
	\lim_{N\to\infty} \, \frac{N}{\log N} \, U_{N,\lambda_N}(n_N)
	\,=\, G_\theta(t) \,.
\end{equation}
This implies \eqref{eq:UNas}, as one can prove by contradiction.

Let us prove \eqref{eq:UNasbis}.
Recalling \eqref{eq:G^N}, we first rewrite
\eqref{eq:Uren},
with $\bar{n} = n_N/2$, as a double integral,
setting $u := l/N$ and $v := m/N$, as follows (we recall that $t_N = \frac{n_N}{N}$):
\begin{equation} \label{eq:quasifin}
	\frac{N}{\log N} \, U_{N,\lambda_N}(n_N) =
	\lambda_N \, 
	\int\limits_{0 \le u < \frac{t_N}{2} \le v \le 
	t_N} G_{\lambda_N}^{(N)}\big( \dd u \big) \; \phi^{(N)}(u,v)
	\; G_{\lambda_N}^{(N)}\big( t_N - \dd v \big) \, ,
\end{equation}
where we set, for $0 \le u < v \le 1$,
\begin{equation*} 
	\phi^{(N)}(u,v) := \big( N \, \log N \big) \, \P\big( 
	T_1^{(N)} = \lfloor Nv\rfloor  - \lfloor Nu \rfloor \big) \,.
\end{equation*}
Note that, by \eqref{eq:XN}-\eqref{eq:RN}, we have
\begin{equation} \label{eq:phiNconv}
	\lim_{N\to\infty} \, \phi^{(N)}(u,v) =
	\phi(u,v) := \frac{1}{v-u} \,.
\end{equation}
By Lemma~\ref{th:weakGU} and \eqref{eq:tN}, we have the 
vague convergence of the product measure
\begin{equation}\label{eq:weaco}
	G_{\lambda_N}^{(N)}\big( \dd u \big)
	\, G_{\lambda_N}^{(N)}\big( t_N - \dd v \big) \,\xrightarrow[\,N\to\infty\,]{v}\,
	G_{\theta}(u) \, G_\theta(t-v) \, \dd u \, \dd v \,.
\end{equation}
Since $\lambda_N \to 1$, see \eqref{eq:lambdaN},
by \eqref{eq:phiNconv} and \eqref{eq:weaco}
it is natural to expect that the right hand side of
\eqref{eq:quasifin} converges to the right hand side of \eqref{eq:Gren}
with $\bar{t}=\frac{t}{2}$.
This is indeed the case, as we now show,
which would complete the proof of \eqref{eq:UNasbis}, hence of Theorem~\ref{th:WR}.

\smallskip

We are left with justifying the convergence of the right hand side of
\eqref{eq:quasifin}. The delicate point is that $\phi(u,v)$ in \eqref{eq:phiNconv}
diverges as $v-u \downarrow 0$.
Fix $\epsilon > 0$ and consider the domain
\begin{equation} \label{eq:DN}
	D_{\epsilon} := \big\{ (u,v): \
	v-u \ge \epsilon \, t \big\} \,.
\end{equation}
The convergence in \eqref{eq:phiNconv} holds \emph{uniformly
over $(u,v) \in D_{\epsilon}$},
and the limiting function $\frac{1}{v-u}$ is bounded and continuous 
on $D_\epsilon$. Then, by \eqref{eq:weaco}, the integral in the right
hand side of \eqref{eq:quasifin} restricted on $D_{\epsilon}$
converges to the integral in the right hand side of \eqref{eq:Gren}
restricted  on $D_{\epsilon}$.

To complete the proof, it remains to show that the integral in the
right hand side of \eqref{eq:quasifin} 
restricted on $D_\epsilon^c = \{v-u \le \epsilon \, t\}$
is small for $\epsilon > 0$ small, uniformly in (large) $N\in\N$.
By the definition \eqref{eq:G^N} of $G^{(N)}_\lambda(\cdot)$,
as well as \eqref{eq:XN}-\eqref{eq:RN},
this contribution is bounded by
\begin{equation} \label{eq:topass}
\begin{split}
	& C_1 \, \sum_{\substack{u, v \in \frac{1}{N} \N_0:\\ 0 \le u < \frac{t_N}{2} \le v \le 
	t_N \,, \ v-u \le \epsilon t}}
	\frac{U_{N,\lambda_N}(N u)}{\log N} 
	\, \frac{1}{v-u} \, 
	\frac{U_{N,\lambda_N}(N (t_N - v))}{\log N} \,,
\end{split}
\end{equation}
where $C_1, C_2, \ldots$ are generic constants.
By the upper bound \eqref{eq:unibo}, this is at most
\begin{equation} \label{eq:dosu}
	C_2 \, \frac{1}{N^2}
	\sum_{\substack{u, v \in \frac{1}{N} \N_0:\\ 0 \le u < \frac{t_N}{2} \le v \le 
	t_N \,, \ v-u \le \epsilon t}}
	G_\theta(u) \, \frac{1}{v-u} \, G_\theta(t_N-v) \,.
\end{equation}
Since $t_N \to t$, see \eqref{eq:tN},
we can bound this Riemann sum by the corresponding integral:
\begin{equation*}
	C_3 \, \int\limits_{0 < u < \frac{t}{2} \le v < t \,, \ v-u \le \epsilon \, t}
	G_\theta(u) \, \frac{1}{v-u} \, G_\theta(t-v) \, \dd u \, \dd v \,.
\end{equation*}
Finally, if we let $\epsilon \downarrow 0$, this integral vanishes
by dominated convergence (recall \eqref{eq:Gren}).\qed

\medskip

\section{Proof of Theorem~\ref{th:WR}: case $T > 1$}
\label{sec:WR>}

In this section we prove Theorem~\ref{th:WR} in case $T > 1$.
Without loss of generality, we may assume that $T \in \N$.
The case $T=1$ was already treated in Section~\ref{sec:WR}.
Proceeding inductively, we 
assume that Theorem~\ref{th:WR} holds for some fixed value of $T \in \N$,
and our goal is to prove that \emph{relations \eqref{eq:UNas}
and \eqref{eq:unibo} hold for $TN < n \le (T+1)N$}.

\smallskip

Let us rewrite relation \eqref{eq:Uren} for $\bar{n} = TN$
and \eqref{eq:Gren} for $\bar{t} = T$:
\begin{align}\label{eq:Uren+}
	U_{N,\lambda}(n) & =
	\lambda \, \sum_{0 \le l < TN \le m \le n} U_{N,\lambda}(l) \, \P(T_1^{(N)}=m-l)
	\, U_{N,\lambda}(n-m) \, , \qquad
	\forall n > TN \,,\\
	\label{eq:Gren+}
	G_\theta(t) &= \int_{0 < u < T \le v < t}
	G_\theta(u) \, \frac{1}{v-u}\, \ind_{(0,1)}(v-u)
	\, G_\theta(t-v) \, \dd u \, \dd v \,, \qquad
	\forall t > T \,.
\end{align}

\subsection{Proof of \eqref{eq:UNas}}

Since we focus on the range $TN < n \le (T+1)N$, in  \eqref{eq:Uren+}
we have both $l \le TN$ and $n-m \le N$,
hence we can bound $U_{N,\lambda_N}(l)$ and $U_{N,\lambda_N}(n-m)$
using \eqref{eq:UNas}, by the inductive assumption.
Bounding $\P(T_1^{(N)}=m-l)$ by \eqref{eq:rn}-\eqref{eq:XN}, we get
\begin{align*}
	U_{N,\lambda_N}(n) & \le C_1 \, \frac{(\log N)^2}{N^2}
	\lambda \, \sum_{0 \le l < TN \le m \le n} G_\theta(\tfrac{l}{N}) 
	\, \frac{\ind_{(0,N]}(m-l)}{(\log N) (m-l)} \, 
	\, G_\theta(\tfrac{n-m}{N}) \,,
\end{align*}
for some constants $C_1, C_2$ (possibly depending on $T$).
By Riemann sum approximation
\begin{equation*}
\begin{split}
	U_{N,\lambda_N}(n) &\le C_1 \, \frac{\log N}{N^3} \, 
	\sum_{0 \le l < TN \le m \le n} G_\theta(\tfrac{l}{N}) 
	\, \frac{1}{(\frac{m-l}{N})} \, \ind_{(0,1]}(\tfrac{m-l}{N})
	\, G_\theta(\tfrac{n-m}{N}) \\
	&\le C_2 \, \frac{\log N}{N} \, 
	\int_{0 < u < T \le v < \frac{n}{N}}
	G_\theta(u) \, \frac{1}{v-u}\, \ind_{(0,1)}(v-u)
	\, G_\theta(\tfrac{n}{N}-v) \, \dd u \, \dd v \,.
\end{split}
\end{equation*}
The integral equals 
$G_\theta(\tfrac{n}{N})$ by \eqref{eq:Gren+}, so relation \eqref{eq:UNas} is proved.

(To check that the Riemann sum approximation constant $C_2$ is uniform for 
$TN < n \le (T+1)N$, one can argue as in Step~3 of Section~\ref{sec:WR}:
just repeat the above steps for an arbitrary but fixed sequence $n = n_N$
such that $\frac{n_N}{N} \to t \in [T,T+1]$. We omit the details.)\qed

\subsection{Proof of \eqref{eq:unibo}}

We can follow Step~3 of Section~\ref{sec:WR} almost verbatim: the only difference
is that for $TN < n \le (T+1)N$ we have $\frac{n_N}{N} \to t \in [T,T+1]$.
To pass from \eqref{eq:topass} to \eqref{eq:dosu}, we can apply the upper bound
\eqref{eq:unibo}, by the inductive assumption.\qed

\medskip

\section{Proof of Theorems~\ref{th:Udiffusive}
and~\ref{th:WR2}}\label{sec:Green}

We first prove Theorem~\ref{th:Udiffusive}, i.e.\
relation \eqref{eq:Udiffusive}, which is easy.
We then reduce the proof of Theorem~\ref{th:WR2} to that
of Theorem~\ref{th:WR}, given
in Section~\ref{sec:WR},
proving separately the upper bound \eqref{eq:unibo2}
and the local limit theorem \eqref{eq:bsUNas}.
We assume for simplicity that $T=1$, i.e.\ we focus on $n \le N$,
because
the case $T > 1$ can be deduced arguing as in Section~\ref{sec:WR>}.

\subsection{Proof of \eqref{eq:Udiffusive}}

By \eqref{eq:bsXN} and \eqref{eq:proppnx},
conditioned on the $T_i^{(N)}$'s,
the random variables $X_i^{(N)}$ are independent
with zero mean and $\E\big[\big| X^{(N)}_i\big|^2 \,\big|\, T^{(N)}_i=n_i\big] \leq c\,  n_i$
for some $c < \infty$, see~\eqref{eq:EX12}.
Recalling \eqref{eq:tauS}, we then have
\begin{equation*}
	\E\Big[ \big| S_k^{(N)}\big|^2 \,\Big|\, T_1^{(N)} = n_1, \ldots,
	T_k^{(N)} = n_k\Big] = \sum_{i=1}^k 
	\E\Big[ \big| X_i^{(N)}\big|^2 \,\Big|\, T_i^{(N)} = n_i\Big] 
	\le c \big( n_1 + \ldots + n_k \big) \,,
\end{equation*}
for any choice of $n_1, \ldots, n_k \in \N$. It follows that
$\E\big[ \big| S_k^{(N)}\big|^2 \,\big|\, \tau_k^{(N)} = n \big] \le c \, n$, hence
\begin{equation*}
	\sum_{x\in\Z^2: \, |x| > M \sqrt{n}}
	\P\big( \tau_k^{(N)} = n \,, \ S_k^{(N)} = x \big)
	= \P\big( \tau_k^{(N)} = n \,, \ |S_k^{(N)}| > M \sqrt{n} \big)
	\le \frac{c}{M^2} \, \P\big( \tau_k^{(N)} = n \big) \,,
\end{equation*}
by Markov's inequality.
Multiplying by $\lambda^k$ and summing over $k$,
we obtain \eqref{eq:Udiffusive}.\qed

\subsection{Proof of \eqref{eq:unibo2}}
Recall the definition \eqref{eq:bsYN} of $\bsY^{(N)}_s$.
From the definition \eqref{UN2} of $\bsU_{N,\lambda}(n,x)$
and the upper bound \eqref{eq:pezzofinlevydens2}, we get
for large $N$ and $n \le N$
\begin{equation*}
\begin{split}
	\bsU_{N,\lambda}(n,x) 
	= \sum_{k\ge 0} \lambda^k \,
	\P\Big(\bsY^{(N)}_{\frac{k}{\log N}} = (\tfrac{n}{N}, \tfrac{x}{\sqrt{N}}) \Big)
	\le C \, \frac{\log N}{N} \, \frac{1}{n^{d/2}} \,
	\Bigg\{ \frac{1}{\log N} \sum_{k \ge 0} \lambda^k \,
	f_{c\frac{k}{\log N}}\big( \tfrac{n}{N}\big) \Bigg\} \,.
\end{split}
\end{equation*}
The bracket is the same as in \eqref{eq:desaim}.
We showed in Subsection~\ref{sec:unibo0} that,
if $\lambda = \lambda_N$ is chosen as in \eqref{eq:lambdaN},
the bracket is at most a constant times $G_\theta(\frac{n}{N})$.
This proves \eqref{eq:unibo2}.\qed

\subsection{Proof of \eqref{eq:bsUNas}}
We proceed in three steps.

\medskip

\noindent
\textbf{Step 1.}
We first prove an ``integrated version'' of \eqref{eq:bsUNas}.
We define a measure $\bsG_\lambda^{(N)}$ on $[0,\infty) \times \R^2$ 
by setting
\begin{equation}\label{eq:bsG^N}
	\bsG_\lambda^{(N)}(\,\cdot\,) :=
	\frac{1}{\log N} \, \sum_{n =0}^{\infty}
	\, \sum_{x \in \Z^2} \, \bsU_{N,\lambda}(n,x) 
	\, \delta_{(\frac{n}{N}, \frac{x}{\sqrt{N}})}(\,\cdot\,) \,,
\end{equation}
where we recall that $\bsU_{N,\lambda}(\cdot)$ is defined in \eqref{UN2}.
Recall also the definition \eqref{eq:G} of $\bsG_\theta(t,x)$.

\begin{lemma}\label{th:weakGU2}
Fix $\theta \in \R$ and choose $\lambda = \lambda_N$ as in \eqref{eq:lambdaN}.
For every bounded and continuous $\bs\phi: [0,\infty) \times \R^2 \to \R$,
which is compactly supported in the first variable,
\begin{equation} \label{eq:limNint2}
	\int_{[0,\infty) \times \R^2} \bs\phi(t,x) \, 
	\bsG_{\lambda_N}^{(N)}(\dd t , \dd x) 
	\, \xrightarrow[\,N\to\infty\,]{} \, \int_{[0,\infty) \times \R^2} \bs\phi(t,x) \, 
	\bsG_\theta(t,x) \, \dd t \, \dd x \,.
\end{equation}
\end{lemma}

\begin{proof}
Arguing as in \eqref{eq:eqsp}, we can write
\begin{equation*}
\begin{split}
	& \int_{[0,\infty) \times \R^2} \bs\phi(t,x) \, \bsG_{\lambda_N}^{(N)}(\dd t , \dd x)
	\,=\, \int_0^\infty
	(\lambda_N)^{\lfloor s \log N \rfloor} \,  \E \Big[ \bs\phi\Big( 
	\tfrac{\tau^{(N)}_{\lfloor s \log N \rfloor}}{N}, 
	\tfrac{S^{(N)}_{\lfloor s \log N \rfloor}}{\sqrt{N}}\Big) 
	\Big] \, \dd s  \,.
\end{split}
\end{equation*}
We can exchange $\lim_{N\to\infty}$
with the integral by dominated convergence,
thanks to Lemma~\ref{large s},
as shown in the proof of Lemma~\ref{th:weakGU}.
Then we get, by Proposition~\ref{prop:YN01},
\begin{equation*}
\begin{split}
	\lim_{N\to\infty}
	\int_{[0,\infty) \times \R^2} \bs\phi(t,x) \, \bsG_{\lambda_N}^{(N)}(\dd t , \dd x)
	& = \int_0^\infty
	e^{\theta s} \,  \E \big[ \bs\phi\big( Y_s, V_s^\sfc \big) 
	\big] \, \dd s \\
	& = \int_0^\infty
	e^{\theta s} \, \bigg( \int_{[0,\infty) \times \R^2} \bs\phi( t, x ) \,
	\bsf_s(t,x) \, \dd t \, \dd x \bigg) \dd s \,,
\end{split}
\end{equation*}
which coincides with the right hand side of \eqref{eq:limNint2}
(recall \eqref{eq:G}).
\end{proof}

\medskip
\noindent
\textbf{Step 2.}
Next we give representation formulas
for $\bsU_{N,\lambda}(n,z)$, $\bsG_\theta(t,x)$:
for any $\bar n, \bar t \in (0,\infty)$
\begin{align} \nonumber
	& \bsU_{N,\lambda}(n,x) =
	\lambda \, \!\!\!\!\!
	\sumtwo{0 \le l < \bar{n} \le m \le n}{y, z \in \Z^2} \!\!\!\!\!
	\bsU_{N,\lambda}(l,y) \, \P\big(T_1^{(N)}=m-l, X_1^{(N)}=z-y\big)
	\, \bsU_{N,\lambda}(n-m, x-z) \\
	& \qquad\qquad\qquad\qquad\qquad\qquad\qquad\qquad\qquad\qquad\qquad\qquad
	\forall n\in\N \cap (\bar{n},\infty) \,, \label{eq:Uren2} \\
	\label{eq:Gren2}
	& \bsG_\theta(t,x) = \int\limits_{\substack{0 < u < \bar{t} \le v < t \\
	y,x \in \R^2}}
	\bsG_\theta(u,y) \, \frac{g_{\sfc(v-u)}(z-y)}{v-u} \, 
	\bsG_\theta(t-v, x-z) \, \dd u \, \dd v \qquad
	\forall t \in (\bar{t},\infty) \,.
\end{align}
These relations are proved in the same way as \eqref{eq:Uren} and \eqref{eq:Gren}.

\medskip
\noindent
\textbf{Step 3.}
We finally prove \eqref{eq:bsUNas}
by combining formulas \eqref{eq:Uren2}-\eqref{eq:Gren2}
with Lemma~\ref{th:weakGU2}. 
It suffices to fix arbitrary sequences $n = n_N \in \{1,\ldots, N\}$ and
$x = x_N \in \Z^2$ such that
\begin{equation}\label{eq:tN2}
	t_N := \frac{n_N}{N} \xrightarrow[\,N\to\infty\,]{} t \in (0,1] \,,
	\qquad w_N := \frac{x_N}{\sqrt{N}}
	\xrightarrow[\,N\to\infty\,]{} w \in \R^2 \,,
\end{equation}
and prove that
\begin{equation}\label{eq:UNasbis2}
	\lim_{N\to\infty} \, 
	\frac{N^{1+d/2}}{\log N} \, \bsU_{N,\lambda_N}(n_N, w_N)
	\,=\, \bsG_\theta(t,w)
	\,=\, G_\theta(t) \, g_{\sfc \theta}(w) \,.
\end{equation}

To prove \eqref{eq:UNasbis2}, we rewrite
the sums in \eqref{eq:Uren2} with $\bar n = \frac{n}{2}$ as integrals, recalling \eqref{eq:bsG^N}:
\begin{equation} \label{eq:quasifin2}
\begin{split}
	& \frac{N^{1+d/2}}{\log N} \, \bsU_{N,\lambda_N}(n_N,w_N) \\
	& = \lambda_N \, \!\!\!\!\!
	\int\limits_{\substack{0 \le u < \frac{t_N}{2} \le v \le 
	t_N \\ y,z \in \R^2}} \!\!\!\!\!
	\bsG_{\lambda_N}^{(N)}\big( \dd u, \dd y \big) \; 
	\bs\phi^{(N)}(u,v; \, y, z)
	\; \bsG_{\lambda_N}^{(N)}\big( t_N - \dd v, w_N - \dd z \big) \, ,
\end{split}
\end{equation}
where we set, for $0 \le u < v \le 1$ and $y,z \in \R^2$,
\begin{equation*} 
	\bs\phi^{(N)}(u,v; y,z) :=  N^{1+d/2} \, \log N \, \P\big( 
	T_1^{(N)} = \lfloor Nv\rfloor  - \lfloor Nu \rfloor, \,
	X_1^{(N)} = \lfloor \sqrt{N} z \rfloor  - \lfloor 
	\sqrt{N} y \rfloor \big) \,.
\end{equation*}
Note that by \eqref{eq:proppnx}, \eqref{eq:bsXN}
and \eqref{eq:XN}-\eqref{eq:RN} we have
\begin{equation} \label{eq:phiNconv2}
	\lim_{N\to\infty} \, \bs\phi^{(N)}(u,v; y,z) =
	\bs\phi(u,v; y,z) := \frac{g_{\sfc(v-u)}(z-y)}{v-u} \,.
\end{equation}
Moreover, by Lemma~\ref{th:weakGU2} and \eqref{eq:tN2} we have the 
convergence of the product measure
\begin{equation}\label{eq:weaco2}
	\bsG_{\lambda_N}^{(N)}\big( \dd u, \dd y \big)
	\, \bsG_{\lambda_N}^{(N)}\big( t_N - \dd v,
	w_N - \dd z \big) \,\xrightarrow[\,N\to\infty\,]{}\,
	\bsG_{\theta}(u,y) \, \bsG_\theta(t-v, w-z) \, \dd u \, 
	\dd y \, \dd v \,\dd z \,.
\end{equation}
Since $\lambda_N \to 1$ (see \eqref{eq:lambdaN}),
we expect by \eqref{eq:phiNconv2} and \eqref{eq:weaco2} that the right hand side of
\eqref{eq:quasifin2} converges to the right hand side of \eqref{eq:Gren2}
as $N\to\infty$, proving our goal \eqref{eq:UNasbis2}.

The difficulty
is that the function $\bs\phi^{(N)}(u,v; y,z)$
converges to a function $\bs\phi(u,v; y,z)$ which is singular as $v-u \to 0$,
see \eqref{eq:phiNconv2}.
This can be controlled as in the proof of Theorem~\ref{th:WR}, see the paragraphs
following \eqref{eq:weaco}.
\begin{itemize} 
\item First we fix $\epsilon > 0$ and restrict the
integral in \eqref{eq:quasifin2} to the domain $D_\epsilon = \{v-u \ge \epsilon \, t\}$.
Here we can apply the convergence
\eqref{eq:weaco2}, because $\bs\phi(u,v; y,z)$ is bounded and the convergence
$\bs\phi^{(N)}(u,v; y,z) \to \bs\phi(u,v; y,z)$ is uniform.

\item Then we consider the contribution to the integral in \eqref{eq:quasifin2} 
from $D_\epsilon^c = \{v-u < \epsilon \, t\}$. Recalling
\eqref{eq:bsG^N}, this contribution can be written as follows:
\begin{equation} \label{eq:isbou}
	\sum\limits_{\substack{u,v \in \frac{1}{N} \N_0, \
	y,z \in \frac{1}{\sqrt{N}}\Z^2 \\
	0 \le u < \frac{t_N}{2} \le v \le 
	t_N \,, \ v-u < \epsilon t}}  \!\!\!
	\frac{\bsU_{N,\lambda_N}\big( Nu, \sqrt{N}y \big)}{\log N} \; 
	\bs\phi^{(N)}(u,v; \, y, z)
	\; \frac{\bsU_{N,\lambda_N}\big( N(t_N - v), \sqrt{N}( w_N -  z) \big)}{\log N} \,.
\end{equation}
We need to show that this is small for $\epsilon > 0$ small, uniformly in large $N\in\N$.

By \eqref{eq:unibo2} we can bound, uniformly in $z \in \frac{1}{\sqrt{N}} \Z^2$,
\begin{equation*}
	\frac{\bsU_{N,\lambda_N}\big( N(t_N - v), \sqrt{N}( w_N -  z) \big)}{\log N} \le 
	C \, \frac{1}{N^{1+\frac{d}{2}}} \, \frac{1}{(t_N - v)^{\frac{d}{2}}} \,
	G_\theta\big(t_N - v \big) \,,
\end{equation*}
and note that $t_N - v \ge \frac{t_N}{2}-\epsilon$.
Next, by definition of $\bs\phi^{(N)}$ and by \eqref{eq:XN}-\eqref{eq:RN},
\begin{equation*}
	\sum_{z \in \frac{1}{\sqrt{N}}\Z^2}
	\bs\phi^{(N)}(u,v; \, y, z)
	= N^{1+ \frac{d}{2}} \, (\log N) \, \P\big( 
	T_1^{(N)} = \lfloor Nv\rfloor  - \lfloor Nu \rfloor \big)
	\le C_1 \, \frac{N^{\frac{d}{2}}}{v-u} \,.
\end{equation*}
Finally we observe that, by \eqref{UN}, \eqref{UN2} and \eqref{eq:unibo},
\begin{equation*}
	\sum_{y \in \frac{1}{\sqrt{N}}\Z^2} \frac{\bsU_{N,\lambda_N}\big( Nu, \sqrt{N}y \big)}{\log N}
	= \frac{U_{N,\lambda_N}( Nu )}{\log N}
	\le C \, \frac{1}{N} \, G_\theta(u) \,.
\end{equation*}
These bounds show that \eqref{eq:isbou} is bounded by a constant times
\begin{equation} \label{eq:dede}
	\frac{1}{N^2} \, \frac{1}{(\frac{t_N}{2} - \epsilon)^{\frac{d}{2}}}
	 \sum\limits_{\substack{u,v \in \frac{1}{N} \N_0 \\
	0 \le u < \frac{t_N}{2} \le v \le 
	t_N \,, \ v-u < \epsilon t }}  \!\!\!
	G_\theta(u) \; 
	\frac{1}{v-u}
	\; G_\theta(t_N-v) \,.
\end{equation}
Since $t_N \to t$, we have $\frac{t_N}{2} > \frac{t}{3}$ for $N$ large, and
if we take $\epsilon < \frac{t}{6}$ we see that the prefactor
$(\frac{t_N}{2} - \epsilon)^{-d/2} \le (\frac{t}{6})^{-d/2}$ is bounded
(recall that $t$ is fixed). The sum in \eqref{eq:dede} is the same as
that in \eqref{eq:dosu}, which we had shown to be small for
$\epsilon > 0$ small, uniformly in large $N\in\N$. This completes
the proof.\qed
\end{itemize}

\appendix

\medskip
\section{Additional results for disordered systems}
\label{sec:appds}

In this appendix we prove some results for disordered systems,
stated in Section~\ref{sec:main3}.

\subsection{Proof of relations~\eqref{eq:ZN2} and~\eqref{eq:ZNdp2}}
\label{sec:polychaos}

We recall the polynomial chaos expansion used in \cite{CSZ17a,CSZ17b}.
Let us introduce the random variables
\begin{equation} \label{eq:etasigma}
	\eta_i := \frac{e^{\beta \omega_i - \lambda(\beta)}}{\sigma_\beta} \,, \qquad
	\text{where} \qquad \sigma_\beta^2 := e^{\lambda(2\beta)-2\lambda(\beta)}-1 \,,
\end{equation}
so that $(\eta_i)$ are i.i.d.\ with zero mean and unit variance
(recall \eqref{eq:assomega}).

Recall the definition \eqref{eq:pinning} of $Z_{N}^{\beta}$ and
note that we can write
\begin{equation} \label{eq:1+}
	e^{(\beta \omega_n - \lambda(\beta)) \ind_{\{X_{2n}=0\}}}
	= 1 + \sigma_\beta \, \eta_n \, \ind_{\{X_{2n}=0\}} \,.
\end{equation}
We now write the exponential in \eqref{eq:pinning} as a product and perform
an expansion, exploiting \eqref{eq:1+}. Recalling the definition \eqref{eq:un}
of $u(n)$, we obtain:
\begin{equation} \label{eq:Zchaos}
\begin{split}
	Z_{N}^{\beta} &= \E\Bigg[ \prod_{n=1}^{N-1} 
	e^{(\beta \omega_n - \lambda(\beta)) \ind_{\{X_{2n}=0\}}}
	\, \ind_{\{X_{2N}=0\}}\Bigg] \\
	&= \sum_{k=1}^N \, (\sigma_\beta)^{k-1} 
	\sum_{0 < n_1 < \ldots < n_{k-1} < n_k := N}
	u(n_1) \, u(n_2-n_1) \, \cdots
	\, u(n_k - n_{k-1})  \\
	& \qquad\qquad\qquad\qquad\qquad\qquad\qquad\qquad\qquad 
	\cdot \eta_{n_1} \, \eta_{n_2} \, \cdots \, \eta_{n_{k-1}} \,  \,.
\end{split}
\end{equation}
This formula expresses $Z_{N}^{\beta}$ as a multilinear polynomial of the random variables.
Since the monomials
for different $k$ are orthogonal
in $L^2(\bbP)$, we get \eqref{eq:ZN2}.

The proof of \eqref{eq:ZNdp2} is similar, because we can represent $\bsZ_{N}^{\beta}(x)$
in \eqref{eq:dpre} as follows:
\begin{equation} \label{eq:Zchaosdpre}
\begin{split}
	\bsZ_{N}^{\beta}(x)
	= \sum_{k=1}^N \, (\sigma_\beta)^{k-1} 
	\!\!\!\!\!\!\!\!\!\!
	\sumtwo{0 < n_1 < \ldots < n_{k-1} < n_k := N}{x_1, \ldots, x_k \in \Z^2: \
	x_k = x} \!\!\!\!\!\!\!\!\!\!
	& q_{n_1}(x_1) \, q_{n_2-n_1}(x_2 - x_1) \, 
	\cdots \, q_{n_k - n_{k-1}}(x_k - x_{k-1}) 
	\\
	& \qquad \, \cdot
	\, \eta_{n_1,x_1} \, \eta_{n_2, x_2}
	\, \cdots \, \eta_{n_{k-1}, x_{k-1}} \,.
\end{split}
\end{equation}
This completes the proof.
\qed

\subsection{Free partition function}
\label{sec:unibo}

For the pinning model, one can consider the 
\emph{free partition function}
$Z_{N}^{\beta,\mathrm{f}}$, in which the constraint $\{X_{2N}=0\}$
is removed from \eqref{eq:pinning}, and the sum is extended up to $N$:
\begin{equation} \label{eq:pinningfree}
	Z_{N}^{\beta,\mathrm{f}} 
	:= \E\Big[ e^{\sum_{n=1}^{N} (\beta \omega_n - \lambda(\beta)) \ind_{\{X_{2n}=0\}}} \Big] \,.
\end{equation}
Then we have the following analogue of Theorem~\ref{th:pinning}.
Let us set, recalling \eqref{eq:G0}-\eqref{eq:G0+},
\begin{equation} \label{eq:barG}
	\overline{G}_\theta(u) := \int_0^u G_\theta(t) \, \dd t
	= \int_0^\infty 
	\frac{ e^{(\theta - \gamma) s} \,  
	u^{s}}{\Gamma(s+1)} \, \dd s \,,
	\qquad \text{for } u \in (0,1] \,. 
\end{equation}

\begin{proposition}[Free pinning model partition function] \label{th:pinning2}
Rescale
$\beta = \beta_N$ as in \eqref{eq:sigmaresc}.
Then, for any fixed $\delta > 0$,
the following relation holds as $N \to \infty$:
\begin{equation} \label{eq:pinningUNasfree}
	\bbE[(Z_{n}^{\beta_N, \mathrm{f}})^2] = (\log N)
	\, \overline{G}_{\theta}(\tfrac{n}{N}) \, (1+o(1)) \,,
	\qquad \text{uniformly for } \ \delta N \le n \le N \,,
\end{equation}
with $\overline{G}(\cdot)$ defined in \eqref{eq:barG}.
Moreover, the following bound holds, for a
suitable $C \in (0,\infty)$:
\begin{equation} \label{eq:pinningunibofree}
	\bbE[(Z_{n}^{\beta_N, \mathrm{f}})^2]
	\le C \, (\log N) \, \overline{G}_{\theta}(\tfrac{n}{N}) \,,
	\qquad \forall 1 \le n \le N \,.
\end{equation}
Finally, since $\bbE[Z_{n}^{\beta_N, \mathrm{f}}] = 1$,
relations \eqref{eq:pinningUNasfree} and \eqref{eq:pinningunibofree}
holds also for $\bbvar[Z_{n}^{\beta_N, \mathrm{f}}]$.
\end{proposition}

\begin{proof}
Arguing as in \S\ref{sec:polychaos}, one can write a decomposition 
for $Z_{n}^{\beta, \mathrm{f}}$ similar to \eqref{eq:Zchaos}.
As a consequence,
the second moment of $Z_{n}^{\beta, \mathrm{f}}$ 
is given by an expression similar to \eqref{eq:ZN2}, namely
\begin{equation} \label{eq:ZN2free}
	\bbE[(Z_{n}^{\beta, \mathrm{f}})^2] = 1 + \sum_{k \ge 1} \, (\sigma_\beta^2)^{k} 
	\sum_{0 < n_1 < \ldots < n_{k} \le n}
	u(n_1)^2 \, u(n_2-n_1)^2 \, \cdots \, u(n_{k} - n_{k-1})^2 \,,
\end{equation}
which yields an analogue of relation
\eqref{eq:repZN}:
\begin{equation*}
\begin{split}
	\bbE[(Z_{n}^{\beta, \mathrm{f}})^2] & = 
	1 + \sum_{k \ge 1} \big( \sigma_\beta^2 \, R_N \big)^{k} \,
	\P(\tau_{k}^{(N)} \le n)
	= 1 + \sum_{\ell=1}^n 
	\sum_{k \ge 1} \big( \sigma_\beta^2 \, R_N \big)^{k} \,
	\P(\tau_{k}^{(N)} = \ell) \\
	& = 1 \, + \, \sum_{\ell = 1}^n U_{N,\lambda}(\ell) \,, \qquad
	\text{where} \qquad \lambda := \sigma_\beta^2 \, R_N  \,.
\end{split}	
\end{equation*}
It then suffices to apply \eqref{eq:UNas} and \eqref{eq:unibo} to get
\eqref{eq:pinningUNasfree}
and \eqref{eq:pinningunibofree}. 
\end{proof}

\smallskip

Also for the directed polymer in random environment we can consider the 
\emph{free (or point-to-plane) partition function}
$\bsZ_{N}^{\beta, \mathrm{f}}$, in which the constraint $\{S_{N}=x\}$
is removed from \eqref{eq:dpre}, and the sum is extended up to $N$:
\begin{equation} \label{eq:dprefree}
	\bsZ_{N}^{\beta,\mathrm{f}}
	:= \E\Big[ e^{\sum_{n=1}^{N} (\beta \omega_{n,S_n} - \lambda(\beta)) } \Big] 
	= \E\Big[ e^{\sum_{n=1}^{N} \sum_{z\in\Z^2} 
	(\beta \omega_{n,z} - \lambda(\beta)) \ind_{\{S_n = z\}} }
	\Big] \,.
\end{equation}
The second moment of $\bsZ_{N}^{\beta, \mathrm{f}}$
turns out to be identical to that of $Z_{N}^{\beta, \mathrm{f}}$ (pinning model).

\begin{proposition}[Free directed polymer partition function] \label{th:dpre2}
Rescale $\beta = \beta_N$ as in \eqref{eq:sigmarescbis}.
Then relations \eqref{eq:pinningUNasfree} and \eqref{eq:pinningunibofree}
hold verbatim for the free partition function $\bsZ_{n}^{\beta_N, \mathrm{f}}$ of
the directed polymer in random environment, defined in \eqref{eq:dprefree}.
\end{proposition}

\begin{proof}
Arguing as in \S\ref{sec:polychaos}, one can write a decomposition 
for $\bsZ_{n}^{\beta, \mathrm{f}}$ similar to \eqref{eq:Zchaosdpre}.
Then the second moment of $\bsZ_{n}^{\beta, \mathrm{f}}$ can be represented as follows:
\begin{equation} \label{eq:ZNdp2free}
\begin{split}
	\bbE\big[(\bsZ_{n}^{\beta, \mathrm{f}})^2 \big] =
	1 + \sum_{k \ge 1} \, (\sigma_\beta^2)^{k} 
	\sumtwo{0 < n_1 < \ldots < n_{k} \le N}{x_1, \ldots, x_k \in \Z^2}
	& q_{n_1}(x_1)^2 \, q_{n_2-n_1}(x_2 - x_1)^2 \, \cdot \\
	& \qquad \cdots \, q_{n_k - n_{k-1}}(x_k - x_{k-1})^2 \,.
\end{split}
\end{equation}
Since $\sum_{x\in\Z^2} q_n(x)^2 = u(n)^2$, see \eqref{eq:P2n0},
we can sum over $x_k, x_{k-1}, \ldots, x_1$ in \eqref{eq:ZNdp2free}
to obtain precisely the same expression as in \eqref{eq:ZN2free}.
In other words, \emph{the free partition functions of the pinning and directed
polymer models have the same second moment}:
\begin{equation*}
	\bbE\big[(\bsZ_{n}^{\beta, \mathrm{f}})^2 \big] =
	\bbE\big[(Z_{n}^{\beta, \mathrm{f}})^2 \big]  \,.
\end{equation*}
This completes the proof.
\end{proof}

\subsection{Proof of Proposition~\ref{prop:EM}}
\label{sec:EM}
Let $T := \min\{m\in\N: \ S_{m} = 0\}$ denote
the first return time to the origin of the simple symmetric random walk on $\Z^2$. 
Let $(\xi_i)_{i\in\N}$ be i.i.d.\
random variables distributed as $T/2$. We define
\begin{equation*}
	L_N := \sum_{n=1}^N \ind_{\{S_{2n}=0\}} 
	= \max\big\{k\in\N_0: \ \xi_1 + \ldots + \xi_k \le N \big\} \,,
\end{equation*}
so that, recalling \eqref{eq:P2n0} and the definition
\eqref{eq:EM} of $R_N$, we can write
\begin{equation*}
	R_N = \sum_{n=1}^N \P(S_{2n}=0)
	= \E[L_N] = \sum_{k=1}^N \P(L_N \ge k)
	= \sum_{k=1}^N \P(\xi_1 + \ldots + \xi_k \le N) \,.
\end{equation*}
Let $(\xi_i^{(N)})_{i\in\N}$ be
i.i.d.\ random variables with the law of
$\xi_1$ conditionally on $\{\xi_1 \le N\}$. Then
we have the following key representation of $R_N$:
\begin{equation}\label{eq:formi}
\begin{split}
	R_N & 
	= \sum_{k=1}^N \P(\xi_1 \le N)^k \, \P(\xi_1^{(N)} + \ldots + \xi_k^{(N)} \le N) \\
	& = \sum_{k=1}^N \P(\xi_1 \le N)^k 
	- \sum_{k=1}^N \P(\xi_1 \le N)^k
	\, \P(\xi_1^{(N)} + \ldots + \xi_k^{(N)} > N) \,.
\end{split}
\end{equation}
We are going to show that the first sum gives the
leading contribution to the right hand side of \eqref{eq:EM},
while the second sum is negligible.

\smallskip

We need estimates on the law of $\xi_1$.
By Corollary~1.2 and Remark~4 in~\cite{U11}, we have
\begin{equation}\label{returndist}
\begin{aligned}
\P(\xi_1=k) = \P(T=2k) & = \frac{\pi}{k} 
\bigg( \frac{1}{(\log 16k)^2} - \frac{2\gamma}{(\log 16k)^3}  
+ O\bigg(\frac{1}{(\log 16k)^4}\bigg) \bigg)  \\
& = \frac{\pi}{k(\log k)^2} - \frac{2\pi(\gamma + \log 16)}{k(\log k)^3}
+ O\bigg(\frac{1}{(\log k)^4}\bigg) \,, \\
\P(\xi_1\geq k) = \P(T\geq 2k) & 
= \frac{\pi}{\log k} 
- \frac{\pi(\gamma+\log 16)}{(\log k)^2} + O\bigg(\frac{1}{(\log k)^3}\bigg) \,,
\end{aligned}
\end{equation}
as $k\to\infty$,
where $\gamma$ is the Euler-Mascheroni constant.
Then, as $N\to\infty$, we can write
\begin{equation*}\label{eq:precest}
\begin{split}
	\frac{\P(\xi_1 \le N)}{\P(\xi_1 > N)} & = \frac{1 - \frac{\pi}{\log N} + O(\frac{1}{(\log N)^2})}
	{\frac{\pi}{\log N} \big(1 - \frac{(\gamma + \log 16)}{(\log N)} 
	+ O\big(\frac{1}{(\log N)^2}\big) \big)} =
	\frac{\log N}{\pi} + 
	\bigg(\frac{\gamma + \log 16}{\pi}-1\bigg) + o(1) \,, \\
	\P(\xi_1 \le N)^N & = \big( 1 - \tfrac{\pi}{\log N} + O(\tfrac{1}{(\log N)^2}) \big)^N =
	e^{-\frac{\pi N}{\log N} (1+o(1))} = o\bigg(\frac{1}{\log N}\bigg) \,.
\end{split}
\end{equation*}
From this we deduce the asymptotic behavior
of the first sum in the last line of \eqref{eq:formi}:
\begin{equation*}
	\sum_{k=1}^N \P(\xi_1 \le N)^k = 
	\frac{\P(\xi_1 \le N)}{\P(\xi_1 > N)} \big(1-\P(\xi_1\le N)^N\big)
	= \frac{\log N}{\pi} + 
	\bigg(\frac{\gamma + \log 16}{\pi}-1\bigg) + o(1) \,,
\end{equation*}
which matches with the right hand side of \eqref{eq:EM}.
It remains to show that the second sum
in the last line of \eqref{eq:formi} is asymptotically vanishing, i.e.\
\begin{equation}\label{eq:toshow}
	\lim_{N\to\infty} \rho_N = 0 \,, \qquad
	\text{where} \qquad \rho_N :=
	\sum_{k=1}^N \P(\xi_1 \le N)^k \, \P(\xi_1^{(N)} + \ldots + \xi_k^{(N)} > N) \,.
\end{equation}

\smallskip

Denoting by $C_1, C_2$ suitable absolute constants,
we have by relation \eqref{returndist}
\begin{equation} \label{eq:T1N}
\begin{split}
	\E\Big[\xi_1^{(N)} \Big] 
	& = \frac{1}{\P(\xi_1 \le N)} \sum_{\ell=1}^{N} \ell \, \P(\xi_1 = \ell)
	\le C_1 \sum_{\ell=1}^{N} \frac{1}{(\log \ell)^2}
	\le C_2 \frac{N}{(\log N)^2} \,,
\end{split}
\end{equation}
hence by Markov's inequality
\begin{equation*}
	\P\big( \xi_1^{(N)} + \ldots + \xi_k^{(N)} > N \big)
	\le C_2 \frac{k}{(\log N)^2} \, .
\end{equation*}
Since $\P(\xi_1 \le N) \le e^{-\frac{1}{\log N}}$ for large $N$,
by \eqref{returndist}, we can control the tail of $\rho_N$ in \eqref{eq:toshow} by
\begin{equation*}
	\rho_N^{> A} :=
	\sum_{k > A \log N} \P(\xi_1 \le N)^k \P(\xi_1^{(N)} + \ldots + \xi_k^{(N)} > N)
	\le C_2 \sum_{k > A \log N} e^{-\frac{k}{\log N}} 
	\frac{k}{(\log N)^2} \,.
\end{equation*}
By a Riemann sum approximation, the last sum converges to
$\int_A^\infty x \, e^{-x} \, \dd x = (1+A) e^{-A}$ as $N\to\infty$.
In particular, for every fixed $A \in (0,\infty)$, we have shown that
\begin{equation} \label{eq:RNA>}
	\limsup_{N\to\infty} \rho_N^{> A} \le (1+A)e^{-A} \,.
\end{equation}

Next we focus on the contribution $\rho_N^{\le A}$ of the terms
with $k \le A \log N$, i.e.\
\begin{equation} \label{eq:rec1}
\begin{split}
	\rho_N^{\le A} & :=
	\sum_{k \le A \log N} \P(\xi_1 \le N)^k \, \P(\xi_1^{(N)} + \ldots + \xi_k^{(N)} > N) \\
	& \le (A \log N) \, \P(\xi_1^{(N)} + \ldots + \xi_{A \log N}^{(N)} > N) \,.
\end{split}
\end{equation}
We fix $\epsilon \in (0,\frac{1}{2})$ and write
\begin{equation*}
	\xi_1^{(N)} + \ldots + \xi_k^{(N)} =
	\sum_{i=1}^k \xi_i^{(N)} \ind_{\{\xi_i^{(N)} \le \epsilon^2 N\}}
	+ \sum_{i=1}^k \xi_i^{(N)} \ind_{\{\xi_i^{(N)} > \epsilon^2 N\}}
	=: U_{-} + U_{+} \,,
\end{equation*}
so that we can decompose
\begin{equation} \label{eq:rec2}
	\P(\xi_1^{(N)} + \ldots + \xi_k^{(N)} > N) \le
	\P( U_{-} > \epsilon N ) +
	\P( U_{+} > (1-\epsilon) N ) \,,
\end{equation}
and we estimate separately each term. In analogy with \eqref{eq:T1N} we have
\begin{equation*}
\begin{split}
	\E\Big[U_{-} \Big] = k
	\E\Big[\xi_1^{(N)} \ind_{\{\xi_1^{(N)} \le \epsilon^2 N\}}\Big] 
	& = k \sum_{\ell=1}^{\epsilon^2 N} \frac{\ell \, \P(\xi_1 = \ell)}{\P(\xi_1 \le N)}
	\le k \sum_{\ell=1}^{\epsilon^2 N} \frac{C_1}{(\log \ell)^2}
	\le C_2 \frac{\epsilon^2 N k}{(\log (\epsilon^2 N))^2} \,,
\end{split}
\end{equation*}
hence by Markov's inequality
\begin{equation} \label{eq:rec3}
	\P\big( U_{-} > \epsilon N \big)
	\le C_2 \frac{\epsilon k}{(\log (\epsilon^2 N))^2} \,.
\end{equation}
Next we observe that 
\begin{equation*}
	\{ U_{+} > (1-\epsilon) N \} \subseteq
	\bigg( \bigcup_{i=1}^k \{\xi_i^{(N)} > (1-\epsilon) N\} \bigg) \cup
	\bigg( \bigcup_{ 1 \le i < j \le k} \{\xi_i^{(N)} > \epsilon^2 N,
	\xi_j^{(N)} > \epsilon^2 N\} \bigg) \,,
\end{equation*}
because either $\xi_i^{(N)} > (1-\epsilon) N$ for a single $i$, or necessarily
$\xi_i^{(N)} > \epsilon^2 N$ and $\xi_j^{(N)} > \epsilon^2 N$ for at least two
distinct $i \ne j$ (otherwise $U_+$ vanishes).
Since for fixed $c \in (0,1)$
\begin{equation*}
	\P(\xi_1^{(N)} > c N) \le C_1 \sum_{\ell = cN}^N \frac{1}{\ell \, (\log \ell)^2}
	\le C_1 \frac{1}{(\log cN)^2} \sum_{\ell = cN}^N \frac{1}{\ell}
	\le C_1 \frac{\log \frac{1}{c}}{(\log cN)^2} \,,
\end{equation*}
it follows that
\begin{equation*}
	\P(U_+ > (1-\epsilon)N)  \le k \,
	C_1 \frac{\log \frac{1}{1-\epsilon}}{(\log ((1-\epsilon)N))^2}
	+ \frac{k(k-1)}{2} \,
	\bigg[ C_1 \frac{\log \frac{1}{\epsilon^2}}{(\log (\epsilon^2 N))^2} \bigg]^2 \,.
\end{equation*}
Recalling \eqref{eq:rec1}-\eqref{eq:rec2}-\eqref{eq:rec3} and plugging $k = A \log N$, we get
\begin{equation*}
	\limsup_{N\to\infty} \rho_N^{\le A} \le A^2 \big( C_2 \, \epsilon +
	C_1 \log \tfrac{1}{1-\epsilon} \big) \,.
\end{equation*}
By \eqref{eq:RNA>}, since $\rho_N = \rho_N^{\le A} + \rho_N^{> A}$,
we obtain \eqref{eq:toshow} by letting $\epsilon \to 0$
and then $A \to \infty$.\qed

\subsection{Explicit asymptotics in terms of $\beta$}
\label{sec:explbeta}

Relation \eqref{eq:sigmaresc} (equivalently \eqref{eq:sigmarescbis})
and relation \eqref{eq:sigmaresc2} can be rewritten more explicitly
in terms of $\beta_N$. To this purpose, we need
the \emph{cumulants} $\kappa_3$, $\kappa_4$ of the distribution of $\omega_i$
(recall \eqref{eq:assomega}), defined by
\begin{equation} \label{eq:cumulants}
	\lambda(\beta) = \frac{1}{2} \beta^2 + \frac{\kappa_3}{3!} \beta^3
	+ \frac{\kappa_4}{4!} \beta^4 + O(\beta^5) \qquad \text{as } \beta \to 0 \,.
\end{equation}
By direct computation
$\sigma_\beta^2 = 
	\beta^2 + \kappa_3 \, \beta^3 + 
	\big( \tfrac{1}{2} + \tfrac{7}{12} \kappa_4\big) \beta^4 + O(\beta^5)$
as $\beta \to 0$, hence
\begin{equation} \label{eq:betaepsilon}
	\sigma_\beta^2 = \epsilon \quad \ \Longrightarrow \quad \
	\beta^2 = \epsilon - \kappa_3 \, \epsilon^{3/2} + 
	(\tfrac{3}{2} \kappa_3^2 - \tfrac{7}{12}\kappa_4 - \tfrac{1}{2}) \, \epsilon^2
	+ o(\epsilon^2) \qquad \text{as } \epsilon \to 0 \,.
\end{equation}
As a consequence, 
we can rewrite \eqref{eq:sigmaresc2} as follows,
with $\alpha := \gamma + \log 16 - \pi$:
\begin{equation*}
	\beta_N^2 = \frac{\pi}{\log N} - \frac{\kappa_3 \, \pi^{3/2}}{(\log N)^{3/2}}
	+ \frac{\pi(\theta - \alpha) 
	+ \pi^2 (\frac{3}{2} \kappa_3^2 - \frac{1}{2} - \frac{7}{12}\kappa_4)}
	{(\log N)^2} \big(1+o(1)\big)  \,.
\end{equation*}

\medskip

\section{On the Dickman subordinator}
\label{sec:Dickman}

Theorem~\ref{th:scalingY}
on the density of the Dickman subordinator can be deduced
from general results about \emph{self-decomposable L\'evy processes},
see \cite[\S53]{cf:Sato}.
\begin{itemize}
\item Let us first derive \eqref{eq:scalingf} for $t \in (0,1]$.
The law of $Y_s$
satisfies the assumptions of \cite[Lemma~53.2]{cf:Sato}
with $n=1$, $a_1 = 1$ and $c = s$, which yields
$f_s(t) = K t^{s-1}$ for $t \in (0,1]$.
To show that $K = e^{-\gamma s}/\Gamma(s)$, as in \eqref{eq:scalingf},
one can apply \cite[Theorem~53.6]{cf:Sato} which gives 
$f_s(t) = (1+o(1)) \kappa \, t^{s-1} / \Gamma(s)$ as $t \downarrow 0$,
with $\kappa = \exp\{ s (\int_0^1 \frac{e^{-x}-1}{x} \dd x + \int_1^\infty 
\frac{e^{-x}}{x} \, \dd x )\}$. The identification $\kappa = \exp\{-\gamma s\}$
follows by \cite[Entry 8.367 (12), page 906]{GR}.

\item We then deduce \eqref{eq:scalingf} for $t \in (1,\infty)$.
We can apply \cite[Theorem~51.1]{cf:Sato}, which reads as follows
(where $\nu(\dd t) = \frac{s}{t} \, \ind_{(0,1)}(t) \, \dd t$, $\gamma_0 = 0$
and $f_s(t)$ is the density of $Y_s$):
\begin{equation*}
	\int_0^t y \, f_s(y) \, \dd y = \int_0^t \left(
	\int_0^{t-y} f_s(u) \, \dd u \right) y \, \frac{s}{y} \,
	\ind_{(0,1)}(y) \, \dd y \,.
\end{equation*}
Differentiating with respect to $t$, for $t > 1$, we get
$t f_s(t) = s \int_0^1 f_s(t-y) \, \dd y$, which already shows that
$f_s(t)$ can be deduced from $\{f_s(u): \, u \in (t-1,t)\}$.
To obtain \eqref{eq:scalingf},
we further differentiate
this relation (note
that $f_s(\cdot) \in C^1$ on $(1,\infty)$, by \cite[Lemma~53.2]{cf:Sato})
to get $f_s(t) + t f_s'(t) = s \, (f_s(t) - f_s(t-1))$, which can be rewritten
as $(t^{1-s} f_s(t))' = -s \, t^{-s} \, f_s(t-1)$. 
Integrating on $(0,t)$, 
since $t^{1-s} f_s(t) \to K = e^{-\gamma s}/\Gamma(s)$ as $t \downarrow 0$,
we obtain $t^{1-s} f_s(t) - K = s \int_0^t \frac{f_s(u-1)}{u^s} \, \dd u$,
which coincides with the second line of \eqref{eq:scalingf}
(note that $f_s(t) \equiv 0$ for $t < 0$).
\end{itemize}
This completes the proof of \eqref{eq:scalingf}.\footnote{This proof was
kindly provided to us by Thomas Simon.}

\smallskip

We now present an alternative proof of Theorem~\ref{th:scalingY},
which exploits a key \emph{scale invariance property} of the Dickman subordinator $Y$.
Let $M_s$ denote the maximal jump  up to time $s$:
\begin{equation} \label{eq:MDeltaY}
	M_s := \max_{u \in (0,s]} \Delta Y_u \,, \qquad
	\text{where} \qquad \Delta Y_u := Y_u - Y_{u-} = Y_u - \lim_{\epsilon \downarrow 0}
	Y_{u-\epsilon} \,.
\end{equation}
We first prove the following result.

\begin{proposition}[Scale-invariance]
\label{th:scale}
Fix $s \in (0,\infty)$, $t \in (0,1)$. Conditional on all jumps of $Y$
up to time $s$ being smaller than $t$, the random variable $Y_s / t$
has the same law as $Y_s$, i.e.
\begin{equation} \label{eq:scale}
	\P\bigg( \frac{Y_s}{t} \in \cdot \,\bigg|\, M_s < t\bigg) = \P(Y_s \in \cdot) \,.
\end{equation}
\end{proposition}

\begin{proof}

We use the standard representation of the L\'evy process $Y = (Y_s)_{s\in[0,\infty)}$ 
in terms of a Poisson Point Process (PPP).
Let $\Pi$ be a PPP on $[0,\infty) \times (0,1)$ with intensity measure
\begin{equation} \label{eq:muu}
	\mu(\dd x, \dd y) := \mathrm{Leb}(\dd x) \otimes \nu(\dd y)
	= \dd x \otimes \frac{\ind_{(0,1)}(y) }{y} \, \dd y \,.
\end{equation}
We recall that $\Pi$ is a random countable subset
of $[0,\infty) \times (0,1)$, whose points we denote by $(s_i, t_i)$.
Let us define
\begin{equation} \label{eq:YPi0}
	\Pi^{(s, t)} := \Pi \cap ([0, s] \times (0,  t)) \,, \qquad
	Y^{(t)}_s := \sum_{(s_i,t_i) \in \Pi^{(s,t)}} t_i \,.
\end{equation}
Then we can represent our L\'evy process $Y_s$ in terms of $\Pi$ as follows:
\begin{equation}\label{eq:YPi}
	Y_s \,\overset{d}{=}\, Y^{(1)}_s \,.
\end{equation}
Let us identify $Y_s$ with $Y^{(1)}_s$. Note that
$\Delta Y_s = t \ne 0$ if and only if $(s,t) \in \Pi$, see \eqref{eq:MDeltaY}.

On the event
$\{M_s < t\} = \{\Pi\cap ([0, s] \times [t,1)) = \emptyset\}$
we have $Y_s^{(1)} = Y^{(t)}_s$, hence
\begin{equation*}
	\P\bigg( \frac{Y_s}{t} \in \cdot \,\bigg|\, M_s < t\bigg)
	= \P\bigg( \frac{Y^{(t)}_s}{t} \in \cdot \,\bigg|\, \Pi\cap ([0, s] \times [t,1)) 
	= \emptyset\bigg) = \P\bigg( \frac{Y^{(t)}_s}{t} \in \cdot \bigg) \,,
\end{equation*}
because 
$Y^{(t)}_s$ is a function of $\Pi^{(s, t)}$,
which is independent of $\Pi\cap ([0, s] \times [t,1))$,
by definition of PPP.
To prove our goal \eqref{eq:scale}, it remains to show that
\begin{equation*}
	\P\bigg( \frac{Y^{(t)}_s}{t} \in \cdot \bigg) 
	= \P\big( Y^{(1)}_s \in \cdot \big) \,.
\end{equation*}
By \eqref{eq:YPi0}, it suffices to prove the following property:
if we denote by $\phi_t: \R^2 \to \R^2$ the map $(x,y) \mapsto (x, \frac{1}{t}y)$,
then the random set
\emph{$\phi_t(\Pi^{(s,t)})$ has the same law as $\Pi^{(s,1)}$}.

\smallskip

Note that $\Pi^{(s, t)}$
is a PPP with intensity measure $\mu^{(s,t)}$ given by the
original intensity measure $\mu$
restricted on $[0, s] \times (0,  t)$ (see \eqref{eq:muu}).
We also observe that the random set
$\phi_t(\Pi^{(s, t)})$ is a PPP with intensity measure given by 
$\mu^{(s,t)} \circ \phi_t^{-1}$, i.e.\
the image law of $\mu^{(s,t)}$ under $\phi_t$.
The proof is completed by noting that $\phi_t$
sends $\mu^{(s,t)}$ to $\mu^{(s,1)}$,
because the map $y\mapsto y /  t$ sends the
measure $\frac{1}{y} \, \ind_{(0, t)}(y) \, \dd y$ to the measure
$\frac{1}{y} \, \ind_{(0,1)}(y) \, \dd y$.
\end{proof}

In our proof of Theorem~\ref{th:scalingY}, we will also need the following estimate.
This can be deduced from \cite[Lemma~6]{cf:RW},
but we give a direct proof in our setting.

\begin{lemma} \label{th:prellem}
As $s \downarrow 0$ we have
\begin{equation} \label{eq:tosh32}
	\P(Y_s > 1) = o(s) \,.
\end{equation}
\end{lemma}

\begin{remark}
The bound \eqref{eq:tosh32} is an intermediate step in establishing
Theorem~\ref{th:scalingY} and it is not optimal. Indeed, it is a consequence
of Theorem~\ref{th:scalingY} that the optimal estimate is
\begin{equation} \label{eq:asrip}
	\P(Y_s > 1) = O(s^2) \qquad \text{as } s \downarrow 0 \,,
\end{equation}
because $\P(Y_s \le 1) = e^{-\gamma s} / \Gamma(s+1)$, by \eqref{eq:scalingf},
and we note that as $s \downarrow 0$ we have
\begin{equation} \label{eq:Gammadev}
	\Gamma(s+1) = \Gamma(1) + \Gamma'(1) s + O(s^2) 
	 = 1 - \gamma s + O(s^2) \,,
\end{equation}
since $\Gamma'(1) = \int_0^\infty \log u \, e^{-u} \, \dd u = - \gamma$. 
Relation \eqref{eq:asrip} then follows.
\end{remark}

\begin{proof}[Proof of Lemma~\ref{th:prellem}]
Fix a function $\alpha_s \to \infty$ as $s \to 0$, to be determined later.
Recall the definition \eqref{eq:MDeltaY} of $\Delta Y_u = Y_u - Y_{u-}$ and define
\begin{equation*}
	N_s \, := \, \sum_{u \in (0,s]} \ind_{\{\Delta Y_u > \frac{1}{\alpha_s}\}} 
	\, = \,
	\text{number of jumps of $Y$ of size $> \tfrac{1}{\alpha_s}$ 
	in the interval $(0,s]$} \,.
\end{equation*}
We recall that $Y$ only increases by jumps, that is
$Y_s = \sum_{u\in (0,s]} \Delta Y_u$. We denote by $Y_s^>$
the contribution to $Y_s$ given by
jumps of size $> \frac{1}{\alpha_s}$, and $Y_s^{\le} := Y_s - Y_s^>$.
Then we bound
\begin{equation} \label{eq:3deco}
	\P(Y_s > 1) 
	\,\le\, \P(N_s \ge 2) + \P(N_s = 1, Y_s > 1) + \P(N_s = 0, Y_s^{\le} > 1)
\end{equation}

For the first term, we note that
$N_s \sim \mathrm{Pois}(\lambda_s)$ with
$\lambda_s = s \int_{1/\alpha_s}^1 \frac{1}{x} \, \dd x
= s \log \alpha_s$,
hence
\begin{gather*}
	\P(N_s \ge 2) = O(\lambda_s^2) = O(s^2 (\log \alpha_s)^2) \,.
\end{gather*}
For the third term, since
$(Y_s^{\le})_{s\ge 0}$ has L\'evy measure
$\frac{1}{x} \, \ind_{(0, \frac{1}{\alpha_s})} (x) \, \dd x$, we can bound
\begin{equation} \label{eq:canbo}
	\P(Y_s^{\le} > 1) \le \E[Y_s^\le] 
	= s \int_0^{\frac{1}{\alpha_s}} x \, \frac{1}{x} \, \dd x
	= \frac{s}{\alpha_s} \,.
\end{equation}
We fix $\alpha_s = 1/s$, so that both $\P(N_s \ge 2)$
and $\P(Y_s^{\le} > 1)$ are $O(s^{3/2})$.

It remains to estimate the second term in the right hand side of \eqref{eq:3deco}.
On the event $\{N_s =1\}$, the random variable
$W :=Y_s^>$ has density
$\frac{1}{\log \alpha_s} \, \frac{1}{x} \, \ind_{(\frac{1}{\alpha_s},1)}(x)$. 
Also note that $Y_s^\le$ is independent of $N_s$.
If we fix $\rho_s \in (1,2)$, to be determined later,
we can write
\begin{equation*}
\begin{split}
	\P(N_s = 1, Y_s > 1) & \le \P(N_s = 1, Y_s^> > \tfrac{1}{\rho_s}) + \P(N_s=1,
	\ Y_s^> \le \tfrac{1}{\rho_s}, \  Y_s^\le > 1-\tfrac{1}{\rho_s} )\\
	& \le \P(N_s = 1) \big\{ \P(W > \tfrac{1}{\rho_s}) +
	\P(Y_s^\le > \tfrac{\rho_s-1}{\rho_s} ) \big\} \\
	& \le \lambda_s \bigg\{  \frac{\log \rho_s}{\log \alpha_s} +
	\frac{\rho_s}{\rho_s-1} \E[Y_s^\le] \bigg\} \,,
\end{split}
\end{equation*}
because $N_s \sim \mathrm{Pois}(\lambda_s)$.
Since $\lambda_s = s \log \alpha_s$ and $\E[Y_s^\le] = \frac{s}{\alpha_s}$,
see \eqref{eq:canbo}, we get
\begin{equation*}
	\P(N_s = 1, Y_s > 1) \le 
	s \log \alpha_s \bigg\{ \frac{\log \rho_s}{\log \alpha_s} +
	\frac{2s}{\alpha_s (\rho_s-1)} \bigg\} 
	= s \log \rho_s + \frac{\log \alpha_s}{\alpha_s} \,
	\frac{2 s^2}{\rho_s-1}
\end{equation*}
Note that $\lim_{s\to 0} \frac{\log \alpha_s}{\alpha_s} = 0$,
because we have fixed $\alpha_s = 1/s$.
We now choose $\rho_s = 1 + \sqrt{s}$ to get
$\P(N_s = 1, Y_s > 1) = O(s^{3/2})$, which completes the proof.
\end{proof}

\smallskip

\begin{proof}[Proof of Theorem \ref{th:scalingY}]

We start proving the first line of \eqref{eq:scalingf},
so we assume $t\in (0,1)$.

Recall that $M_s$ was defined in \eqref{eq:MDeltaY}.
Plainly, we can write
\begin{equation*}
	\P(Y_s \le t) = \P(Y_s \le t, \, M_s < t)
	= \P(M_s < t) \, \P(Y_s \le t \,|\, M_s < t) \,.
\end{equation*}
We use the PPP representation of $Y_s$ that we introduced in the
proof of Proposition~\ref{th:scale}. In particular, if $\Pi$
denotes a PPP with intensity measure $\mu$ in \eqref{eq:muu}, we can write
\begin{equation*}
	\P(M_s < t) = \P(\Pi \cap ([0,s] \times [t,1)) = \emptyset)
	= e^{-\mu([0,s] \times [t,1))} = e^{-s \int_t^1 \frac{1}{y} \, \dd y} = t^s \,.
\end{equation*}
For $t \in (0,1)$ we have 
$\P(Y_s \le t \,|\, M_s < t) = \P(Y_s \le 1)$, by Proposition~\ref{th:scale}, hence
\begin{equation}\label{eq:firstprove}
	\P(Y_s \le t) =  t^s \, \P(Y_s \le 1) \qquad \text{for } t \in (0,1) \,.
\end{equation}
This leads to
\begin{equation} \label{eq:almostden}
	f_s(t) = s \, t^{s-1} \, F_s(1) \qquad \text{for } t \in (0,1) \,, \qquad 
	\text{where} \qquad F_s(t) := \P(Y_s \le t).
\end{equation}

It remains to identify $F_s(1)$.
Since $(Y_s)_{s\geq 0}$ has stationary and independent increments,
for any $n\in\N$, the density $f_s$ is the convolution of $f_{s/n}$
with itself $n$ times. Then for any $t \in (0,1)$ we can write, by \eqref{eq:almostden},
\begin{equation*}
\begin{split}
	f_s(t) & = \int\limits_{0 < t_1 < \ldots < t_{n-1} < t}
	f_{\frac{s}{n}}(t_1) \, f_{\frac{s}{n}}(t_2-t_1) \, \cdots f_{\frac{s}{n}}(t-t_{n-1})
	\, \dd t_1 \ldots \dd t_{n-1} \\
	& = \big( \tfrac{s}{n} \, F_{\frac{s}{n}}(1) \big)^n
	\int\limits_{0 < t_1 < \ldots < t_{n-1} < t} t_1^{\frac{s}{n}-1}
	\, (t_2-t_1)^{\frac{s}{n}-1} \cdots (t-t_{n-1})^{\frac{s}{n}-1}
	\, \dd t_1 \ldots \dd t_{n-1} \\
	& = \big( \tfrac{s}{n} \, F_{\frac{s}{n}}(1) \big)^n
	\, t^{s-1} \,
	\int\limits_{0 < u_1 < \ldots < u_{n-1} < 1} u_1^{\frac{s}{n}-1}
	\, (u_2-u_1)^{\frac{s}{n}-1} \cdots (1-u_{n-1})^{\frac{s}{n}-1}
	\, \dd u_1 \ldots \dd u_{n-1} \\
	& = \big( \tfrac{s}{n} \, F_{\frac{s}{n}}(1) \big)^n
	\, t^{s-1} \, \frac{\Gamma(\frac{s}{n})^n}{\Gamma(s)}
	 = \big( F_{\frac{s}{n}}(1) \big)^n
	\, t^{s-1} \, \frac{\Gamma(1+\frac{s}{n})^n}{\Gamma(s)} \,,
\end{split}
\end{equation*}
where we recognized the density of the Dirichlet
distribution (with parameters $n$ and $\frac{s}{n}$) and, in the last step,
we used the property $\Gamma(1+x) = x \, \Gamma(x)$. 
By \eqref{eq:Gammadev}
\begin{equation*}
	\Gamma(1+\tfrac{s}{n})^n \, \xrightarrow[n\to\infty]{} \,
	e^{-\gamma \, s} \,.
\end{equation*}
Since $F_u(1) = 1 - o(u)$ as $u \to 0$, by Lemma~\ref{th:prellem},
we have $\big(F_{\frac{s}{n}}(1) \big)^n \to 1$. This yields
\begin{equation*}
	f_s(t) = \lim_{n\to\infty}
	\big( F_{\frac{s}{n}}(1) \big)^n
	\, t^{s-1} \, \frac{\Gamma(1+\frac{s}{n})^n}{\Gamma(s)}
	= \frac{t^{s-1} \, e^{-\gamma \, s}}{\Gamma(s)}
	= \frac{s \, t^{s-1} \, e^{-\gamma \, s}}{\Gamma(s+1)} \,,
\end{equation*}
which proves the first line of \eqref{eq:scalingf}.

\smallskip

It remains to prove the second line of \eqref{eq:scalingf}.
We exploit the PPP construction of 
$Y_s$, see \eqref{eq:muu}-\eqref{eq:YPi}.
By identifying the largest jump $M_s = u$,
see \eqref{eq:MDeltaY}, we have for any $t \in (0,\infty)$
\begin{equation}\label{eq:fst1}
\begin{split}
	\P(Y_s \in \dd t) & = \int_0^{t \wedge 1} \P(Y_s \in \dd t \,|\, M_s = u) \,
	\P(M_s \in \dd u) \\
	& = \int_0^{t\wedge 1} \Big\{ \tfrac{1}{u} \, f_s\big(\tfrac{t-u}{u}\big)
	\, \dd t \Big\} \,
	\Big\{ \tfrac{s}{u} \, e^{-s\int_u^1 \tfrac{\dd x}{x}} \, \dd u \Big\} \\
	& = \bigg(
	\int_0^{t\wedge 1} f_s\big(\tfrac{t-u}{u}\big)  \, s \, u^{s-2} \, \dd u \bigg) \,
	\dd t \,.
\end{split}
\end{equation}
The second equality holds for the following reasons.
\begin{itemize}

\item $Y_s$ conditioned on $\{M_s < u\}$ has the same law as $u Y_s$,
by Proposition~\ref{th:scale}, hence
\begin{equation*}
	\P(Y_s \in \dd t \,|\, M_s = u)
	= \P(Y_s \in \dd t - u \,|\, M_s < u)
	= \tfrac{1}{u} \, f_s\big(\tfrac{t-u}{u}) \, \dd u \,.
\end{equation*}

\item $\tfrac{s}{u}$ is the Poisson intensity 
of finding a jump of size $u$ in the time interval $[0,s]$,
while $e^{-s\int_u^1 \frac{\dd x}{x}}=u^s$ is the probability that all other jumps 
are smaller than $u$, hence
\begin{align*}
	& \P(M_s \in \dd u) = \mu([0,s] \times \dd u)
	\, e^{-\mu([0,s] \times (u,1))} = 
	\tfrac{s}{u} \, \dd u \, e^{- s \int_u^1 \frac{1}{x} \, \dd x} \,.
\end{align*}
\end{itemize}

Making the change of variable $a:=\frac{t-u}{u}$, we can rewrite \eqref{eq:fst1} as
\begin{equation}
\begin{aligned}
	f_s(t) & = s\, t^{s-1} \int_{(t-1)^+}^\infty \frac{f_s(a)}{(1+a)^s} \dd a \\
	& = s\,t^{s-1} \bigg(\int_0^\infty \frac{f_s(a)}{(1+a)^s} \dd a - \int_0^{(t-1)^+}	
	\frac{f_s(a)}{(1+a)^s} \dd a \bigg).
\end{aligned}
\end{equation}
For $t\in (0,1)$, the second integral equals $0$, while $f_s(t)
= \frac{s\, t^{s-1} \, e^{-\gamma \, s}}{\Gamma(s+1)}$
by the first line of \eqref{eq:scalingf}, that we have already proved.
This implies that the first integral must equal $\frac{e^{-\gamma \, s}}{\Gamma(s+1)}$. 
This concludes the proof of the second line of \eqref{eq:scalingf}. 
\end{proof}

\medskip
\section*{Acknowledgements}

We are very grateful to Thomas Simon for 
pointing out to us the connection with the Dickman function, and how
Theorem~\ref{th:scalingY} follows from results in \cite{cf:Sato}.
We thank Andreas Kyprianou for pointing out reference \cite{BKKK14}
and
Ester Mariucci for pointing out reference \cite{cf:RW}.
F.C. is supported by the PRIN Grant 
20155PAWZB ``Large Scale Random Structures''.
R.S. is supported by NUS grant R-146-000-253-114.
N.Z. is supported by EPRSC through grant EP/R024456/1.


\end{document}